\documentclass[11pt, reqno,usenames, dvipsnames]{amsart}
\usepackage[T1]{fontenc}
\usepackage[foot]{amsaddr}
\usepackage{times}
\usepackage[margin=1in]{geometry}
\usepackage{pdflscape}

\usepackage{multirow}

\makeatletter
\def\paragraph{\@startsection{paragraph}{4}%
  \z@\z@{-\fontdimen2\font}%
  {\normalfont\bfseries}}
\makeatother

\usepackage{amssymb}
\usepackage{amsthm}
\usepackage[normalem]{ulem}
\usepackage{mathabx}
\usepackage{bbm}
\usepackage[hyperindex,breaklinks, colorlinks=true,urlcolor=NavyBlue, linkcolor=NavyBlue,citecolor=NavyBlue]{hyperref}
\usepackage{amsmath}
\usepackage{mathtools}

\usepackage{array}
\newcolumntype{L}{>{\centering\arraybackslash}m{5cm}}

\usepackage{mathrsfs}
\usepackage{framed}
\usepackage{latexsym}
\usepackage{pb-diagram}
\usepackage{graphicx}
\usepackage{subcaption}
\usepackage[all]{xy} 
\usepackage{color}
\usepackage{tikz}
\usepackage{algorithm}
\usepackage[noend]{algpseudocode}
\usepackage{thmtools,thm-restate} 

\usetikzlibrary{backgrounds,calc,positioning,shapes,shadows,arrows,fit}
\usepackage{tikz-cd}

\newcommand{\N}{\mathbb{N}}
\newcommand{\Z}{\mathbb{Z}}

\newcommand{\R}{\mathbb{R}}

\newcommand{\mc}{\mathcal}
\newcommand{\mb}{\mathbb}
\newcommand{\mf}{\mathfrak}

\renewcommand{\b}{\beta}
\newcommand{\g}{\gamma}
\renewcommand{\d}{\delta}
\newcommand{\e}{\varepsilon}
\newcommand{\w}{\omega}
\newcommand{\s}{\sigma}
\newcommand{\ph}{\varphi}
\newcommand{\p}{\partial}
\renewcommand{\t}{\tau}

\newcommand{\Si}{\Sigma}

\newcommand{\Mod}[1]{\ (\mbox{mod}\ #1)}

\newcommand{\set}[1]{\left\{#1\right\}}
\newcommand{\la}{\langle}
\newcommand{\ra}{\rangle}

\renewcommand{\r}{\rightarrow}
\newcommand{\xr}{\xrightarrow}
\newcommand{\hr}{\hookrightarrow}

\renewcommand{\emptyset}{\varnothing}

\newcommand{\diam}{\operatorname{diam}}

\newcommand{\inv}{^{-1}}
\newcommand{\bd}{\operatorname{Bd}}

\newcommand{\Ncal}{\mathcal{N}}		

\newcommand{\dis}{\operatorname{dis}}

\newcommand{\dn}{d_{\mathcal{N}}}	
\newcommand{\db}{d_{\operatorname{B}}}
\newcommand{\di}{d_{\operatorname{I}}}

\renewcommand{\H}{\mathcal{H}}
\newcommand{\U}{\mathcal{U}}

\newcommand{\pvec}{{\operatorname{\mathbf{PVec}}}}
\newcommand{\pers}{{\operatorname{\mathbf{Pers}}}}
\newcommand{\rank}{\operatorname{rank}}

\newcommand{\card}{\operatorname{card}}

\newcommand{\so}{\operatorname{so}}
\newcommand{\si}{\operatorname{si}}

\newcommand{\im}{\operatorname{im}}
\newcommand{\src}{\mf{D}^{\operatorname{so}}}
\newcommand{\sink}{\mf{D}^{\operatorname{si}}}

\newcommand{\id}{\operatorname{id}}
\newcommand{\pow}{\operatorname{pow}}

\renewcommand{\1}{^{(1)}}

\newcommand{\ceil}[1]{\left \lceil{#1}\right \rceil }
\newcommand{\floor}[1]{\left \lfloor{#1}\right \rfloor }

\newcommand{\Rsc}{\mathscr{R}}	
\renewcommand{\H}{\mathcal{H}}	

\newcommand{\Cc}{\mathcal{C}}
\newcommand{\dgm}{\operatorname{Dgm}}
\newcommand{\V}{\mathcal{V}}

\newcommand{\us}{\mathbb{S}} 
\newcommand{\ud}{\mathbb{D}}

\newcommand{\cech}{\operatorname{\mathbf{\check{C}}}}
\newcommand{\rips}{\operatorname{\mathbf{VR}}}
\newcommand{\vrg}{\overrightarrow{\operatorname{VR}}}
\newcommand{\wf}{\operatorname{wf}}
\newcommand{\ind}{\operatorname{Ind}}

\newtheorem{theorem}{Theorem}
\newtheorem{question}{Question}

\newtheorem{proposition}[theorem]{Proposition}
\newtheorem{lemma}[theorem]{Lemma}
\newtheorem{corollary}[theorem]{Corollary}

\theoremstyle{definition}
\newtheorem{definition}{Definition}
\newtheorem{example}[theorem]{Example}
\newtheorem{remark}[theorem]{Remark}
\newtheorem{conjecture}{Conjecture}
\newtheorem{claim}{Claim}

\newenvironment{subproof}[1][\proofname]{%
  \begin{proof}[#1]%
}{%
  \end{proof}%
}
\title[Persistent homology of asymmetric networks]{A functorial Dowker theorem and persistent homology of asymmetric networks}
\author{Samir Chowdhury and Facundo M\'emoli}

\begin{document}

\date{\today}
\subjclass[2010]{55U99, 68U05 55N35}

\email{chowdhury.57@osu.edu, memoli@math.osu.edu}

\address[S. Chowdhury]{Department of Mathematics, The Ohio State University. Phone: (614) 292-6805.}
\address[F. M\'emoli]{Department of Mathematics and Department of
  Computer Science and Engineering, The Ohio State University. Phone: (614) 292-4975,
Fax: (614) 292-1479.}
\begin{abstract} 
We study two methods for computing network features with topological underpinnings: the Rips and Dowker persistent homology diagrams. Our formulations work for general networks, which may be asymmetric and may have any real number as an edge weight.  We study the sensitivity of Dowker persistence diagrams to asymmetry via numerous theoretical examples, including a family of highly asymmetric cycle networks that have interesting connections to the existing literature. In particular, we characterize the Dowker persistence diagrams arising from asymmetric cycle networks. We investigate the stability properties of both the Dowker and Rips persistence diagrams, and use these observations to run a classification task on a dataset comprising simulated hippocampal networks. Our theoretical and experimental results suggest that Dowker persistence diagrams are particularly suitable for studying asymmetric networks. As a stepping stone for our constructions, we prove a functorial generalization of a theorem of Dowker, after whom our constructions are named.
\end{abstract} 

\maketitle

\tableofcontents
\section{Introduction} 

Networks are used throughout the sciences for representing the complex
relations that exist between the objects of a dataset
\cite{newman2003structure,kleinberg-book}. Network data arises from applications in social science \cite{kumar2010structure,kleinberg-book}, commerce and economy \cite{elliott2014financial,kleinberg-book,acemoglu2015systemic}, neuroscience \cite{sporns2011networks,sporns2012discovering,sporns2004motifs,rubinov2010complex,
pessoa2014understanding}, biology \cite{barabasi2004network,huson2010phylogenetic}, and defence \cite{masys2014networks}, to name a few sources. Networks are often directed, in the sense that weights attached to edges do not satisfy any symmetry property, and this asymmetry often precludes the applicability of many standard methods for data analysis. 

Network analysis problems come in a wide
range of flavors. One problem is in \emph{exploratory data analysis}:
given a network representing a dataset of societal, economic, or scientific value,
the goal is to obtain insights that are meaningful to the interested
party and can help uncover interesting phenomena. Another problem is
\emph{network classification}: given a ``bag" of networks representing
multiple instances of different phenomena, one wants to obtain a
clustering which groups the networks together according to the
different phenomena they represent. 

Because networks are often too complex to deal with directly, one
typically extracts certain invariants of networks, and infers structural properties
of the networks from properties of these invariants. While there are
numerous such network invariants in the existing literature, there is
growing interest in adopting a particular invariant arising from
\emph{persistent homology} \cite{frosini1992measuring,robins1999towards,edelsbrunner2002topological,zomorodian2005computing}, known as a \emph{persistence diagram}, to
the setting of networks. Persistence diagrams are used in the context
of finite metric space or point cloud data to pick out \emph{features of
significance} while rejecting random noise
\cite{edelsbrunner2002topological,carlsson2009topology}. Since a network on $n$ nodes is regarded,
in the most general setting, as an $n\times n$ matrix of real numbers,
i.e. as a generalized metric space, it is conceivable that one should
be able to describe persistence diagrams for networks as well. 

The motivation for computing persistence diagrams of networks is at least
two-fold: (1) comparing persistence diagrams has been shown to be a
viable method for \emph{shape matching} applications \cite{frosini1992measuring,frosini1999size, collins2004barcode,bot-stab,carlsson2005persistence,dgh-pers},
analogous to the network classification problem described above, and
(2) persistence diagrams have been successfully applied to feature
detection, e.g. in detecting the structure of protein molecules (see \cite{krishnamoorthy2007topological,xia2014persistent} and 
\cite[\S 6]{edelsbrunner2002topological}) and solid materials (see \cite{hiraoka2016hierarchical}) and might thus be a useful tool for exploratory analysis of network datasets.

We point the reader to \cite{ghrist2008barcodes,
  edelsbrunner2008persistent,
  carlsson2009topology,zigzag,weinberger2011persistent,burghelea2013topological,dey2014computing} for surveys of
persistent homology and its applications, and some recent extensions.

Some extant approaches that obtain persistence diagrams from networks assume that the underlying network data actually satisfies metric properties \cite{lee2011computing, khalid2014tracing}. A more general approach for obtaining persistence diagrams from networks is followed in \cite{horak2009persistent, carstens2013persistent, giusti2015clique,
petri2013topological}, albeit with the restriction that the input data sets are required to be symmetric matrices.

Our chief goal is to devise notions of persistent homology that are directly applicable to asymmetric networks in the most general sense, and are furthermore capable of absorbing structural information contained in the asymmetry. 

\subsection{Contributions and an overview of our approach}

\begin{figure}
\begin{tikzpicture}
\node (N) at (0,0){$\Ncal$};
\node (F) at (3,0){$\mc{F}$};
\node (D) at (6,0){$\dgm$};

\draw (N) edge[loop above, out=150, in=90, looseness=8, ->] node[left]{$\mf{s}$} (N);
\draw (N) edge[loop above, out=210, in=270, looseness=8, ->] node[left]{$\mf{t}$} (N);

\draw (N) edge[->, bend left, looseness =1.5] node[above]{$\mf{R}$} (F);
\draw (N) edge[->] node[above]{$\sink$} (F);
\draw (N) edge[->, bend right, looseness=1.5] node[above]{$\src$} (F);

\draw (F) edge[->] node[above]{$H_k$} (D);
\end{tikzpicture}
\captionsetup{width=.9\linewidth}
\caption{A schematic of some of the objects studied in this paper. $\Ncal$ is the collection of all weighted, directed networks (i.e. digraphs with possibly asymmetric real weights). $\mc{F}$ is the collection of filtered simplicial complexes. $\dgm$ is the collection of persistence diagrams. We study the Rips ($\mf{R}$) and Dowker ($\sink,\src$) filtrations, each of which takes a network as input and produces a filtered simplicial complex. $\mf{s}$ and $\mf{t}$ denote the network transformations of symmetrization (replacing a pair of weights between two nodes by the maximum weight) and transposition (swapping the weights between pairs of nodes). $\mf{R}$ is insensitive to both $\mf{s}$ and $\mf{t}$. But $\sink\circ \mf{t} = \src$, $\src\circ \mf{t} = \sink$, and in general, $\sink$ and $\src$ are \emph{not} invariant under $\mf{t}$ (Theorem \ref{thm:sym-trans-summary}).}
\label{fig:overview}
\end{figure}
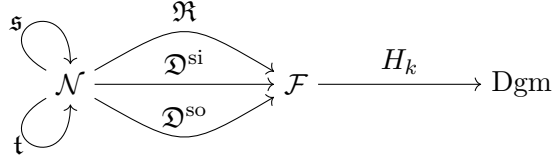

In this paper, we study two types of persistence diagrams: the
\emph{Rips} and \emph{Dowker} diagrams. We define both invariants in the setting of asymmetric networks with real-valued
weights, without assuming any metric properties at all (not symmetry and not even that the
matrix representing the networks weights vanishes on the diagonal). As a key step in defining the Dowker persistence diagram, we first define two dual constructions, each of which can be referred to as a Dowker persistence diagram, and then prove a \emph{functorial Dowker theorem} which implies that these two possible diagrams are equivalent. Following the line of work in
\cite{dgh-pers}, where stability of Rips persistence diagrams arising
from finite metric spaces was first established, we formulate similar stability results for the Rips and Dowker persistence diagrams of a network. Through various examples, in particular a family of \emph{cycle networks}, we espouse the idea that Dowker persistence diagrams are more appropriate than Rips persistence diagrams for studying asymmetric networks. 
We test our methods by solving a network classification problem on a database of simulated hippocampal networks.

The first step in constructing a persistence diagram from a network is
to construct a nested sequence of simplicial complexes, i.e. a
simplicial filtration, which, in our work, will be the \emph{Rips} or
\emph{Dowker} filtrations associated to a network. Rips and Dowker
simplicial complexes and their associated filtrations are classically
defined for metric spaces \cite{de2004topological, ghrist-eat}, and the
generalization to networks that we use is a natural extension of the
metric versions. After producing the simplicial filtrations, the
standard framework of \emph{persistent homology} takes over, and we obtain the Rips or Dowker persistence diagrams. 

Practitioners of persistent homology might recall that there are
\emph{two} Dowker complexes \cite[p. 73]{ghrist-eat}, which we
describe as the \emph{source} and \emph{sink} Dowker complexes. A
subtle point to note here is that each of these Dowker complexes can
be used to construct a persistence diagram. A folklore result in the
literature about persistent homology of metric spaces, known as
\emph{Dowker duality}, is that the two persistence diagrams arising
this way are equal \cite[Remark 4.8]{chazal2014persistence}. In this
paper we prove a stronger result---a functorial Dowker theorem---from
which the duality follows easily. Furthermore, the context of this
result is strictly more general than that of metric spaces (see below
for a more thorough description of the functorial version of Dowker's theorem).

Providing a construction of Rips and Dowker persistence diagrams is not enough: in order for these invariants to be useful in practice, one must verify that the diagrams are \emph{stable}. In this context, stability means the following: the dissimilarity between two Rips (resp. Dowker) persistence diagrams obtained from two networks should be bounded above by a function of the dissimilarity between the two networks. To our knowledge, stability is not addressed in the existing literature on producing persistence diagrams from networks. In our work, we provide stability results for both the Rips and Dowker persistence diagrams (Propositions \ref{prop:rips-stab} and \ref{prop:dowker-stab}). One key ingredient in our proof of this result is a notion of \emph{network distance} that follows previous work in \cite{clust-net, nets-allerton, nets-icassp}. This network distance is analogous to the Gromov-Hausdorff distance between metric spaces, which has previously been used to prove stability results for hierarchical clustering \cite{carlsson2008persistent,clust-um} and Rips persistence diagrams obtained from finite metric spaces \cite[Theorem 3.1]{dgh-pers}. The Gromov-Hausdorff distance was later used in conjunction with the Algebraic Stability Theorem of \cite{chazal2009proximity} to provide alternative proofs of stability results for Rips and Dowker persistence diagrams arising from metric spaces \cite{chazal2014persistence}. Our proofs also involve this Algebraic Stability Theorem, but the novelty of our approach lies in a reformulation of the network distance (Proposition \ref{prop:dn-ko}) that yields direct maps between two networks, thus passing naturally into the machinery of the Algebraic Stability Theorem (without having to define auxiliary constructions such as multivalued maps, as in \cite{chazal2014persistence}).

A crucial issue that we point out in this paper is that even though we
can construct both Rips and Dowker persistence diagrams out of
asymmetric networks, Rips persistence diagrams appear to be
\emph{blind} to asymmetry, whereas Dowker persistence diagrams do exhibit
sensitivity to asymmetry. In the case of Rips complexes, this
purported insensitivity to asymmetry can be immediately seen from its
definition. In the case of Dowker complexes, we argue about its
sensitivity to asymmetry in two different ways. Firstly, we do so by
explicitly computing Dowker persistence diagrams of multiple examples
of asymmetric networks. In particular, we consider a family of highly asymmetric
networks, the \emph{cycle networks}, and by bulding upon results from
\cite{adamaszek2015vietoris,adamaszek2016nerve} we prove a complete characterization
result for the Dowker persistence diagrams---across all dimensions---of any
network belonging to this family.  These networks constitute directed
analogues of circles and may be \emph{motifs} of interest in
different applications related to network data analysis. More specifically, appearance of nontrivial 1-dimensional persistence in the Dowker persistence diagram of asymmetric network data may suggest the presence of directed cycles in the data. 

Some of our experimental results suggest that the Rips persistence diagrams of
this family of networks are pathological, in the sense that they do
not represent the signatures one would expect from the underlying
dataset, which is a directed circle. Dowker persistence diagrams, on
the other hand, are well-behaved in this respect in that they succeed
at capturing relevant features. Secondly, we study the degree to which
Dowker persistence diagrams are insensitive to changes
(such as edge flips, or transposition) in
the network structure. An overview of this thread of work is provided in Figure \ref{fig:overview}.\\

\paragraph{Dowker's theorem and a functorial generalization}

Let $X,Y$ be two totally ordered sets, and let $R\subseteq X\times Y$ be a nonempty relation. Then one can define two simplicial complexes $E_R$ and $F_R$ as follows. A finite subset $\s \subseteq X$ belongs to $E_R$ whenever there exists $y \in Y$ such that $(x,y) \in R$ for each $x\in \s$. Similarly a finite subset $\t \subseteq Y$ belongs to $F_R$ whenever there exists $x\in X$ such that $(x,y) \in R$ for each $y\in \t$. These constructions can be traced back to \cite{dowker1952homology}, who proved the following result that we refer to as \emph{Dowker's theorem}:

\begin{theorem}[Dowker's theorem; Theorem 1a, \cite{dowker1952homology}]\label{thm:dowker} Let $X,Y$ be two totally ordered sets, let $R\subseteq X\times Y$ be a nonempty relation, and let $E_R, F_R$ be as above. Then for each $k \in \Z_+$,
\[H_k(E_R) \cong H_k(F_R).\]
\end{theorem}

There is also a strong form of Dowker's theorem that Bj\"{o}rner proves via the classical \emph{nerve theorem} \cite[Theorems 10.6, 10.9]{bjorner-book}:

\begin{theorem}[The strong form of Dowker's theorem; Theorem 10.9, \cite{bjorner-book}] \label{thm:dowker-strong}
Under the assumptions of Theorem \ref{thm:dowker}, we in fact have $|E_R| \simeq |F_R|$.
\end{theorem}

The Functorial Dowker Theorem is the following generalization of the strong form of Dowker's theorem: instead of a single nonempty relation $R \subseteq X\times Y$, consider any pair of nested, nonempty relations $R\subseteq R' \subseteq X\times Y$. Then there exist homotopy equivalences between the geometric realizations of the corresponding complexes that commute with the canonical inclusions, up to homotopy. We formalize this statement below.

\begin{theorem}[The Functorial Dowker Theorem (FDT)]
\label{thm:dowker-functorial}
Let $X,Y$ be two totally ordered sets, let $R\subseteq R' \subseteq X\times Y$ be two nonempty relations, and let $E_R, F_R, E_{R'}, F_{R'}$ be their associated simplicial complexes. Then there exist homotopy equivalences $\Gamma_{|E_R|}:|F_R| \r |E_R|$ and $\Gamma_{|E_{R'}|}: |F_{R'}| \r |E_{R'}|$ such that the following diagram commutes up to homotopy:

\begin{center}
\begin{tikzpicture}[]
\node (1) at (0,0){$|F_R|$};
\node (2) at (3,0){$|F_{R'}|$};

\node (3) at (0,-2){$|E_{R}|$};
\node (4) at (3,-2){$|E_{R'}|$};

\draw (1) edge[->] node[above]{$|\iota_E|$} (2);
\draw (3) edge[->] node[above]{$|\iota_F|$}(4);
\draw (1) edge[->] node[left]{$\Gamma_{|E_R|}$} node[right] {$\simeq$}(3);
\draw (2) edge[->] node[right]{$\Gamma_{|E_{R'}|}$} node[left] {$\simeq$} (4);

\end{tikzpicture}
\end{center}

In other words, we have $|\iota_F|\circ \Gamma_{|E_R|} \simeq \Gamma_{|E_{R'}|} \circ |\iota_E|$, where $\iota_E,\iota_F$ are the canonical inclusions.
\end{theorem}

From Theorem \ref{thm:dowker-functorial} we automatically obtain
Theorem \ref{thm:dowker-strong} (the strong form of Dowker's theorem)
as an immediate corollary. The strong form does not appear in Dowker's
original paper \cite{dowker1952homology}, but Bj\"{o}rner has given a
proof using the nerve theorem \cite[Theorems 10.6,
10.9]{bjorner-book}. Moreover, Bj\"{o}rner writes in a remark
following \cite[Theorem 10.9]{bjorner-book} that the nerve theorem and
the strong form of Dowker's theorem are equivalent, in the sense that
one implies the other. We were not able to find an elementary proof of
the strong form of Dowker's theorem in the existing
literature. However, such an elementary proof is provided by our proof
of Theorem \ref{thm:dowker-functorial} (given in Section \ref{sec:dowker-dual}), which we obtained by extending ideas in Dowker's original proof of Theorem \ref{thm:dowker}.\footnote{A thread with ideas towards the proof of Theorem \ref{thm:dowker-strong} was discussed in \cite[last accessed 4.24.2017]{Nlab-dowker}, but the proposed strategy was incomplete. We have inserted an addendum in \cite{Nlab-dowker} proposing a complete proof with a slightly different construction.} 

Whereas the Functorial Dowker Theorem and our elementary proof are of independent
interest, it has been suggested in \cite[Remark
4.8]{chazal2014persistence} that such a functorial version of Dowker's
theorem could also be proved using a functorial nerve theorem
\cite[Lemma 3.4]{chazal2008towards}. Despite being an interesting
possibility, we were not able to find a detailed proof of this
claim in the literature. In addition, Bj\"{o}rner's remark 
regarding the equivalence between the nerve theorem and the strong
form of Dowker's theorem suggests the following question:
\begin{question}
\label{q:f-nerve-f-dowker}
Are the Functorial Nerve Theorem (FNT) of
\cite{chazal2008towards} and the Functorial Dowker Theorem (FDT, Theorem
\ref{thm:dowker-functorial}) equivalent?
\end{question}

This question is of fundamental importance because the Nerve Theorem is a crucial tool in the applied topology literature and its functorial generalizations are equally important in persistent homology. In general, the answer is \emph{no}, and moreover, one (of the FNT and FDT) is not stronger than the other. The FNT of \cite{chazal2008towards} is stated for paracompact spaces, which are more general than the simplicial complexes of the FDT. However, the FNT of \cite{chazal2008towards} is stated for spaces with \emph{finitely-indexed} covers, so the associated nerve complexes are necessarily finite. All the complexes involved in the statement of the FDT are allowed to be infinite, so the FDT is more general than the FNT in this sense.

To clarify these connections, we formulate a simplicial Functorial Nerve Theorem (Theorem \ref{thm:nerve-functorial-II}) and prove it via a finite formulation of the FDT (Theorem \ref{thm:dowker-functorial-finite}). In turn, we show that the simplicial FNT implies the finite FDT, thus proving the equivalence of these formulations (Theorem \ref{thm:dowker-nerve-eq}).

\begin{remark}
Dowker complexes are also known to researchers who use Q-analysis to study social networks \cite{johnson2013hypernetworks,
atkin1975mathematical, atkin1972cohomology}. We perceive that
viewing Dowker complexes through the modern lens of persistence will enrich the classical framework of Q-analysis by incorporating additional information about the \emph{meaningfulness} of features, thus potentially opening new
avenues in the social sciences. 
\end{remark}

An announcement of part of our work has appeared in \cite{dowker-asilo}.

\subsection{Implementations}

Following work in \cite{curto2008cell, dabaghian2012topological}, we
implement our methods in the setting of classifying simulated
hippocampal networks. We simulate the activity pattern of hippocampal
cells in an animal as it moves around arenas with a number of
obstacles, and compile this data into a network which can be
interpreted as the transition matrix for the time-reversal of a Markov
process. The motivating idea is to ascertain whether, by just
observing hippocampal activity and not using any higher reasoning
ability, one might be able to determine the number of obstacles in the
arena that the animal has just finished traversing. The results of computing Dowker persistence diagrams suggest that the hippocampal activity is indeed sufficient to accurately count the number of obstacles in each arena.

Our datasets and software are available on \url{https://research.math.osu.edu/networks/Datasets.html} as part of the \texttt{PersNet} software package.

\subsection{Organization of the paper} Notation used globally is defined directly
below. \S\ref{sec:background} contains the necessary background on
persistent homology. \S\ref{sec:nets} contains our formulations for
networks, as well as some key ingredients of our stability
results. \S\ref{sec:rips} contains details about the Rips persistence
diagram. The first part of \S\ref{sec:dowker} contains details about
the Dowker persistence diagram. \S\ref{sec:dowker-dual} contains the
Functorial Dowker Theorem. The connection between the simplicial Functorial Nerve Theorem and the finite Functorial Dowker Theorem is detailed in \S\ref{sec:dowker-nerve-equiv}. In \S\ref{sec:symmetry} we show that Dowker complexes are sensitive to asymmetry. \S\ref{sec:cycle} contains a family of asymmetric networks, the \emph{cycle networks},
and a full characterization of their Dowker persistence diagrams. 
In \S\ref{sec:exp} we provide details on an implementation of our
methods. Finally, proofs of statements not contained in the main body
of the paper are relegated to Appendix \ref{app:proofs}, whereas details
about the characterization results for Dowker persistence diagrams of
cycle networks are given in Appendix \ref{sec:cycle-addendum}.

\subsection{Notation}

We will write $\mathbb{K}$ to denote a field, which we will fix and use throughout the paper. We will write $\Z_+$ and $\R_+$ to denote the nonnegative integers and reals, respectively. The extended real numbers $\R\cup \set{\infty, -\infty}$ will be denoted $\overline{\R}$. The cardinality of a set $X$ will be denoted $\card(X)$. The collection of nonempty subsets of a set $X$ will be denoted $\pow(X)$. The natural numbers $\set{1,2,3,\ldots}$ will be denoted by $\N$. The dimension of a vector space $V$ will be denoted $\dim(V)$. The rank of a linear transformation $f$ will be denoted $\rank(f)$. An isomorphism between vector spaces $V$ and $W$ will be denoted $V\cong W$. A homotopy relation for two maps $f,g:X\r Y$ between topological spaces will be denoted $f\simeq g$. Occasionally we will need to take about multisets, i.e. sets where elements can have multiplicity greater than 1. We will use square bracket notation $[\ldots]$ to denote multisets. Identity maps will be denoted by the notation $\id_\bullet$. Given a simplicial complex $\Si$, we will often write $V(\Si)$ to denote the vertex set of $\Si$. We will write $\bd(\s)$ to denote the boundary of a simplex $\s$.

\section{Background on persistent homology}\label{sec:background}

We assume that the reader is familiar with terms and concepts related to simplicial homology, and refer to \cite{munkres-book} for details. Here we describe our choices of notation. Whenever we have a simplicial complex over a set $X$ and a $k$-simplex $\set{x_0,x_1,\ldots, x_k}$, $k\in \Z_+$, we will assume that the simplex is \emph{oriented} by the ordering $x_0< x_1 < \ldots < x_k$. We will write $[x_0,x_1,\ldots,x_k]$ to denote the equivalence class of the even permutations of this chosen ordering, and $-[x_0,x_1,\ldots,x_k]$ to denote the equivalence class of the odd permutations of this ordering. Given a simplicial complex $\Si$, we will denote its geometric realization by $|\Si|$. The \emph{weak topology} on $|\Si|$ is defined by requiring that a subset $A \subseteq |\Si|$ is closed if and only if $A \cap |\s|$ is closed in $|\s|$ for each $\s \in \Si$.
A simplicial map $f: \Si \r \Xi$ between two simplicial complexes induces a map $|f|: |\Si| \r |\Xi|$ between the geometric realizations, defined as $|f|(\sum_{v\in \Si}a_v v):= \sum_{v\in \Si}a_v f(v)$. These induced maps satisfy the usual composition identity: given simpicial maps $f:\Si \r \Xi$ and $g:\Xi \r \Upsilon$, we have $|g\circ f| = |g| \circ |f|$. To see this, observe the following: 
\begin{equation}\label{eq:htpy-func}
|g\circ f|(\sum_{v\in \Si}a_v v) = \sum_{v\in \Si} a_vg(f(v)) = |g|(\sum_{v\in \Si}a_v f(v)) = |g|\circ|f|(\sum_{v\in \Si}a_v v).
\end{equation}

A \emph{filtration} of a simplicial complex $\Si$ (also called a \emph{filtered simplicial complex}) is defined to be a nested sequence $\{\Si^{\d}\subseteq \Si^{\d'}\}_{\d\leq \d' \in \R}$ of simplicial complexes satisfying the condition that there exist $\d_I,\, \d_F \in \R$ such that $\Si^{\d} = \emptyset $ for all $\d \leq \d_I$, and $\Si^{\d} = \Si \text{ for all }\d \geq \d_F$.

Fix a field $\mb{K}$. 
Given a finite simplicial complex $\Si$ and a dimension $k \in \Z_+$, we will denote a \emph{$k$-chain} in $\Si$ as $\sum_ia_i\s_i$, where each $a_i \in \mathbb{K}$ and $\s_i \in \Si$. 
We write $C_k(\Si)$ or just $C_k$ to denote the $\mathbb{K}$-vector space of all $k$-chains. We will write $\p_k$ to denote the associated \emph{boundary map} $\p_k : C_k \r C_{k-1}$:
\[\p_k[x_0,\ldots,x_k]:=\sum_i(-1)^i[x_0,\ldots,\hat{x}_i,\ldots, x_k], \text{ where $\hat{x}_i$ denotes omission of $x_i$ from the sequence.} \]  

We will write $\Cc=(C_k,\p_k)_{k\in \Z_+}$ to denote a \emph{chain complex}, i.e. a sequence of vector
spaces  with boundary maps such that
$\p_{k-1}\circ \p_k =0$. 
Given a chain complex $\Cc$ and any $k\in \Z_+$, the \emph{$k$-th homology of the chain complex $\Cc$} is denoted $H_k(\Cc) :=\ker(\p_k)/\im(\p_{k+1})$. The \emph{$k$-th Betti number} of $\Cc$ is denoted $\b_k(\Cc)$.

Given a simplicial map $f$ between simplicial complexes, we write $f_*$ to denote the induced chain map between the corresponding chain complexes \cite[\S 1.12]{munkres-book}, and $(f_k)_{\#}$ to denote the linear map on $k$th homology vector spaces induced for each $k \in \Z_+$.

The operations of passing from simplicial complexes and simplicial maps to chain complexes and induced chain maps, and then to homology vector spaces with induced linear maps, will be referred to as \emph{passing to homology}. Recall the following useful fact, often referred to as \emph{functoriality of homology} \cite[Theorem 12.2]{munkres-book}: given a composition $g\circ f$ of simplicial maps, we have
\begin{equation}
(g_k\circ f_k)_\# = (g_k)_\#\circ (f_k)_\# \qquad\text{ for each } k\in \Z_+.
\label{eq:functoriality}
\end{equation}

A \emph{persistence vector space} is defined to be a family of vector spaces $\{U^\d\xr{\mu_{\d,\d'}} U^{\d'}\}_{\d\leq \d'\in \R}$ such that: (1) $\mu_{\d,\d}$ is the identity for each $\d \in \R$, and (2) $\mu_{\d,\d''} = \mu_{\d',\d''}\circ \mu_{\d,\d'}$ for each $\d \leq \d' \leq \d'' \in \R$. The persistence vector spaces that we consider in this work also satisfy the following conditions: (1) $\dim(U^\d) <\infty$ at each $\d\in \R$, (2) there exist $\d_I,\, \d_F \in \R$ such that all maps $\mu_{\d,\d'}$ are isomorphisms for $\d,\d' \geq \d_F$ and for $\d,\d' \leq \d_I$, and (3) there are only finitely many values of $\d \in \R$ such that $U^{\d-\e} \not\cong U^{\d}$ for each $\e>0$. Here $\d$ is referred to as a \emph{resolution} parameter, and such a persistence vector space is described as being \emph{$\R$-indexed}. The collection of all such persistence vector spaces is denoted $\pvec(\R)$. Observe that by fixing $k \in \Z_+$ and passing to the $k$th homology vector space at each step $\Si^{\d}$ of a filtered simplicial complex $(\Si^{\d})_{\d \in \R}$, the functoriality of homology gives us the $k$th persistence vector space associated to $(\Si^{\d})_{\d \in \R}$, denoted 
\[\H_k(\Si) := \{H_k(\Cc^{\d})\xr{(\iota_{\d,\d'})_\#} H_k(\Cc^{\d'})\}_{\d \leq \d' \in \R}.\]

The elements of $\pvec(\R)$ contain only a finite number of vector spaces, up to isomorphism. By the classification results in \cite[\S5.2]{carlsson2005persistence}, it is possible to associate a full invariant, called a \emph{persistence barcode} or \emph{persistence diagram}, to each element of $\pvec(\R)$. 
This barcode is a multiset of \emph{persistence intervals}, and is represented as a set of lines over a single axis. The barcode of a persistence vector space $\V$ is denoted $\pers(\V)$. The intervals in $\pers(\V)$ can be represented as the \emph{persistence diagram of $\V$}, which is as a multiset of points lying on or above the diagonal in $\overline{\R}^2$, counted with multiplicity. More specifically,
\[\dgm(\V):=\big[(\d_i,\d_{j+1}) \in \overline{\R}^2 : [\d_i,\d_{j+1}) \in \pers(\V) \big],\]
where the multiplicity of $(\d_i,\d_{j+1})\in \overline{\R}^2$ is given by the multiplicity of $[\d_i,\d_{j+1}) \in \pers(\V)$.

Persistence diagrams can be compared using the \emph{bottleneck distance}, which we denote by $\db$. Details about this distance, as well as the other material related to persistent homology, can be found in \cite{chazal2012structure}. Numerous other formulations of the material presented above can be found in \cite{edelsbrunner2002topological, zomorodian2005computing, zigzag, edelsbrunner2010computational, edelsbrunner2014persistent,  bauer-isom, ph-self}.

\begin{remark}\label{rem:trivial-diag} Whenever we describe a persistence diagram as being \emph{trivial}, we mean that either it is empty, or it does not have any off-diagonal points. 
\end{remark}

\subsection{Interleaving distance and stability of persistence vector spaces.}
\label{sec:background-int}

In what follows, we will consider $\R$-indexed persistence vector spaces $\pvec(\R)$.

Given $\e \geq 0$, two $\R$-indexed persistence vector spaces $\V=\{V^\d\xr{\nu_{\d,\d'}} V^{\d'}\}_{\d\leq \d'}$ and $\U=\{U^\d\xr{\mu_{\d,\d'}} U^{\d'}\}_{\d\leq \d'}$ are said to be \emph{$\e$-interleaved} \cite{chazal2009proximity,bauer-isom} if there exist two families of linear maps
\begin{align*}
\{\ph_{\d,\d+\e}&:V^\d \r V^{\d + \e}\}_{\d \in \R},\\ 
\{\psi_{\d,\d+\e}&:U^\d \r U^{\d + \e}\}_{\d \in \R}
\end{align*}
such that the following diagrams commute for all $\d' \geq \d\in \R$:

\[ \begin{tikzcd}[column sep=large]
V^\d \arrow{r}{\nu_{\d,\d'}} \arrow[swap]{dr}{\ph_\d} & 
V^{\d'}\arrow{dr}{\ph_{\d'}} & 
{} & 
{} & 
V^{\d+\e} \arrow{r}{\nu_{\d+\eta,\d'+\eta}} &
V^{\d'+\e} \\
{} &
U^{\d+\e} \arrow{r}{\mu_{\d+\eta,\d'+\eta}} & 
U^{\d'+\e} &
U^{\d} \arrow{ur}{\psi_\d} \arrow{r}{\mu_{\d,\d'}} &
U^{\d'} \arrow[swap]{ur}{\psi_{\d'}} & {}
\end{tikzcd} \]

\[ \begin{tikzcd}[column sep=large]
V^\d \arrow{rr}{\nu_{\d,\d+2\e}} \arrow[swap]{dr}{\ph_\d} & {} &
V^{\d+2\e} & 
{}&
V^{\d+\e} \arrow{dr}{\psi_{\d+\e}}\\
{} &
U^{\d+\e} \arrow[swap]{ur}{\ph_{\d+\e}} & 
{}&
U^{\d} \arrow{ur}{\psi_{\d}} \arrow{rr}{\mu_{\d,\d+2\e}} & {} &
U^{\d+2\eta} 
\end{tikzcd} \]

The purpose of introducing $\e$-interleavings is to define a pseudometric on the collection of persistence vector spaces. The \emph{interleaving distance} between two $\R$-indexed persistence vector spaces $\V,\U$ is given by:
\[\di(\U,\V) := \inf \{\e \geq 0 : \text{$\U$ and $\V$ are $\e$-interleaved}\}.\]
One can verify that this definition induces a pseudometric on the collection of persistence vector spaces \cite{chazal2009proximity, bauer-isom}. The interleaving distance can then be related to the bottleneck distance as follows:
\begin{theorem}[Algebraic Stability Theorem, \cite{chazal2009proximity}] Let $\U, \V$ be two $\R$-indexed persistence vector spaces. Then,
\[\db(\dgm(\U),\dgm(\V))\leq \di(\U,\V).\]
\end{theorem}

Stability results are at the core of persistent homology, beginning with the classical bottleneck stability result in \cite{bot-stab}. One of our key contributions is to use the Algebraic Stability Theorem stated above, along with Lemma \S\ref{sec:nets} stated below, to prove stability results for methods of computing persistent homology of a network. 

Before stating the following lemma, recall that two simplicial maps $f,g: \Si \r \Xi$ are \emph{contiguous} if for any simplex $\s \in \Si$, $f(\s) \cup g(\s)$ is a simplex of $\Xi$. Contiguous maps satisfy the following useful properties:
\begin{proposition}[Properties of contiguous maps]
\label{prop:contigo-props}
Let $f,g: \Si \r \Xi$ be two contiguous simplicial maps. Then,
\begin{enumerate}
\item $|f|,|g|:|\Si| \r |\Xi|$ are homotopic \cite[\S 3.5]{spanier-book}, and  
\item The chain maps induced by $f$ and $g$ are chain homotopic, and as a result, the induced maps $f_\#$ and $g_\#$ for homology are equal \cite[Theorem 12.5]{munkres-book}.
\end{enumerate}
\end{proposition}

\begin{lemma}[Stability Lemma]
\label{lem:stab}
Let $\mf{F}, \mf{G}$ be two filtered simplicial complexes written as
\[\{\mf{F}^\d \xr{s_{\d,\d'}} \mf{F}^{\d'}\}_{\d'\geq \d\in \R} \text{ and } \{\mf{G}^\d \xr{t_{\d,\d'}} \mf{G}^{\d'}\}_{\d'\geq \d\in \R},\]
where $s_{\d,\d'}$ and $t_{\d,\d'}$ denote the natural inclusion maps.  
Suppose $\eta\geq 0$ is such that there exist families of simplicial maps $\set{\ph_\d:\mf{F}^\d \r \mf{G}^{\d+\eta}}_{\d\in \R}$ and $\set{\psi_\d:\mf{G}^\d \r \mf{F}^{\d+\eta}}_{\d\in \R}$ such that the following are satisfied for any $\d' \geq \d$:
\begin{enumerate}
\item $t_{\d+\eta,\d'+\eta}\circ \ph_\d$ and $\ph_{\d'}\circ s_{\d,\d'}$ are contiguous
\item $s_{\d+\eta,\d'+\eta}\circ \psi_\d$ and $\psi_{\d'}\circ t_{\d,\d'}$ are contiguous
\item $\psi_{\d+\eta}\circ \ph_\d$ and $s_{\d,\d+2\eta}$ are contiguous
\item $\ph_{\d+\eta}\circ \psi_\d$ and $t_{\d,\d+2\eta}$ are contiguous.
\end{enumerate}

All the diagrams are as below:

\[ \begin{tikzcd}[column sep=large]
\mf{F}^\d \arrow{r}{s_{\d,\d'}} \arrow[swap]{dr}{\ph_\d} & 
\mf{F}^{\d'}\arrow{dr}{\ph_{\d'}} & 
{} & 
{} & 
\mf{F}^{\d+\eta} \arrow{r}{s_{\d+\eta,\d'+\eta}} &
\mf{F}^{\d'+\eta} \\
{} &
\mf{G}^{\d+\eta} \arrow{r}{t_{\d+\eta,\d'+\eta}} & 
\mf{G}^{\d'+\eta} &
\mf{G}^{\d} \arrow{ur}{\psi_\d} \arrow{r}{t_{\d,\d'}} &
\mf{G}^{\d'} \arrow[swap]{ur}{\psi_{\d'}} & {}
\end{tikzcd} \]

\[ \begin{tikzcd}[column sep=large]
\mf{F}^\d \arrow{rr}{s_{\d,\d+2\eta}} \arrow[swap]{dr}{\ph_\d} & {} &
\mf{F}^{\d+2\eta} & 
{} & 
{}&
\mf{F}^{\d+\eta} \arrow{dr}{\ph_{\d+\eta}}\\
{} &
\mf{G}^{\d+\eta} \arrow[swap]{ur}{\psi_{\d+\eta}} & 
{} &
{}&
\mf{G}^{\d} \arrow{ur}{\psi_\d} \arrow{rr}{t_{\d,\d+2\eta}} & {} &
\mf{G}^{\d+2\eta} 
\end{tikzcd} \]

For each $k\in \Z_+$, let $\H_k(\mf{F}), \H_k(\mf{G})$ denote the $k$-dimensional persistence vector spaces associated to $\mf{F}$ and $\mf{G}$. Then for each $k\in \Z_+$,
\[\db(\dgm_k(\H_k(\mf{F})),\dgm_k(\H_k(\mf{G}))) \leq \di(\H_k(\mf{F}),\H_k(\mf{G})) \leq \eta.\] 
\end{lemma}

\section{Background on networks and our network distance}
\label{sec:nets}
A \emph{network} is a pair $(X,\w_X)$ where $X$ is a finite set and $\w_X: X\times X \r \R$ is a \emph{weight function}. Note that $\w_X$ need not satisfy the triangle inequality, any symmetry condition, or even the requirement that $\w_X(x,x) = 0$ for all $x\in X$. The weights are even allowed to be negative. The collection of all such networks is denoted $\Ncal$.

When comparing networks, one needs a way to correlate points in one network with points in the other. To see how this can be done, let $(X,\w_X), (Y,\w_Y) \in \Ncal$. Let $R$ be any nonempty relation between $X$ and $Y$, i.e. a nonempty subset of $X \times Y$. The \emph{distortion} of the relation $R$ is given by:
\[\dis(R):=\max_{(x,y),(x',y')\in R}|\w_X(x,x')-\w_Y(y,y')|.\] 

A \emph{correspondence between $X$ and $Y$} is a relation $R$ between $X$ and $Y$ such that $\pi_X(R)=X$ and $\pi_Y(R)=Y$, where $\pi_X:X\times Y \r X$ and $\pi_Y:X\times Y \r Y$ denote the natural projections. The collection of all correspondences between $X$ and $Y$ will be denoted $\Rsc(X,Y)$.

Following previous work in \cite{clust-net, nets-allerton, nets-icassp} the \emph{network distance} $\dn:\Ncal \times \Ncal \r \R_+$ is then defined as:
\[\dn(X,Y):=\frac{1}{2}\min_{R\in\Rsc}\dis(R).\]

It can be verified that $\dn$ as defined above is a pseudometric, and that the networks at 0-distance can be completely characterized \cite{nets-allerton}. Next we wish to prove the reformulation in Proposition \ref{prop:dn-ko}. First we define the distortion of a map between two networks. Given any $(X,\w_X),(Y,\w_Y)\in \Ncal$ and a map $\ph:(X,\w_X) \r (Y,\w_Y)$, the \emph{distortion} of $\ph$ is defined as:
\[\dis(\ph):= \max_{x,x'\in X}|\w_X(x,x')-\w_Y(\ph(x),\ph(x'))|.\]
Next, given maps $\ph:(X,\w_X)\r (Y,\w_Y)$ and $\psi:(Y,\w_Y)\r (X,\w_X)$, we define two \emph{co-distortion} terms: 
\begin{align*}C_{X,Y}(\ph,\psi) &:= \max_{(x,y)\in X\times Y}|\w_X(x,\psi(y)) - \w_Y(\ph(x),y)|,\\  C_{Y,X}(\psi,\ph) &:= \max_{(y,x)\in Y\times X}|\w_Y(y,\ph(x)) - \w_X(\psi(y),x)|.
\end{align*}

\begin{proposition}
\label{prop:dn-ko}
Let $(X,\w_X), (Y,\w_Y)\in \Ncal$. Then,
\[ \dn(X,Y) = \tfrac{1}{2}\min\{\max(\dis(\ph),\dis(\psi),C_{X,Y}(\ph,\psi), C_{Y,X}(\psi,\ph)) : \ph:X \r Y, \psi:Y \r X \text{ any maps}\}.\] 
\end{proposition}

\begin{remark} Proposition \ref{prop:dn-ko} is analogous to a result of Kalton and Ostrovskii \cite[Theorem 2.1]{kalton1997distances} where---instead of $\dn$---one has the Gromov-Hausdorff distance between metric spaces. We remark that when restricted to the special case of networks that are also metric spaces, the network distance $\dn$ agrees with the Gromov-Hausdorff distance. Details on the Gromov-Hausdorff distance can be found in \cite{burago}. 

An important remark is that in the Kalton-Ostrovskii formulation, there is only one co-distortion term. When Proposition \ref{prop:dn-ko} is applied to metric spaces, the two co-distortion terms become equal by symmetry, and thus the Kalton-Ostrovskii formulation is recovered. But \emph{a priori}, the lack of symmetry in the network setting requires us to consider both terms. 
\end{remark}

\begin{remark} In the following sections, we  propose methods for computing persistent homology of networks, and prove that they are stable via Lemma \ref{lem:stab}. Note that similar results, valid in the setting of metric spaces, have appeared in \cite{dgh-pers,chazal2014persistence}. Whereas the proofs in \cite{chazal2014persistence} invoke an auxiliary construction of multivalued maps arising from correspondences, our proofs simply use the maps $\ph, \psi$ arising directly from the reformulation of $\dn$ (Proposition \ref{prop:dn-ko}), thus streamlining the treatment. \end{remark}

When studying the effect of asymmetry on persistent homology, it will be useful to consider the network transformations that we define next.

\begin{definition}[Symmetrization and Transposition] 
\label{defn:sym-trans}
Define the \emph{max-symmetrization} map $\mf{s}:\Ncal \r \Ncal$ by $(X,\w_X)\mapsto (X,\widehat{\w_X})$, where for any network $(X,\w_X)$, we define $\widehat{\w_X}:X\times X \r \R$ as follows:
\[\widehat{\w_X}(x,x'):= \max(\w_X(x,x'),\w_X(x',x)), \text{ for } x,x'\in X.\]
Also define the \emph{transposition} map $\mf{t}:\Ncal \r \Ncal$ by $(X,\w_X) \mapsto (X,\w_X^\top)$, where for any $(X,\w_X) \in \Ncal$, we define $\w_X^\top(x,x'):= \w_X(x',x)$ for $x,x'\in X$. For convenience, we denote $X^\top:=\mf{t}(X)$ for any network $X$. 
\end{definition}

We are now ready to formulate our two methods for computing persistent
homology of networks. The Rips filtration is the ``workhorse'' of
persistent homology of metric spaces so it is natural to consider its
generalization to general asymmetric networks.

\section{The Rips filtration of a network}
\label{sec:rips}
Recall that for a metric space $(X,d_X)$, the \emph{Rips complex} is defined for each $\d \geq 0$ as follows:
\[\mf{R}^\d_X := \set{\s \in \pow(X): \diam(\s) \leq \d}, \text{ where } \diam(\s) := \max_{x,x'\in \s}d_X(x,x').\]

Following this definition, we can define the Rips complex for a network $(X,\w_X)$ as follows:
\[\mf{R}^\d_X:=\{\s \in \pow(X) : \max_{x,x'\in \s}\w_X(x,x') \leq \d\}.\]

To any network $(X,\w_X)$, we may associate the \emph{Rips filtration} $\{\mf{R}^\d_X\hr \mf{R}^{\d'}_X\}_{\d\leq \d'}$. We denote the $k$-dimensional persistence vector space associated to this filtration by $\H_k^{\mf{R}}(X)$, and the corresponding persistence diagram by $\dgm_k^{\mf{R}}(X)$. The Rips filtration is stable to small perturbations of the input data:9

\begin{proposition}
\label{prop:rips-stab}
Let $(X,\w_X), (Y,\w_Y) \in \Ncal$. Then $\db(\dgm_k^{\mf{R}}(X),\dgm_k^{\mf{R}}(Y)) \leq 2\dn(X,Y).$
\end{proposition}
We omit the proof because it is similar to that of Proposition \ref{prop:dowker-stab}, which we will prove in detail.

\begin{remark}
\label{rem:rips-benefits}
The preceding proposition serves a dual purpose: (1) it shows that the
Rips persistence diagram is robust to noise in input data, and (2) it
shows that instead of computing the network distance between two
networks, one can compute the bottleneck distance between their Rips
persistence diagrams as a suitable proxy. The advantage to computing
bottleneck distance is that it can be done in polynomial time (see
\cite{efrat2001geometry}), whereas computing $\dn$ is NP-hard in
general. This follows from the fact that the problem of computing $\dn$ includes the problem of computing the Gromov-Hausdorff distance between finite metric spaces, which is an NP-hard problem \cite{schmiedl}. We remark that the idea of computing Rips persistence diagrams to compare finite metric spaces first appeared in \cite{dgh-pers}, and moreover, that Proposition \ref{prop:rips-stab} is an extension of Theorem 3.1 in \cite{dgh-pers}.
\end{remark}

The Rips filtration in the setting of symmetric networks has been used in \cite{horak2009persistent, carstens2013persistent, giusti2015clique,
petri2013topological}, albeit without addressing stability results. To our knowledge, Proposition \ref{prop:rips-stab} is the first quantitative result justifying the constructions in these prior works.
 
\begin{remark}[Rips is insensitive to asymmetry] 
\label{rem:rips-symm}
A critical weakness of the Rips complex construction is that it is not sensitive to asymmetry. To see this, recall the symmetrization map defined in Definition \ref{defn:sym-trans}, and let $(X,\w_X) \in \Ncal$. Now for any $\s \in \pow(X)$, we have $\max_{x,x' \in \s}\w_X(x,x') = \max_{x,x'\in \s}\widehat{\w_X}(x,x').$ It follows that for each $\d \geq 0$, the Rips complexes of $(X,\w_X)$ and $(X,\widehat{\w_X})=\mf{s}(X,\w_X)$ are equal, i.e. $\mf{R} = \mf{R} \circ \mf{s}$. Thus the Rips persistence diagrams of the original and max-symmetrized networks are equal. 
\end{remark}

\section{The Dowker filtration of a network}\label{sec:dowker}

Given $(X,\omega_X)\in \mathcal{N}$,  and for any $\d \in \R$,
consider the following relation on $X$: 
\begin{equation}
R_{\d,X}:=\set{(x,x') : \w_X(x,x') \leq \d}. 
\label{eq:relation}
\end{equation}
Then $R_{\d,X} \subseteq X \times X$, and $R_{\d_F,X} = X \times X$ for some sufficiently large $\d_F$. Furthermore, for any $\d' \geq \d$, we have $R_{\d,X} \subseteq R_{\d',X}$. Using $R_{\d,X}$, we build a simplicial complex $\sink_\d$ as follows: 
\begin{equation}\label{eq:d-sink}
\sink_{\d,X}:=\set{\s=[x_0,\ldots, x_n] : \text{ there exists } x'\in X \text{ such that } (x_i,x')\in R_{\d,X} \text{ for each } x_i}.
\end{equation}

If $\s \in \sink_{\d,X}$, it is clear that any face of $\s$ also belongs to $\sink_{\d,X}$. We call $\sink_{\d,X}$ the \emph{Dowker $\d$-sink simplicial complex} associated to $X$, and refer to $x'$ as a \emph{$\d$-sink} for $\s$ (where $\s$ and $x'$ should be clear from context). 

Since $R_{\d,X}$ is an increasing sequence of sets, it follows that $\sink_{\d,X}$ is an increasing sequence of simplicial complexes. In particular, for $\d'\geq \d$, there is a natural inclusion map $\sink_{\d,X} \hr \sink_{\d',X}$.  We write $\sink_X$ to denote the filtration $\{\sink_{\d,X} \hr \sink_{\d',X}\}_{ \d \leq \d'}$ associated to $X$. We call this the \emph{Dowker sink filtration on $X$}. 
We will denote the $k$-dimensional persistence diagram arising from this filtration by $\dgm_k^{\si}(X)$.

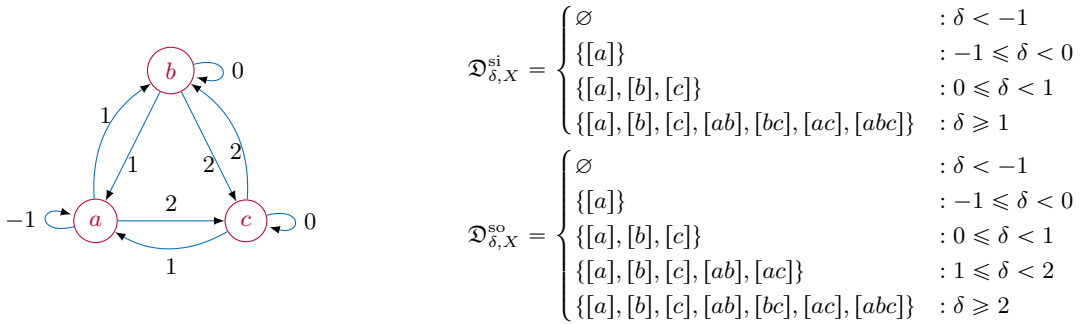
\begin{figure}[b]
\begin{center}
\begin{tikzpicture}[every node/.style={font=\footnotesize}]
\tikzset{>=latex}

\begin{scope}[draw,purple]
\node[circle,draw](2) at (2,0){$a$};
\node[circle,draw](3) at (3,2){$b$};
\node[circle,draw](4) at (4,0){$c$};
\end{scope}

\begin{scope}[draw,xshift=6cm,yshift=1cm]
\node[](5) at (5,1){
$
\sink_{\d,X}=
\begin{cases}
\emptyset &: \d<-1 \\
\{[a]\} &:-1\leq \d<0\\
\{[a],[b],[c]\} &: 0\leq \d < 1\\
\{[a],[b],[c],[ab],[bc],[ac],[abc]\} &: \d \geq 1
\end{cases}
$
};
\end{scope}

\begin{scope}[draw,xshift=6cm,yshift=-1.2cm]
\node[](50) at (5,1){
$
\src_{\d,X}=
\begin{cases}
\emptyset &: \d<-1 \\
\{[a]\} &:-1\leq \d<0\\
\{[a],[b],[c]\} &: 0\leq \d < 1\\
\{[a],[b],[c],[ab],[ac]\} &: 1\leq \d < 2\\
\{[a],[b],[c],[ab],[bc],[ac],[abc]\} &: \d \geq 2
\end{cases}
$
};
\end{scope}

\begin{scope}[xshift=1cm,yshift=1cm]
\node (0) at (4,0){};
\node (1) at (6,0){};
\end{scope}

\begin{scope}[draw=NavyBlue]
\path[->] (2) edge [loop left] node[left]{$-1$}(2);
\path[->] (3) edge [loop right] node[right]{$0$}(3);
\path[->] (4) edge [loop right] node[right]{$0$}(4);
\path[->] (2) edge [bend left] node[above]{$1$} (3);
\path[->] (3) edge [] node[below]{$1$} (2);
\path[->] (2) edge [] node[above]{$2$} (4);
\path[->] (3) edge [] node[below]{$2$} (4);
\path[->] (4) edge [bend left] node[below]{$1$} (2);
\path[->] (4) edge [bend right] node[below]{$2$} (3);
\end{scope}


\end{tikzpicture}
\caption{Computing the Dowker sink and source complexes of a network $(X,\w_X)$. Observe that the sink and source complexes are different in the range $1\leq \d < 2$.}
\label{fig:dowker-three-node}
\end{center}
\end{figure}

Note that we can define a dual construction as follows: 
\begin{equation}\label{eq:d-src}
\src_{\d,X}:=\set{\s=[x_0,\ldots, x_n] : \text{ there exists } x'\in X \text{ such that } (x',x_i)\in R_{\d,X} \text{ for each } x_i}.
\end{equation}

We call $\src_{\d,X}$ the \emph{Dowker $\d$-source simplicial complex} associated to $X$. The filtration $\{\src_{\d,X} \hr \src_{\d',X}\}_{ \d \leq \d'}$ associated to $X$ is called the \emph{Dowker source filtration}, denoted $\src_X$. 
We denote the $k$-dimensional persistence diagram arising from this filtration by $\dgm_k^{\so}(X)$. Notice that any construction using $\sink_{\d,X}$ can also be repeated using $\src_{\d,X}$, so we focus on the case of the sink complexes and restate results in terms of source complexes where necessary. 
In particular, we will prove in \S\ref{sec:dowker-dual} that 
\[\dgm_k^{\si}(X) = \dgm_k^{\so}(X) \text{ for any } k\in \Z_+,\] 
so it makes sense to talk about ``the" Dowker diagram associated to $X$.

The sink and source filtrations are not equal in general; this is illustrated in Figure \ref{fig:dowker-three-node}.

As in the case of the Rips filtration, both the Dowker sink and source
filtrations are stable.

\begin{proposition}
\label{prop:dowker-stab}
Let $(X,\w_X), (Y,\w_Y) \in \Ncal$. Then $\db(\dgm_k^{\bullet}{(X)},\dgm_k^{\bullet}(Y)) \leq 2\dn(X,Y).$ Here $\dgm^{\bullet}$ refers to either of $\dgm^{\si}$ and $\dgm^{\so}$. 
\end{proposition}

\begin{proof}[Proof of Proposition \ref{prop:dowker-stab}]
Both cases are similar, so we just prove the result for $\dgm^{\si}$. 
Let $\eta=2\dn(X,Y)$. Then by Proposition \ref{prop:dn-ko}, there exist maps $\ph:X \r Y, \psi: Y \r X$ such that 
\[\max(\dis(\ph),\dis(\psi), C_{X,Y}(\ph,\psi),C_{Y,X}(\psi,\ph))\leq \eta.\]
First we check that $\ph,\psi$ induce simplicial maps $\ph_\d:\sink_{\d,X} \r \sink_{\d+\eta,Y}$ and $\psi_\d:\sink_{\d,Y} \r \sink_{\d+\eta,Y}$ for each $\d\in \R$. 

Let $\d' \geq \d \in \R$. Let $\s=[x_0,\ldots, x_n] \in \sink_{\d,X}$. Then there exists $x' \in X$ such that $\w_X(x_i,x') \leq \d$ for each $0\leq i \leq n$. Fix such an $x'$. Since $\dis(\ph) \leq \eta$, we have the following for each $i$:
\[\vert \w_X(x_i,x') - \w_Y(\ph(x_i),\ph(x'))\vert \leq \eta.\]

So $\w_Y(\ph(x_i),\ph(x')) \leq \w_X(x_i,x') + \eta \leq \d + \eta$ for each $0\leq i\leq n$. Thus $\ph_{\d}(\s):=\set{\ph(x_0),\ldots, \ph(x_n)}$ is a simplex in $\sink_{\d + \eta,Y}$. Thus the map on simplices $\ph_\d$ induced by $\ph$ is simplicial for each $\d\in \R$.

Similarly we can check that the map $\psi_\d$ on simplices induced by $\psi$ is simplicial. Now to prove the result, it will suffice to check the contiguity conditions in the statement of Lemma \ref{lem:stab}. Consider the following diagram:

\[ \begin{tikzcd}[column sep=large]
\sink_{\d,X} \arrow{r}{s_{\d,\d'}} \arrow[swap]{dr}{\ph_\d} & \sink_{\d',X}\arrow{dr}{\ph_{\d'}} & \\
&\sink_{\d+\eta,Y} \arrow{r}{t_{\d+\eta,\d'+\eta}} & \sink_{\d'+\eta,Y}
\end{tikzcd} \]

Here $s_{\d,\d'}$ and $t_{\d+\eta,\d'+\eta}$ are the inclusion maps. 
We claim that $t_{\d+\eta,\d'+\eta} \circ \ph_\d$ and $\ph_{\d'} \circ s_{\d,\d'}$ are contiguous simplicial maps. To see this, let $\s \in \sink_{\d,X}$. Since $s_{\d,\d'}$ is just the inclusion, it follows that $t_{\d+\eta,\d'+\eta}(\ph_\d(\s)) \cup \ph_{\d'}(s_{\d,\d'}(\s))=\ph_\d(\s),$ which is a simplex in $\sink_{\d+\eta,Y}$ because $\ph_\d$ is simplicial, and hence a simplex in $\sink_{\d'+\eta,Y}$ because the inclusion $t_{\d+\eta,\d'+\eta}$ is simplicial. Thus $t_{\d+\eta,\d'+\eta} \circ \ph_\d$ and $\ph_{\d'} \circ s_{\d,\d'}$ are contiguous, and their induced linear maps for homology are equal. By a similar argument, one can show that $s_{\d+\eta,\d'+\eta}\circ \psi_\d$ and $\psi_{\d'}\circ t_{\d,\d'}$ are contiguous simplicial maps as well. 

Next we check that the maps $\psi_{\d+\eta}\circ \ph_\d$ and $s_{\d,\d+2\eta}$ in the figure below are contiguous.

\[ \begin{tikzcd}
\sink_{\d,X} \arrow{rr}{s_{\d,\d+2\eta}} \arrow[swap]{dr}{\ph_\d} & {} &
\sink_{\d+2\eta,X}\\
{} &
\sink_{\d+\eta,Y} \arrow[swap]{ur}{\psi_{\d+\eta}} & 
{} 
\end{tikzcd} \]

Let $x_i\in \s$. Note that for our fixed $\s = [x_0,\ldots, x_n]\in \sink_{\d,X}$ and $x'$, we have:
\begin{align*}
|\w_X(x_i,x')-\w_X(\psi(\ph(x_i)),\psi(\ph(x')))|&\leq 
|\w_X(x_i,x') - \w_Y(\ph(x_i),\ph(x'))|\\
&+ |\w_Y(\ph(x_i),\ph(x')) - \w_X(\psi(\ph(x_i)),\psi(\ph(x')))| \\
&\leq 2\eta.\\
\text{Thus we obtain }\hfill \w_X(\psi(\ph(x_i)),\psi(\ph(x')))&\leq \w_X(x_i,x')+ 2\eta \leq \d+2\eta.
\end{align*}

Since this holds for any $x_i \in \s$, it follows that $\psi_{\d+\eta}(\ph_\d(\s)) \in \sink_{\d+2\eta,X}$. We further claim that 
\[\t:=\s \cup \psi_{\d+\eta}(\ph_\d(\s)) = \set{x_0,x_1,\ldots, x_n, \psi(\ph(x_0)),\ldots, \psi(\ph(x_n))}\] 
is a simplex in $\sink_{\d+2\eta,X}$. Let $0\leq i\leq n$. It suffices to show that $\w_X(x_i,\psi(\ph(x')) \leq \d + 2\eta$. 

Notice that from the reformulation of $\dn$ (Proposition \ref{prop:dn-ko}), we have 
\[C_{X,Y}(\ph,\psi) = \max_{(x,y)\in X\times Y}|\w_X(x,\psi(y)) - \w_Y(\ph(x),y)| \leq \eta .\]
Let $y = \ph(x')$. Then $|\w_X(x_i,\psi(y)) - \w_Y(\ph(x_i),y)| \leq \eta$. In particular,
\[\w_X(x_i,\psi(\ph(x'))) \leq \w_Y(\ph(x_i),\ph(x')) + \eta \leq \w_X(x_i,x') + 2\eta \leq \d+2\eta.\]

Since $0\leq i\leq n$ were arbitrary, it follows that $\t \in \sink_{\d+2\eta,X}$. Thus $\psi_{\d+\eta}\circ \ph_\d$ and $s_{\d,\d+2\eta}$ are contiguous. Similarly, one can use the $\dis(\psi)$ and $C_{Y,X}(\psi,\ph)$ terms to show that $t_{\d,\d+2\eta}$ and $\ph_{\d+\eta}\circ \psi_\d$ are contiguous. 

The result now follows by an application of Lemma \ref{lem:stab}. \end{proof}

\begin{remark}
\label{rem:dowker-benefits}
The preceding proposition shows that the Dowker persistence diagram is robust to noise in input data, and that the bottleneck distance between Dowker persistence diagrams arising from two networks can be used as a proxy for computing the actual network distance. Note the analogy with Remark \ref{rem:rips-benefits}.
\end{remark}

Both the Dowker and Rips filtrations are valid methods for computing persistent homology of networks, by virtue of their stability results (Propositions \ref{prop:rips-stab} and \ref{prop:dowker-stab}). However, we present the Dowker filtration as an appropriate method for capturing directionality information in directed networks. In \S\ref{sec:symmetry} we discuss this particular feature of the Dowker filtration in full detail.

\begin{remark}[Symmetric networks]
In the setting of symmetric networks, the Dowker sink and source simplicial filtrations coincide, and so we automatically obtain $\dgm_k^{\so}(X)=\dgm_k^{\si}(X)$ for any $k\in \Z_+$ and any $(X,\w_X)\in \Ncal$. 
\end{remark}

\begin{remark}[The metric space setting and relation to witness complexes]
When restricted to the setting of metric spaces, the Dowker complex resembles a construction called the witness complex \cite{de2004topological}. In particular, a version of the Dowker complex for metric spaces, constructed in terms of \emph{landmarks} and \emph{witnesses}, was discussed in \cite{chazal2014persistence}, along with stability results. When restricted to the special networks that are pseudo-metric spaces, our definitions and results agree with those presented in \cite{chazal2014persistence}.
\end{remark}

\subsection{The Functorial Dowker Theorem and equivalence of diagrams}\label{sec:dowker-dual}

Let $(X,\w_X)\in \Ncal$, and let $\d \in \R$ be such that $R_{\d,X}$ is nonempty. By applying Dowker's theorem (Theorem \ref{thm:dowker}) to the setting $Y=X$, we have $H_k(\sink_{\d,X}) \cong H_k(\src_{\d,X})$, for any $k \in \Z_+$. We still have this equality in the case where $R_{\d,X}$ is empty, because then $\sink_{\d,X}$ and $\src_{\d,X}$ are both empty. Thus we obtain:

\begin{corollary}\label{cor:dowker} Let $(X,\w_X)\in\Ncal$, $\d \in \R$, and $k\in \Z_+$. Then,
\[H_k(\sink_{\d,X}) \cong H_k(\src_{\d,X}).\]
\end{corollary}

In the persistent setting, Theorem \ref{thm:dowker} and Corollary \ref{cor:dowker} suggest the following question: 
\begin{quote}
\textit{Given a network $(X,\w_X)$ and a fixed dimension $k\in \Z_+$, are the persistence diagrams of the Dowker sink and source filtrations of $(X,\w_X)$ necessarily equal?}
\end{quote}

In what follows, we provide a positive answer to the question above. Our strategy is to use the Functorial Dowker Theorem (Theorem \ref{thm:dowker-functorial}), for which we will provide a complete proof below. The Functorial Dowker Theorem implies equality between sink and source persistence diagrams.

\begin{corollary}[Dowker duality]
\label{cor:dowker-dual} Let $(X,\w_X)\in \Ncal$, and let $k\in \Z_+$. Then,
\[\dgm_k^{\si}(X)=\dgm_k^{\so}(X).\]
Thus we may call either of the diagrams above the \emph{$k$-dimensional Dowker diagram of $X$}, denoted $\dgm_k^{\mf{D}}(X)$.
\end{corollary}

Before proving the corollary, we state an $\R$-indexed variant of the Persistence Equivalence Theorem \cite{edelsbrunner2010computational}. This particular version follows from the \emph{isometry theorem} \cite{bauer-isom}, and we refer the reader to \cite[Chapter 5]{chazal2012structure} for an expanded presentation of this material.

\begin{theorem}[Persistence Equivalence Theorem]
\label{thm:pet} 
Consider two persistence vector spaces $\U=\{U^{\d} \xr{\mu_{\d,\d'}} U^{\d'}\}_{\d\leq\d' \in \R }$ and $\V=\{V^{\d} \xr{\nu_{\d,\d'}} V^{\d'}\}_{\d \leq \d'\in \R}$ with connecting maps $f_{\d}:U^{\d}\r V^{\d'}$.
\begin{center}
\begin{tikzpicture}[]
\node (00) at (-3,0){$\cdots$};
\node (1) at (0,0){$V^{\d}$};
\node (2) at (3,0){$V^{\d'}$};
\node (3) at (6,0){$V^{\d''}$};
\node (01) at (9,0){$\cdots$};

\node (02) at (-3,2){$\cdots$};
\node (4) at (0,2){$U^{\d}$};
\node (5) at (3,2){$U^{\d'}$};
\node (6) at (6,2){$U^{\d''}$};
\node (03) at (9,2){$\cdots$};

\draw (00) edge[->] (1);
\draw (3) edge[->] (01);
\draw (02) edge[->] (4);
\draw (6) edge[->] (03);

\draw (1) edge[->] (2);
\draw (2) edge[->] (3);
\draw (4) edge[->] (5);
\draw (5) edge[->] (6);
\draw (1) edge[<-] node[right]{$f_{\d}$} (4);
\draw (2) edge[<-] node[right]{$f_{\d'}$} (5);
\draw (3) edge[<-] node[right]{$f_{\d''}$} (6);
\end{tikzpicture}
\end{center}

If the $f_{\d}$ are all isomorphisms and each square in the diagram above commutes, then:
\[\dgm(\U) = \dgm(\V).\]
\end{theorem}

\begin{proof}[Proof of Corollary \ref{cor:dowker-dual}] Let $\d \leq \d' \in \R$, and consider the relations $R_{\d,X} \subseteq R_{\d',X} \subseteq X\times X$. Suppose first that $R_{\d,X}$ and $R_{\d',X}$ are both nonempty. By applying Theorem \ref{thm:dowker-functorial}, we obtain homotopy equivalences between the source and sink complexes that commute with the canonical inclusions up to homotopy. Passing to the $k$-th homology level, we obtain persistence vector spaces that satisfy the commutativity properties of Theorem \ref{thm:pet}. The result follows from Theorem \ref{thm:pet}. 

In the case where $R_{\d,X}$ and $R_{\d',X}$ are both empty, there is nothing to show because all the associated complexes are empty. Suppose $R_{\d,X}$ is empty, and $R_{\d',X}$ is nonempty. Then $\sink_{\d,X}$ and $\src_{\d,X}$ are empty, so their inclusions into $\sink_{\d',X}$ and $\src_{\d',X}$ induce zero maps upon passing to homology. Thus the commutativity of Theorem \ref{thm:pet} is satisfied, and the result follows by Theorem \ref{thm:pet}.\end{proof}

\paragraph{The proof of the Functorial Dowker Theorem}
It remains to prove Theorem \ref{thm:dowker-functorial}. Because the proof involves numerous maps, we will adopt the notational convention of adding a subscript to a function to denote its codomain---e.g. we will write $f_B$ to denote a function with codomain $B$. 

First we recall the construction of a combinatorial barycentric subdivision (see \cite[\S 2]{dowker1952homology}, \cite[\S 4.7]{lefschetz1942algebraic}, \cite[Appendix A]{barmak2011algebraic}).

\begin{definition}[Barycentric subdivisions]\label{def:subdivision} For any simplicial complex $\Si$, one may construct a new simplicial complex $\Si^{(1)}$, called the \emph{first barycentric subdivision}, as follows:
\[\Si^{(1)}:=\set{[\s_1,\s_2,\ldots, \s_p] : 
\s_1 \subseteq \s_2 \subseteq \ldots \subseteq \s_p, \text{ each } \s_i \in \Si}.\]
Note that the vertices of $\Si^{(1)}$ are the simplices of $\Si$, and the simplices of $\Si^{(1)}$ are nested sequences of simplices of $\Si$. Furthermore, note that given any two simplicial complexes $\Si, \Xi$ and a simplicial map $f:\Si \r \Xi$, there is a natural simplicial map $f\1:\Si\1\r \Xi\1$ defined as:
\[f\1([\s_1,\ldots,\s_p]):=[f(\s_1),\ldots,f(\s_p)], \qquad \s_1\subseteq \s_2\subseteq \ldots, \s_p, \text{ each } \s_i\in \Si.\]
To see that this is simplicial, note that $f(\s_i) \subseteq f(\s_j)$ whenever $\s_i\subseteq \s_j$. 
As a special case, observe that any inclusion map $\iota:\Si \hr \Xi$ induces an inclusion map $\iota\1:\Si\1 \hr \Xi\1$.  

Given a simplex $\s=[x_0,\ldots, x_k]$ in a simplicial complex $\Si$, one defines the \emph{barycenter} to be the point $\mc{B}(\s):= \sum_{i=0}^k \tfrac{1}{k+1}x_i \in |\Si|$. Then the spaces $|\Si\1|$ and $|\Si|$ can be identified via a homeomorphism $\mc{E}_{|\Si|}:|\Si\1| \r |\Si|$ defined on vertices by $\mc{E}_{|\Si|}(\s):= \mc{B}(\s)$ and extended linearly. 
\end{definition}

Details on the preceding list of definitions can be found in \cite[\S 2.14-15, 2.19]{munkres-book}, \cite[\S 3.3-4]{spanier-book}, and also \cite[Appendix A]{barmak2011algebraic}. The next proposition follows from the discussions in these references, and is a simple restatement of \cite[Proposition A.1.5]{barmak2011algebraic}. We provide a proof in the appendix for completeness.

\begin{proposition}[Simplicial approximation to $\mc{E}_\bullet$]\label{prop:subdiv-identity} Let $\Si$ be a simplicial complex, and let $\Phi: \Sigma\1 \r \Sigma$ be a simplicial map such that $\Phi(\s) \in \s$ for each $\s \in \Sigma$. Then $|\Phi| \simeq \mc{E}_{|\Sigma|}$. 
\end{proposition}

We now introduce some auxiliary constructions dating back to \cite{dowker1952homology} that use the setup stated in Theorem \ref{thm:dowker-functorial}. For any nonempty relation $R\subseteq X\times Y$, one may define \cite[\S 2]{dowker1952homology} an associated map $\Phi_{E_R} : E_R\1 \r E_R$ as follows: first define $\Phi_{E_R}$ on vertices of $E_R\1$ by $\Phi_{E_R}(\s)=s_{\s}$, where $s_{\s}$ is the least vertex of $\s$ with respect to the total order. Next, for any simplex $[\s_1,\ldots,\s_k]$ of $E_R\1$, where $\s_1\subseteq \ldots \subseteq \s_k$, we have $\Phi_{E_R}(\s_i)=s_{\s_i} \in \s_k$ for all $1\leq i\leq k$. Thus $[\Phi_{E_R}(\s_1),\ldots,\Phi_{E_R}(\s_k)]=[s_{\s_1},s_{\s_2},\ldots,s_{\s_k}]$ is a face of $\s_k$, hence a simplex of $\Si$. This defines $\Phi_{E_R}$ as a simplicial map $E_R\1 \r E_R$. This argument also shows that $\Phi_{E_R}$ is order-reversing: if $\s \subseteq \s'$, then $\Phi_{E_R}(\s) \geq \Phi_{E_{R}}(\s')$.

\begin{remark}\label{rem:subdiv-id} 
Applying Proposition \ref{prop:subdiv-identity} to the setup above, one sees that $|\Phi_{E_R}| \simeq \mc{E}_{|E_R|}$. After passing to a second barycentric subdivision $E_R^{(2)}$ (obtained by taking a barycentric subdivision of $E_R\1$) and obtaining a map $\Phi_{E_R\1}:E_R^{(2)} \r E_R\1$, one also has $|\Phi_{E_R\1}| \simeq \mc{E}_{|E_R\1|}$.
\end{remark}

One can also define \cite[\S 3]{dowker1952homology} a simplicial map $\Psi_{F_R}: E_R\1 \r F_R$ as follows. Given a vertex $\s=[x_0,\ldots, x_k] \in E_R\1$, one defines $\Psi_{F_R}(\s)=y_\s$, for some $y_\s \in Y$ such that $(x_i,y_\s) \in R$ for each $i$. To see why this vertex map is simplicial, let $\s\1 = [\s_0,\ldots, \s_k]$ be a simplex in $E_R\1$. Let $x \in \s_0$. Then, because $\s_0 \subseteq \s_1 \subseteq \ldots \subseteq \s_k$, we automatically have that $(x,\Psi_{F_R}(\s_i)) \in R$, for each $i=0,\ldots, k$. Thus $\Psi_{F_R}(\s\1)$ is a simplex in $F_R$. This definition involves a choice of $y_\s$ when writing $\Psi_{F_R}(\s) = y_\s$, but all the maps resulting from such choices are contiguous \cite[\S 3]{dowker1952homology}. 

The preceding map induces a simplicial map $\Psi_{F_R\1}:E_R^{(2)} \r F_R\1$ as follows. Given a vertex $\t\1 = [\t_0,\ldots, \t_k] \in E_R^{(2)}$, i.e. a simplex in $E_R\1$, we define $\Psi_{F_R\1}(\t\1) := [\Psi_{F_R}(\t_0),\ldots, \Psi_{F_R}(\t_k)]$. Since $\Psi_{F_R}$ is simplicial, this is a simplex in $F_R$, i.e. a vertex in $F_R\1$. Thus we have a vertex map $\Psi_{F_R\1}:E_R^{(2)} \r F_R\1$. To check that this map is simplicial, let $\t^{(2)} = [\t\1_0,\ldots, \t\1_p]$ be a simplex in $E_R^{(2)}$. Then $\t\1_0 \subseteq \t\1_1 \subseteq \ldots \subseteq \t\1_p$, and because $\Psi_{F_R}$ is simplicial, we automatically have 
\[\Psi_{F_R}(\t\1_0) \subseteq \Psi_{F_R}(\t\1_1) \subseteq \ldots \subseteq \Psi_{F_R}(\t\1_p).\] 
Thus $\Psi_{F_R\1}(\t^{(2)})$ is a simplex in $F_R\1$.

\begin{proof}[Proof of Theorem \ref{thm:dowker-functorial}]
We write $F_R^{(2)}$ to denote the barycentric subdivision of $F_R\1$, and obtain simplicial maps $\Phi_{F_R\1}:F_R^{(2)} \r F_R\1$, $\Phi_{F_R}:F_R^{(1)} \r F_R$, $\Psi_{E_R\1}:F_R^{(2)} \r E_R\1$, and $\Psi_{F_R}:E_R^{(1)} \r F_R$ as above. Consider the following diagram:

\begin{center}
\begin{tikzpicture}
\node (1) at (0,0){$F_R^{(2)}$};
\node (2) at (3,-1){$F_R^{(1)}$};
\node (3) at (6,-2){$F_R$};
\node (4) at (8,0){$F_{R'}^{(2)}$};
\node (5) at (11,-1){$F_{R'}^{(1)}$};
\node (6) at (14,-2){$F_{R'}$};
\node (7) at (-1,-3){$E_R^{(2)}$};
\node (8) at (2,-4){$E_R\1$};
\node (9) at (5,-5){$E_R$};
\node (10) at (7,-3){$E_{R'}^{(2)}$};
\node (11) at (10,-4){$E_{R'}\1$};
\node (12) at (13,-5){$E_{R'}$};

\draw (1) edge[->,violet!80,thick] node[right]{\small{$\Phi_{F_R\1}$}} (2);
\draw (2) edge[->,violet!80,thick] node[above]{\small{$\Phi_{F_R}$}} (3);
\draw (1) edge[->] node[left]{} (4);
\draw (4) edge[->] node[above right]{$\Phi_{F_{R'}\1}$} (5);
\draw (5) edge[->] node[above right]{$\Phi_{F_{R'}}$} (6);
\draw (3) edge[->] node[above]{} (6);

\draw (7) edge[->,teal!80,thick] node[below]{\small{$\Phi_{F_R\1}$}} (8);
\draw (8) edge[->,teal!80,thick] node[below]{\small{$\Phi_{F_R}$}} (9);
\draw (7) edge[->,dashed] node[left]{} (10);
\draw (10) edge[->,dashed] node[above right]{} (11);
\draw (11) edge[->,dashed] node[above]{} (12); 
\draw (9) edge[->] node[above]{} (12);

\draw (1) edge[->,orange,thick] node[left,pos=0.2]{\small{$\Psi_{E_R\1}$}} (8);
\draw (8) edge[->,orange,thick] node[below right,pos=0.2]{\small{$\Psi_{F_R}$}} (3);
\draw (4) edge[->,dashed] node[left]{} (11);
\draw (11) edge[->,dashed] node[below right]{\small{}} (6);

\begin{scope}[on background layer]
\draw (7) edge[->,NavyBlue!80,thick] node[above left,pos=0.2]{\small{$\Psi_{F_R\1}$}} (2);
\draw (2) edge[->,NavyBlue!80,thick] node[left,pos=0.2]{\small{$\Psi_{E_R}$}} (9);
\draw (10) edge[->,dashed] node[left]{} (5);
\draw (5) edge[->,dashed] node[below right]{\small{}} (12);
\end{scope}
\end{tikzpicture}
\end{center}

We proceed by claiming contiguity of the following.

\begin{center}
\begin{tikzpicture}[]
\begin{scope}
\node (1) at (0,0){$E_R^{(2)}$};
\node (2) at (2,0){$E_R^{(1)}$};
\node (3) at (4,0){$E_R$};
\node (7) at (2,2){$F_R\1$};
\draw (1) edge[->,teal] node[below]{\small{$\Phi_{E_R\1}$}} (2);
\draw (2) edge[->,teal] node[below]{\small{$\Phi_{E_R}$}} (3);
\draw (1) edge[->,NavyBlue] node[left]{\small{$\Psi_{F_R\1}$}} (7);
\draw (7) edge[->,NavyBlue] node[right]{\small{$\Psi_{E_R}$}} (3);
\end{scope}
\begin{scope}[xshift=3in]
\node (1) at (0,2){$F_R^{(2)}$};
\node (2) at (2,2){$F_R^{(1)}$};
\node (3) at (4,2){$F_R$};
\node (7) at (2,0){$E_R\1$};
\draw (1) edge[->,violet] node[below]{\small{$\Phi_{F_R\1}$}} (2);
\draw (2) edge[->,violet] node[below]{\small{$\Phi_{F_R}$}} (3);
\draw (1) edge[->,orange] node[below left]{\small{$\Psi_{E_R\1}$}} (7);
\draw (7) edge[->,orange] node[below right]{\small{$\Psi_{F_R}$}} (3);
\end{scope}
\begin{scope}[yshift=-1.5in]
\node (1) at (0,2){$F_R^{(2)}$};
\node (2) at (3,2){$F_R^{(1)}$};
\node (7) at (3,0){$E_R\1$};
\node (3) at (6,0){$E_R$};
\draw (1) edge[->,violet] node[below]{\small{$\Phi_{F_R\1}$}} (2);
\draw (2) edge[->,NavyBlue] node[left]{\small{$\Psi_{E_R}$}} (3);
\draw (1) edge[->,orange] node[below left]{\small{$\Psi_{E_R\1}$}} (7);
\draw (7) edge[->,teal] node[above]{\small{$\Phi_{E_R}$}} (3);
\end{scope}
\begin{scope}[yshift=-1.5in,xshift=3in]
\node (1) at (0,0){$E_R^{(2)}$};
\node (2) at (3,0){$E_R^{(1)}$};
\node (7) at (3,2){$F_R\1$};
\node (3) at (6,2){$F_R$};
\draw (1) edge[->,teal] node[above]{\small{$\Phi_{E_R\1}$}} (2);
\draw (2) edge[->,orange] node[below right]{\small{$\Psi_{F_R}$}} (3);
\draw (1) edge[->,NavyBlue] node[above left]{\small{$\Psi_{F_R\1}$}} (7);
\draw (7) edge[->,violet] node[below]{\small{$\Phi_{F_R}$}} (3);
\end{scope}
\end{tikzpicture}
\end{center}

\begin{claim}\label{cl:dowker-contigo-items} More specifically:
\begin{enumerate}
\item $\Phi_{E_R}\circ \Phi_{E_R\1}$ and $\Psi_{E_R}\circ \Psi_{F_R\1}$ are contiguous. \label{item:dowker-contigo-1}
\item $\Phi_{F_R}\circ \Phi_{F_R\1}$ and $\Psi_{F_R}\circ \Psi_{E_R\1}$ are contiguous. \label{item:dowker-contigo-2}
\item $\Psi_{E_R}\circ \Phi_{F_R\1}$ and $\Phi_{E_R}\circ \Psi_{E_R\1}$ are contiguous. \label{item:dowker-contigo-3}
\item $\Psi_{F_R}\circ \Phi_{E_R\1}$ and $\Phi_{F_R}\circ \Psi_{F_R\1}$ are contiguous. \label{item:dowker-contigo-4}
\end{enumerate}
\end{claim}

Items (\ref{item:dowker-contigo-1}) and (\ref{item:dowker-contigo-3}) appear in the proof of Dowker's theorem \cite[Lemmas 5, 6]{dowker1952homology}, and it is easy to see that a symmetric argument shows Items (\ref{item:dowker-contigo-2}) and (\ref{item:dowker-contigo-4}). For completeness, we will verify these items in this paper, but defer this verification to the end of the proof. 

By passing to the geometric realization and applying Proposition \ref{prop:contigo-props} and Remark \ref{rem:subdiv-id}, we obtain the following from Item (\ref{item:dowker-contigo-3}) of Claim \ref{cl:dowker-contigo-items}:
\begin{align*}
|\Psi_{E_R}|\circ |\Phi_{F_R\1}| &\simeq |\Phi_{E_R}|\circ|\Psi_{E_R\1}|,\\ 
|\Psi_{E_R}|\circ \mc{E}_{|F_R\1|} &\simeq \mc{E}_{|E_R|}\circ|\Psi_{E_R\1}|, &&\text{(Remark \ref{rem:subdiv-id})}\\
\mc{E}\inv_{|E_R|} \circ |\Psi_{E_R}| \circ \mc{E}_{|F_R\1|} &\simeq |\Psi_{E_R\1}|. &&\text{($\mc{E}$ is a homeomorphism, hence invertible)}
\end{align*}
Replacing this term in the expression for Item (\ref{item:dowker-contigo-2}) of Claim \ref{cl:dowker-contigo-items}, we obtain:
\begin{align*}
|\Psi_{F_R}| \circ |\Psi_{E_R\1}| &\simeq |\Phi_{F_R}|\circ |\Phi_{F_R\1}| \simeq \mc{E}_{|F_R|}\circ \mc{E}_{|F_R\1|},\\
|\Psi_{F_R}| \circ \mc{E}\inv_{|E_R|} \circ |\Psi_{E_R}| \circ \mc{E}_{|F_R\1|} &\simeq \mc{E}_{|F_R|}\circ \mc{E}_{|F_R\1|},\\
|\Psi_{F_R}| \circ \mc{E}\inv_{|E_R|} \circ |\Psi_{E_R}|\circ \mc{E}\inv_{|F_R|}  &\simeq \id_{|F_R|}.
\end{align*}
Similarly, we obtain the following from Item (\ref{item:dowker-contigo-4}) of Claim \ref{cl:dowker-contigo-items}:
\[|\Psi_{F_R}|\circ |\Phi_{E_R\1}| \simeq |\Phi_{F_R}| \circ |\Psi_{F_R\1}|, \text{ so } \mc{E}\inv_{|F_R|} \circ |\Psi_{F_R}| \circ \mc{E}_{|E_R\1}| \simeq |\Psi_{F_R\1}|.\]
Replacing this term in the expression for Item (\ref{item:dowker-contigo-1}) of Claim \ref{cl:dowker-contigo-items}, we obtain:
\begin{align*}
|\Psi_{E_R}|\circ |\Psi_{F_R\1}| & \simeq|\Phi_{E_R}|\circ |\Phi_{E_R\1}| \simeq \mc{E}_{|E_R|}\circ \mc{E}_{|E_R\1|}, \\
|\Psi_{E_R}|\circ \mc{E}\inv_{|F_R|} \circ |\Psi_{F_R}| \circ \mc{E}_{|E_R\1}| &\simeq \mc{E}_{|E_R|}\circ \mc{E}_{|E_R\1|}\\
|\Psi_{E_R}|\circ \mc{E}\inv_{|F_R|} \circ |\Psi_{F_R}|\circ \mc{E}\inv_{|E_R|} &\simeq \id_{|E_R|}
\end{align*}
Define $\Gamma_{|E_R|}: |F_R| \r |E_R|$ by $\Gamma_{|E_R|} := |\Psi_{E_R}|\circ \mc{E}\inv_{|F_R|}$. Then $\Gamma_{|E_R|}$ is a homotopy equivalence, with inverse given by $|\Psi_{F_R}|\circ \mc{E}\inv_{|E_R|}$. This shows that $|F_R|\simeq |E_R|$, for any nonempty relation $R\subseteq X\times Y$.

Next we need to show that $\Gamma_{|E_R|}$ commutes with the canonical inclusion. Consider the following diagram, where the $\iota_\bullet$ maps denote the respective canonical inclusions (cf. Definition \ref{def:subdivision}):

\begin{center}
\begin{tikzpicture}
\node (1) at (0,0){$F_R\1$};
\node (2) at (3,0){$F_{R'}\1$};
\node (3) at (0,-1.5){$F_R$};
\node (4) at (3,-1.5){$F_{R'}$};
\node (5) at (-2,-3){$E_R$};
\node (6) at (5,-3){$E_{R'}$};
\draw (1) edge[->] node[above]{\small{$\iota_{F\1}$}} (2);
\draw (3) edge[->] node[above]{\small{$\iota_F$}} (4);
\draw (5) edge[->] node[above]{\small{$\iota_E$}} (6);
\draw (1) edge[->] node[right]{\small{$\Phi_{F_R}$}} (3);
\draw (2) edge[->] node[left]{\small{$\Phi_{F_{R'}}$}} (4);
\draw (1) edge[->] node[left]{\small{$\Psi_{E_{R}}$}} (5);
\draw (2) edge[->] node[right]{\small{$\Psi_{E_{R'}}$}} (6);
\end{tikzpicture}
\end{center}

\begin{claim}\label{cl:dowker-func-1}
$\iota_E\circ \Psi_{E_R}$ and $\Psi_{E_{R'}}\circ \iota_{F\1}$ are contiguous.
\end{claim}
\begin{claim}\label{cl:dowker-func-2}
$\iota_F\circ \Phi_{F_R}$ and $\Phi_{F_{R'}}\circ \iota_{F\1}$ are contiguous.
\end{claim}
Suppose Claim \ref{cl:dowker-func-2} is true. Then, upon passing to geometric realizations, we have:
\begin{align*}
|\iota_F| \circ \mc{E}_{|F_R|} \simeq |\iota_F|\circ |\Phi_{F_R}| \simeq |\Phi_{F_{R'}}|\circ |\iota_{F\1}| &\simeq \mc{E}_{|F_{R'}|}\circ |\iota_{F\1}|,\\
\mc{E}\inv_{|F_{R'}|}\circ |\iota_F| \circ \mc{E}_{|F_R|} &\simeq |\iota_{F\1}|.
\end{align*}
Suppose also that Claim \ref{cl:dowker-func-1} is true. Then we have:
\begin{align*}
|\Psi_{E_{R'}}| \circ |\iota_{F\1}| &\simeq |\iota_E|\circ |\Psi_{E_R}|,\\
|\Psi_{E_{R'}}| \circ \mc{E}\inv_{|F_{R'}|}\circ |\iota_F| \circ \mc{E}_{|F_R|} &\simeq |\iota_E|\circ |\Psi_{E_R}|,\\
|\Psi_{E_{R'}}| \circ \mc{E}\inv_{|F_{R'}|}\circ |\iota_F| &\simeq |\iota_E|\circ |\Psi_{E_R}| \circ \mc{E}\inv_{|F_R|}, \text{ i.e. }\\
\Gamma_{|E_{R'}|}\circ |\iota_F| &\simeq |\iota_E| \circ \Gamma_{|E_R|}.
\end{align*}

This proves the theorem. It only remains to prove the various claims.

\begin{subproof}[Proof of Claim \ref{cl:dowker-contigo-items}]
In proving Claim \ref{cl:dowker-contigo-items}, we supply the proofs of Items (\ref{item:dowker-contigo-2}) and (\ref{item:dowker-contigo-4}). These arguments are adapted from \cite[Lemmas 1, 5, and 6]{dowker1952homology}, where the proofs of Items (\ref{item:dowker-contigo-1}) and (\ref{item:dowker-contigo-3}) appeared.

For Item (\ref{item:dowker-contigo-2}), let $\t^{(2)}=[\t_0\1,\ldots,\t_k\1]$ be a simplex in $F_R^{(2)}$, where $\t_0\1 \subseteq \ldots \subseteq \t_k\1$ is a chain of simplices in $F_R\1$. By the order-reversing property of the map $\Phi_{F_R\1}$, we have that $\Phi_{F_R\1}(\t_0\1) \supseteq \Phi_{F_R\1}(\t_i\1)$ for each $i=0,\ldots, k$. Define $x:= \Psi_{E_R}(\Phi_{F_R\1}(\t_0\1))$. Then $(x,y) \in R$ for each $y \in \Phi_{F_R\1}(\t_0\1)$. But we also have $(x,\Phi_{F_R}(\Phi_{F_R\1}(\t_i\1))) \in R$ for each $i=0,\ldots, k$, because $\Phi_{F_R}(\Phi_{F_R\1}(\t_i\1)) \in \Phi_{F_R\1}(\t_i\1) \subseteq \Phi_{F_R}(\t_0\1)$ for each $i=0,\ldots, k$.

Next let $0\leq i \leq k$. For each $\t \in \t\1_i$, we have $\Psi_{E_R}(\t) \in \Psi_{E_R\1}(\t\1_i)$ (by the definition of $\Psi_{E_R\1}$). Because $\Phi_{F_R\1}(\t_0\1) \in \t_0\1 \subseteq \t_i\1$, we then have $x = \Psi_{E_R}(\Phi_{F_R\1}(\t_0\1)) \in \Psi_{E_R\1}(\t\1_i)$, which is a vertex of $E_R\1$ or alternatively a simplex of $E_R$. But then, by definition of $\Psi_{F_R}$, we have that $(x,\Psi_{F_R}(\Psi_{E_R\1}(\t_i\1))) \in R$. This holds for each $0\leq i \leq k$. Since $\t^{(2)}$ was arbitrary, this shows that $\Phi_{F_R}\circ \Phi_{F_R\1}$ and $\Psi_{F_R}\circ \Psi_{E_R\1}$ are contiguous. 

For Item (\ref{item:dowker-contigo-4}), let $\s^{(2)} = [\s\1_0,\ldots,\s\1_k]$ be a simplex in $E_R^{(2)}$. Let $0\leq i \leq k$. Then $\s\1_0 \subseteq \ldots \subseteq \s\1_k$, and $\Phi_{E_R\1}(\s\1_i) \in \s\1_i \subseteq \s\1_k$. So $\Psi_{F_R}(\Phi_{E_R\1}(\s_i\1)) \in \Psi_{F_R\1}(\s_k\1)$. On the other hand, we have $\Psi_{F_R\1}(\s_i\1) \subseteq \Psi_{F_R\1}(\s_k\1)$. Then $\Phi_{F_R}(\Psi_{F_R\1}(\s_i\1)) \in \Psi_{F_R\1}(\s_i\1) \subseteq \Psi_{F_R\1}(\s_k\1)$. Since $i$ was arbitrary, this shows that $\Psi_{F_R}\circ \Phi_{E_R\1}$ and $\Phi_{F_R}\circ \Psi_{F_R\1}$ both map the vertices of $\s^{(2)}$ to the simplex $\Psi_{F_R\1}(\s_k\1)$, hence are contiguous. This concludes the proof of the claim.
\end{subproof}

\begin{subproof}[Proof of Claim \ref{cl:dowker-func-1}]
Let $\t\1=[\t_0,\t_1,\ldots, \t_k] \in F_R\1$, where $\t_0 \subseteq \t_1\subseteq \ldots \subseteq \t_k$ is a chain of simplices in $F_R$. Then $\iota_{F\1}(\t\1) = \t\1$, and $\Psi_{E_{R'}}(\t\1)=[x_{\t_0},\ldots, x_{\t_k}]$, for some choice of $x_{\t_i}$ terms. Also we have $\iota_E\circ \Psi_{E_{R}}(\t\1) = [x'_{\t_0},\ldots, x'_{\t_k}]$ for some other choice of $x'_{\t_i}$ terms. For contiguity, we need to show that 
\[ [x_{\t_0},\ldots, x_{\t_k}, x'_{\t_0},\ldots, x'_{\t_k}] \in E_{R'}.\]
But this is easy to see: letting $y \in \t_0$, we have $\set{(x_{\t_0},y),\ldots,(x_{\t_k},y),(x'_{\t_0},y),\ldots,(x'_{\t_k},y)} \subseteq R$. Since $\t\1$ was arbitrary, it follows that we have contiguity. 
\end{subproof}

\begin{subproof}[Proof of Claim \ref{cl:dowker-func-2}]
Let $\t\1=[\t_0,\t_1,\ldots, \t_k] \in F_R\1$, where $\t_0 \subseteq \t_1\subseteq \ldots \subseteq \t_k$ is a chain of simplices in $F_R$. Then $\Phi_{F_R}(\t_i) \in \t_k$ for each $0\leq i \leq k$. Thus $\iota_F\circ \Phi_{F_R}(\t\1)$ is a face of $\t_k$. Similarly, $\Phi_{F_{R'}}\circ \iota_{F\1}(\t\1)$ is also a face of $\t_k$. Since $\t\1$ was an arbitrary simplex of $F_R\1$, it follows that $\iota_F\circ \Phi_{F_R}$ and $\Phi_{F_{R'}}\circ \iota_{F\1}$ are contiguous.\end{subproof}
\end{proof}

\subsection{The equivalence between the finite FDT and the simplicial FNTs}
\label{sec:dowker-nerve-equiv}

In this section, we present our answer to Question \ref{q:f-nerve-f-dowker}. We begin with a weaker formulation of Theorem \ref{thm:dowker-functorial} and some simplicial Functorial Nerve Theorems.

\begin{theorem}[The finite FDT]
\label{thm:dowker-functorial-finite}
Let $X,Y$ be two totally ordered sets, and without loss of generality, suppose $X$ is finite. Let $R\subseteq R' \subseteq X\times Y$ be two nonempty relations, and let $E_R, F_R, E_{R'}, F_{R'}$ be their associated simplicial complexes (as in Theorem \ref{thm:dowker-functorial}). Then there exist homotopy equivalences $\Gamma_{|E_R|}:|F_R| \r |E_R|$ and $\Gamma_{|E_{R'}|}: |F_{R'}| \r |E_{R'}|$ that commute up to homotopy with the canonical inclusions.
\end{theorem}

The finite FDT (Theorem \ref{thm:dowker-functorial-finite}) is an immediate consequence of the general FDT (Theorem \ref{thm:dowker-functorial}).

\begin{definition} Let $\mc{A} = \set{A_i}_{i\in I}$ be a family of nonempty sets indexed by $I$. The \emph{nerve} of $\mc{A}$ is the simplicial complex $\mc{N}(\mc{A}):= \{\s \in \pow(I) : \s \text{ is finite, nonempty, and } \cap_{i \in \s}A_i \neq \emptyset\}$.

\end{definition}

\begin{definition}[Covers of simplices and subcomplexes] Let $\Si$ be a simplicial complex. Then a collection of subcomplexes $\mc{A}_\Si = \{\Si_i\}_{i\in I}$ is said to be a \emph{cover of subcomplexes} for $\Si$ if $\Si = \cup_{i\in I}\Si_i$.
Furthermore, $\mc{A}_\Si$ is said to be a \emph{cover of simplices} if each $\Si_i \in \mc{A}_\Si$ has the property that $\Si_i = \pow(V(\Si_i))$. In this case, each $\Si_i$ has precisely one top-dimensional simplex, consisting of the vertex set $V(\Si_i)$.
\end{definition}

We present two \emph{simplicial} formulations of the Functorial Nerve Theorem that turn out to be equivalent; the statements differ in that one is about covers of simplices and the other is about covers of subcomplexes.


\begin{theorem}[Functorial Nerve I] 
\label{thm:nerve-functorial-I}
Let $\Si \subseteq \Si'$ be two simplicial complexes, and let $\mc{A}_\Si=\{\Si_i\}_{i\in I}$, $\mc{A}_{\Si'}=\{\Si'_i\}_{i\in I'}$ be finite covers of simplices for $\Si$ and $\Si'$ such that $I\subseteq I'$ and $\Si_i \subseteq \Si'_i$ for each $i \in I$. In particular, $\card(I') < \infty$. Suppose that for each finite subset $\s \subseteq I'$, the intersection $\cap_{i \in \s}\Si'_i$ is either empty or contractible (and likewise for $\cap_{i \in \s}\Si_i$). Then $|\Si| \simeq |\mc{N}(\mc{A}_\Si)|$ and $|\Si'| \simeq |\mc{N}(\mc{A}_{\Si'})|$, via maps that commute up to homotopy with the canonical inclusions.
\end{theorem}

\begin{theorem}[Functorial Nerve II] 
\label{thm:nerve-functorial-II}
The statement of Theorem \ref{thm:nerve-functorial-I} holds even if $\mc{A}_\Si$ and $\mc{A}_{\Si'}$ are covers of subcomplexes. Explicitly, the statement is as follows. Let $\Si \subseteq \Si'$ be two simplicial complexes, and let $\mc{A}_\Si=\{\Si_i\}_{i\in I}$, $\mc{A}_{\Si'}=\{\Si'_i\}_{i\in I'}$ be finite covers of subcomplexes for $\Si$ and $\Si'$ such that $I\subseteq I'$ and $\Si_i \subseteq \Si'_i$ for each $i \in I$. In particular, $\card(I') < \infty$. Suppose that for each finite subset $\s \subseteq I'$, the intersection $\cap_{i \in \s}\Si'_i$ is either empty or contractible (and likewise for $\cap_{i \in \s}\Si_i$). Then $|\Si| \simeq |\mc{N}(\mc{A}_\Si)|$ and $|\Si'| \simeq |\mc{N}(\mc{A}_{\Si'})|$, via maps that commute up to homotopy with the canonical inclusions.
\end{theorem}

The following result summarizes our answer to Question \ref{q:f-nerve-f-dowker}.

\begin{theorem}[Equivalence]\label{thm:dowker-nerve-eq}
The finite FDT, the FNT I, and the FNT II are all equivalent. Moreover, all of these results are implied by the FDT, as below:
\begin{center}
\begin{tikzpicture}
\node (fdt) at (-2,0){Theorem \ref{thm:dowker-functorial}};
\node (ffdt) at (1,0){Theorem \ref{thm:dowker-functorial-finite}};
\node (sfnt1) at (3,1){Theorem \ref{thm:nerve-functorial-I}};
\node (sfnt2) at (5,0){Theorem \ref{thm:nerve-functorial-II}};

\draw[-implies,double equal sign distance] (fdt) --  (ffdt);
\draw[-implies,double equal sign distance] (ffdt) --  (sfnt1);
\draw[-implies,double equal sign distance] (sfnt1) --  (sfnt2);
\draw[-implies,double equal sign distance] (sfnt2) --  (ffdt);
\end{tikzpicture}
\end{center}
\end{theorem}
We present the proof of Theorem \ref{thm:dowker-nerve-eq} over the course of the next few subsections.

\begin{remark} By virtue of Theorem \ref{thm:dowker-nerve-eq}, we will write \emph{simplicial FNT} to mean either of the FNT I or FNT II.
\end{remark}

\subsection*{Theorem \ref{thm:dowker-functorial-finite} implies Theorem \ref{thm:nerve-functorial-I}}

\begin{proof}[Proof of Theorem \ref{thm:nerve-functorial-I}] Let $V, V'$ denote the vertex sets of $\Si,\Si'$, respectively. We define the relations $R\subseteq V\times I$ and $R' \subseteq V'\times I'$ as follows: $(v,i) \in R \iff v \in \Si_i$ and $(v',i') \in R' \iff v' \in \Si_i'.$ Then $R \subseteq R'$, the set $I'$ is finite by assumption, and so we are in the setting of the finite FDT (Theorem \ref{thm:dowker-functorial-finite}) (perhaps invoking the Axiom of Choice to obtain the total order on $V'$). It suffices to show that $E_R = \Si$, $E_{R'} = \Si'$, $F_R = \mc{N}(\mc{A}_\Si)$, and $F_{R'} = \mc{N}(\mc{A}_{\Si'})$, where $E_R, E_{R'}, F_R, F_{R'}$ are as defined in Theorem \ref{thm:dowker-functorial}.

First we claim the $E_R = \Si$. By the definitions of $R$ and $E_R$, we have $E_R = \{\s \subseteq V : \; \exists  i \in I, \;
(v,i) \in R \; \forall \; v\in \s\} 
= \{\s \subseteq V : \;\exists  i \in I, \;
v\in \Si_i \; \forall \; v\in \s\}.$ Let $\s \in E_R$, and let $i \in I$ be such that $v \in \Si_i$ for all $v \in \s$. 
Then $\s \subseteq V(\Si_i)$, and since $\Si_i = \pow(V(\Si_i))$ by the assumption about covers of simplices, we have $\s \in \Si_i \subseteq \Si$. Thus $E_R \subseteq \Si$. Conversely, let $\s \in \Si$. Then $\s \in \Si_i$ for some $i$. Thus for all $v \in \s$, we have $(v,i) \in R$. It follows that $\s \in E_R$. This shows $E_R = \Si$. The proof that $E_{R'} = \Si'$ is analogous.

Next we claim that $F_R = \mc{N}(\mc{A}_\Si)$. By the definition of $F_R$, we have $F_R = \{\t \subseteq I : \; \exists  v \in V, \;
(v,i) \in R \; \forall \; i\in \t\}.$ Let $\t \in F_R$, and let $v \in V$ be such that $(v,i) \in R$ for each $i \in \t$. Then $\cap_{i \in \t}\Si_i \neq \emptyset$, and so $\t \in \mc{N}(\mc{A}_\Si)$. Conversely, let $\t \in \mc{N}(\mc{A}_\Si)$. Then $\cap_{i \in \t}\Si_i \neq \emptyset$, so there exists $v \in V$ such that $v \in \Si_i$ for each $i \in \t$. Thus $\s \in F_R$. This shows $F_R = \mc{N}(\mc{A}_\Si)$. The case for $R'$ is analogous.

An application of Theorem \ref{thm:dowker-functorial-finite} now completes the proof. \end{proof}

\subsection*{Theorem \ref{thm:nerve-functorial-II} implies Theorem \ref{thm:dowker-functorial-finite}}

\begin{proof}

Let $X$ and $Y$ be two sets, and suppose $X$ is finite. Let $R\subseteq R' \subseteq X \times Y$ be two relations. Consider the simplicial complexes $E_R, F_R, E_{R'}, F_{R'}$ as defined in Theorem \ref{thm:dowker-functorial}. Let $V_R:=V(E_R)$. For each $x \in V_R$, define $A_x:=\{\t \in F_R : (x,y) \in R \text{ for all } y \in \t\}$. Then $A_x$ is a subcomplex of $F_R$. Furthermore, $\cup_{x\in V_R}A_x = F_R$. To see this, let $\t\in F_R$. Then there exists $x\in X$ such that $(x,y) \in R$ for all $y\in \t$, and so $\t \in A_x$. 

Let $\mc{A}:= \{A_x : x \in V_R\}$. We have seen that $\mc{A}$ is a cover of subcomplexes for $F_R$. It is finite because the indexing set $V_R$ is a subset of $X$, which is finite by assumption. Next we claim that $\mc{N}(\mc{A}) = E_R$. Let $\s \in E_R$. Then there exists $y \in Y$ such that $(x,y) \in R$ for all $x\in \s$. Thus $\cap_{x\in \s}A_x \neq \emptyset$, and so $\s \in \mc{N}(\mc{A})$. Conversely, let $\s \in \mc{N}(\mc{A})$. Then $\cap_{x \in \s}A_x \neq \emptyset$, and so there exists $y \in Y$ such that $(x,y) \in R$ for all $x \in \s$. Thus $\s \in E_R$. 

Next we check that nonempty finite intersections of elements in $\mc{A}$ are contractible. Let $\s \in \mc{N}(\mc{A}) = E_R$. Let $V_\s:= \cap_{x \in \s}V(A_x) \subseteq V(F_R)$. We claim that $\cap_{x \in \s}A_x = \pow(V_\s)$, i.e. that the intersection is a full simplex in $F_R$, hence contractible. The inclusion $\cap_{x \in \s}A_x \subseteq \pow(V_\s)$ is clear, so we show the reverse inclusion. Let $\t \in \pow(V_\s)$, and let $y \in \t$. Then $y \in \cap_{x \in \s}A_x$, so $(x,y) \in R$ for each $x \in \s$. This holds for each $y \in \t$, so it follows that $\t \in \cap_{x\in \s}A_x$. Thus $\cap_{x \in \s}A_x = \pow(V_\s)$. We remark that this also shows that $\mc{A}$ is a cover of simplices for $F_R$. 

Now for each $x \in V(E_{R'})$, define $A'_x:= \{\t \in F_{R'} : (x,y) \in R' \text{ for all } y \in \t\}$. Also define $\mc{A}':= \{A'_x : x \in V(E_{R'})\}$. The same argument shows that $\mc{A}'$ is a finite cover of subcomplexes (in particular, a cover of simplices) for $F_{R'}$ with all finite intersections either empty or contractible, and that $E_{R'} = \mc{N}(\mc{A}')$. An application of Theorem \ref{thm:nerve-functorial-II} now shows that $|E_R| \simeq |F_R|$ and $|E_{R'}| \simeq |F_{R'}|$, via maps that commute up to homotopy with the inclusions $|E_R| \hr |E_{R'}|$ and $|F_R| \hr |F_{R'}|$. \end{proof}

\subsection*{Theorem \ref{thm:nerve-functorial-I} implies Theorem \ref{thm:nerve-functorial-II}}

We lead with some remarks about the ideas involved in this proof. Theorem \ref{thm:nerve-functorial-II} is a functorial statement in the sense that it is about an arbitrary inclusion $\Si \subseteq \Si'$. Restricting the statement to just $\Si$ would lead to a non-functorial statement. A proof of this non-functorial statement, via a non-functorial analogue of Theorem \ref{thm:nerve-functorial-I}, can be obtained using techniques presented in \cite{bjorner1985homotopy} (see also \cite[Theorem 15.24]{kozlov2007combinatorial}). We strengthen these techniques to our functorial setting and thus obtain a proof of Theorem \ref{thm:nerve-functorial-II} via Theorem \ref{thm:nerve-functorial-I}.

We first present a lemma related to barycentric subdivisions and several lemmas about gluings and homotopy equivalences. These will be used in proving Theorem \ref{thm:nerve-functorial-II}.

\begin{definition}[Induced subcomplex] Let $\Si$ be a simplicial complex, and let $\Delta$ be a subcomplex. Then $\Delta$ is an \emph{induced subcomplex} if $\Delta = \Si \cap \pow(V(\Delta))$. 
\end{definition}

\begin{lemma} 
\label{lem:induced-cplx}
Let $\Si$ be a simplicial complex, and let $\Delta$ be a subcomplex. Then $\Delta\1$ is an \emph{induced subcomplex} of $\Si\1$, i.e. $\Delta\1 = \Si\1 \cap \pow(V(\Delta\1))$. 
\end{lemma}

\begin{proof} Let $\s$ be a simplex of $\Delta\1$. Then $\s$ belongs to $\Si\1$, and also to the full simplex $\pow(V(\Delta\1))$. Thus $\Delta\1 \subseteq \Si\1 \cap \pow(V(\Delta\1))$. Conversely, let $\s \in \Si\1 \cap \pow(V(\Delta\1))$. Since $\s \in \Si\1$, we can write $\s = [\t_0,\ldots, \t_k]$, where $\t_0 \subseteq \ldots \subseteq \t_k$. Since $\s \in \pow(V(\Delta\1))$ and the vertices of $\Delta\1$ are simplices of $\Delta$, we also know that each $\t_i$ is a simplex of $\Delta$. Thus $\s \in \Delta\1$. The equality follows.\end{proof}

\begin{lemma}[Carrier Lemma, \cite{bjorner1985homotopy} \S4]
\label{lem:carrier}
Let $X$ be a topological space, and let $\Si$ be a simplicial complex. Also let $f,g:X \r |\Si|$ be any two continuous maps such that $f(x), g(x)$ belong to the same simplex of $|\Si|$, for any $x \in X$. Then $f \simeq g$.
\end{lemma}

\begin{lemma}[Gluing Lemma, see Lemmas 4.2, 4.7, 4.9, \cite{bjorner1985homotopy}]
\label{lem:glue}
Let $\Si$ be a simplicial complex, and let $U \subseteq V(\Si)$. Suppose $|\Si \cap \pow(U)|$ is contractible. Then there exists a homotopy equivalence $\ph: |\Si \cup \pow(U)| \r |\Si|$. 
\end{lemma}

The Gluing and Carrier Lemmas presented above are classical. We provide full details for the Gluing lemma inside the proof of the following functorial generalization of Lemma \ref{lem:glue}.

\begin{lemma}[Functorial Gluing Lemma]
\label{lem:glue-func}
Let $\Si \subseteq \Si'$ be two simplicial complexes. Also let $U \subseteq V(\Si)$ and $U' \subseteq V(\Si')$ be such that $U \subseteq U'$. Suppose $|\Si \cap \pow(U)|$ and $|\Si' \cap \pow(U')|$ are contractible. Then,
\begin{enumerate}
\item There exists a homotopy equivalence $\ph: |\Si \cup \pow(U)| \r |\Si|$ such that $\ph(x)$ and $\id_{|\Si \cup \pow(U)|}(x)$ belong to the same simplex of $|\Si\cup \pow(U)|$ for each $x \in |\Si \cup \pow(U)|$. Furthermore, the homotopy inverse is given by the inclusion $\iota: |\Si| \hr |\Si \cup \pow(U)|$.  
\item Given a homotopy equivalence $\ph: |\Si \cup \pow(U)| \r |\Si|$ as above, there exists a homotopy equivalence $\ph': |\Si' \cup \pow(U')| \r |\Si'|$ such that $\ph'|_{|\Si\cup \pow(U)|} = \ph$, and $\ph'(x)$ and $\id_{|\Si' \cup \pow(U')|}(x)$ belong to the same simplex of $|\Si'\cup \pow(U')|$ for each $x \in |\Si' \cup \pow(U')|$. Furthermore, the homotopy inverse is given by the inclusion $\iota': |\Si'| \hr |\Si' \cup \pow(U')|$. 
\end{enumerate}
\end{lemma}

\begin{proof}[Proof of Lemma \ref{lem:glue-func}] 
The proof uses this fact: any continuous map of an $n$-sphere $\us^n$ into a contractible space $Y$ can be continuously extended to a mapping of the $(n+1)$-disk $\ud^{n+1}$ into $Y$, where $\ud^{n+1}$ has $\us^n$ as its boundary \cite[p. 27]{spanier-book}. 
First we define $\ph$. On $|\Si|$, define $\ph$ to be the identity. Next let $\s$ be a minimal simplex in $|\pow(U) \setminus \Si|$. By minimality, the boundary of $\s$ (denoted $\bd(\s)$) belongs to $|\Si \cap \pow(U)|$, and $|\Si|$ in particular. 
Thus $\ph$ is defined on $\bd(\s)$, which is an $n$-sphere for some $n \geq 0$. Furthermore, $\ph$ maps $\bd(\s)$ into the contractible space $|\Si \cap \pow(U)|$. 
Then we use the aforementioned fact to extend $\ph$ continuously to all of $\s$ so that $\ph$ maps $\s$ into $|\Si \cap \pow(U)|$. Furthermore, both $\id_{|\Si \cup \pow(U)|}(\s) = \s$ and $\ph(\s)$ belong to the simplex $|\pow(U)|$.
By iterating this procedure, we obtain a retraction $\ph: |\Si \cup \pow(U)| \r |\Si|$ such that $\ph(x)$ and $x$ belong to the same simplex in $|\Si \cup \pow(U)|$, for each $x \in |\Si \cup \pow(U)|$. 
Thus $\ph$ is homotopic to $\id_{|\Si \cup \pow(U)|}$ by Lemma \ref{lem:carrier}. Thus we have a homotopy equivalence:
\[\id_{|\Si|}= \ph\circ \iota, \;  \iota \circ \ph \simeq \id_{|\Si \cup \pow(U)|}
\tag{here $\iota:= \iota_{|\Si| \hr |\Si \cup \pow(U)|}$} .\]

%
For the second part of the proof, suppose that a homotopy equivalence $\ph: |\Si \cup \pow(U)| \r |\Si|$ as above is provided. We need to extend $\ph$ to obtain $\ph'$. Define $\ph'$ to be equal to $\ph$ on $|\Si \cup \pow(U)|$, and equal to the identity on $G:= |\Si'|\setminus |\Si \cup \pow(U)|$. 
Let $\s$ be a minimal simplex in $|\pow(U')| \setminus G$. 
Then by minimality, $\bd(\s)$ belongs to $|\Si' \cap \pow(U')|$. 
As before, we have $\ph'$ mapping $\bd(\s)$ into the contractible space $|\Si' \cap \pow(U')|$, and we extend $\ph'$ continuously to a map of $\s$ into $|\Si' \cap \pow(U')|$. 
Once again, $\id_{|\Si' \cup \pow(U')|}(x)$ and $\ph'(x)$ belong to the same simplex $|\pow(U')|$, for all $x \in \s$.
Iterating this procedure gives a continuous map $\ph':|\Si' \cup \pow(U')| \r |\Si'|$. This map is not necessarily a retraction, because there may be a simplex $\s \in |\Si\cup \pow(U)| \cap |\Si'|$ on which $\ph'$ is not the identity. However, it still holds that $\ph'$ is continuous, and that $x, \ph'(x)$ get mapped to the same simplex for each $x \in |\Si'\cup \pow(U')|$. Thus Lemma \ref{lem:carrier} still applies to show that $\ph'$ is homotopic to $\id_{|\Si' \cup \pow(U')|}$. 

We write $\iota'$ to denote the inclusion $\iota': |\Si'| \hr |\Si' \cup \pow(U')|$. By the preceding work, we have $\iota'\circ \ph' \simeq \id_{|\Si' \cup \pow(U')|}$. Next let $x \in |\Si'|$. Then either $x \in |\Si'| \cap |\Si \cup \pow(U)|$, or $x \in G$. In the first case, we know that $\ph'(x) = \ph(x)$ and $\id_{|\Si'|}(x) = \id_{|\Si\cup \pow(U)|}(x)$ belong to the same simplex of $|\Si\cup \pow(U)|$ by the assumption on $\ph$. In the second case, we know that $\ph'(x) = x = \id_{|\Si'|}(x)$. Thus for any $x \in |\Si'|$, we know that $\ph'(x)$ and $\id_{|\Si'|}(x)$ belong to the same simplex in $|\Si' \cup \pow(U')|$. By Lemma \ref{lem:carrier}, we then have $\ph'|_{|\Si'|} \simeq \id_{|\Si'|}$. Thus $\ph'\circ \iota' \simeq \id_{|\Si'|}$. This shows that $\ph'$ is the necessary homotopy equivalence. \end{proof}

Now we present the proof of Theorem \ref{thm:nerve-functorial-II}. 

\noindent
\textbf{Notation.} Let $I$ be an ordered set. For any subset $J\subseteq I$, we write $(J)$ to denote the sequence $(j_1,j_2,j_3,\ldots)$, where the ordering is inherited from the ordering on $I$.

\begin{proof}[Proof of Theorem \ref{thm:nerve-functorial-II}] The first step is to functorially deform $\mc{A}_\Si$ and $\mc{A}_{\Si'}$ into covers of simplices while still preserving all associated homotopy types. Then we will be able to apply Theorem \ref{thm:nerve-functorial-I}. We can assume by Lemma \ref{lem:induced-cplx} that each subcomplex $\Si_i$ is induced, and likewise for each $\Si'_i$. We start by fixing an enumeration $I' = \{l_1,l_2,\ldots\}$. Thus $I'$ becomes an ordered set.

\paragraph*{Passing to covers of simplices.} We now define some inductive constructions. In what follows, we will define complexes denoted $\Si^\bullet, \Si'^\bullet$ obtained by ``filling in" $\Si$ and $\Si'$ while  preserving homotopy equivalence, as well as covers of these larger complexes denoted $\Si_{\star,\bullet}, \Si'_{\star,\bullet}$. First define:
\begin{align*}
\Si^{(l_1)} &:= {\begin{cases}
\Si \cup \pow(V(\Si_{l_1})) &: \text{ if }l_1 \in I\\
\Si &: \text{ otherwise}.
\end{cases}}\\
\Si'^{(l_1)} &:= \Si' \cup \pow(V(\Si'_{l_1})).
\end{align*}
Next, for all $i\in I$, define 
\begin{align*}
\Si_{i,(l_1)} &:= {\begin{cases}
\Si_i \cup \pow(V(\Si_i) \cap V(\Si_{l_1}))&: \text{ if }l_1 \in I\\
\Si_i &: \text{ otherwise}.
\end{cases}}
\end{align*}
And for all $i \in I'$, define
\begin{align*}
\Si'_{i,(l_1)} &:= \Si'_i \cup \pow(V(\Si'_i) \cap V(\Si'_{l_1})).
\end{align*}
Now by induction, suppose $\Si^{(l_1,\ldots, l_n)}$ and $\Si_{i,(l_1,\ldots, l_n)}$ are defined for all $i \in I$. 
Also suppose $\Si'^{(l_1,\ldots, l_n)}$ and $\Si'_{i,(l_1,\ldots, l_n)}$ are defined for all $i \in I'$. Then we define:
\begin{align*}
\Si^{(l_1,\ldots, l_n,l_{n+1})} &:= {\begin{cases}
\Si^{(l_1,\ldots,l_n)} \cup \pow(V(\Si_{l_{n+1},(l_1,\ldots,l_n)})) &: \text{ if }l_{n+1} \in I\\
\Si^{(l_1,\ldots,l_n)} &: \text{ otherwise}.
\end{cases}}\\
\Si'^{(l_1,\ldots, l_n,l_{n+1})} &:= \Si'^{(l_1,\ldots,l_n)} \cup \pow(V(\Si'_{l_{n+1},(l_1,\ldots,l_n)})).
\end{align*}
For all $i \in I$, we have 
\begin{align*}
\Si_{i,(l_1,l_2,\ldots, l_{n+1})} &:= {\begin{cases}
\Si_{i,(l_1,l_2,\ldots, l_n)} \cup \pow(V(\Si_{i,(l_1,l_2,\ldots, l_n)}) \cap V(\Si_{l_{n+1},(l_1,l_2,\ldots, l_n)}))&: \text{ if }l_{n+1} \in I\\
\Si_{i,(l_1,l_2,\ldots, l_n)} &: \text{ otherwise}.
\end{cases}}
\end{align*}
And for all $i \in I'$, we have
\begin{align*}
\Si'_{i,(l_1,l_2,\ldots, l_{n+1})} &:= \Si'_{i,(l_1,l_2,\ldots, l_n)} \cup \pow(V(\Si'_{i,(l_1,l_2,\ldots, l_n)}) \cap V(\Si'_{l_{n+1},(l_1,l_2,\ldots, l_n)})).
\end{align*}

Finally, for any $n \leq \card(I')$, we define $\mc{A}_{\Si,(l_1,\ldots, l_n)} :=\{\Si_{i,(l_1,\ldots, l_{n})} : i \in I\}$ and $\mc{A}_{\Si',(l_1,\ldots, l_n)} :=\{\Si'_{i,(l_1,\ldots, l_{n})} : i \in I'\}$. We will show that these are covers of $\Si^{(l_1,l_2,\ldots, l_n)}$ and $\Si'^{(l_1,l_2,\ldots, l_n)}$, respectively.

The next step is to prove by induction that for any $n \leq \card(I')$, we have $|\Si|\simeq |\Si^{(l_1,l_2,\ldots, l_n)}|$ and $|\Si'|\simeq |\Si'^{(l_1,l_2,\ldots, l_n)}|$, that $\mc{N}(\mc{A}_\Si) = \mc{N}(\mc{A}_{\Si,(l_1,l_2,\ldots, l_n)})$ and $\mc{N}(\mc{A}_{\Si'}) = \mc{N}(\mc{A}_{\Si',(l_1,l_2,\ldots, l_n)})$, and that nonempty finite intersections of the new covers $\mc{A}_{\Si,(l_1,l_2,\ldots, l_n)}, \mc{A}_{\Si',(l_1,l_2,\ldots, l_n)}$ remain contractible. For the base case $n=0$, we have $\Si = \Si^{()}$, $\Si' = \Si'^{()}$. Thus the base case is true by assumption. We present the inductive step next. 

\begin{claim}
\label{cl:nerve-cover-of-simplices} 
For this claim, let $\bullet$ denote $l_1,\ldots, l_n$, where $0 < n < \card(I')$. Define $l:= l_{n+1}$. Suppose the following is true:
\begin{enumerate}
\item The collections $\mc{A}_{\Si,(\bullet)}$ and $\mc{A}_{\Si',(\bullet)}$ are covers of $\Si^{(\bullet)}$ and $\Si'^{(\bullet)}$. 
\item The nerves of the coverings are unchanged: $\mc{N}(\mc{A}_\Si) = \mc{N}(\mc{A}_{\Si,(\bullet)})$ and $\mc{N}(\mc{A}_{\Si'}) = \mc{N}(\mc{A}_{\Si',(\bullet)})$.
\item Each of the subcomplexes $\Si_{i,(\bullet)}$, $i\in I$, and $\Si'_{j,(\bullet)}$, $j\in I'$ is induced in $\Si^{(\bullet)}$ and $\Si'^{(\bullet)}$, respectively.
\item Let $\s \subseteq I$. If $\cap_{i\in \s}\Si_{i,(\bullet)}$ is nonempty, then it is contractible. Similarly, let $\t \subseteq I'$. If $\cap_{i\in \t}\Si'_{i,(\bullet)}$ is nonempty, then it is contractible.
\item We have homotopy equivalences $|\Si|\simeq |\Si^{(\bullet)}|$ and $|\Si'|\simeq |\Si'^{(\bullet)}|$ via maps that commute with the canonical inclusions.
\end{enumerate}
Then the preceding statements are true for $\Si^{(\bullet,l)}$, $\Si'^{(\bullet,l)}$, $\mc{A}_{\Si,(\bullet,l)}$, and $\mc{A}_{\Si',(\bullet,l)}$ as well.
\end{claim}

\begin{subproof}
For the first claim, we have $\Si^{(\bullet, l)} = \Si^{(\bullet)} \cup \pow(V(\Si_{l,(\bullet)})) \subseteq \cup_{i \in I}\Si_{i,(\bullet, l)}$. For the inclusion, we used the inductive assumption that $\Si^{(\bullet)} = \cup_{i \in I}\Si_{i,(\bullet)}$. Similarly, $\Si'^{(\bullet, l)} \subseteq \cup_{i \in I'} \Si'_{i,(\bullet,l)}$. 

For the second claim, let $i \in I$. Then $V(\Si_{i,(l_1)}) = V(\Si_i)$, and in particular, we have $V(\Si_{i,(\bullet,l)}) = V(\Si_{i,(\bullet)}) = V(\Si_i)$. Next observe that for any $\s \subseteq I$, the intersection 
\[\cap_{i \in \s}\Si_i \neq \emptyset \iff \cap_{i \in \s}V(\Si_i) \neq \emptyset \iff \cap_{i \in \s}V(\Si_{i,(\bullet,l)}) \neq \emptyset \iff \cap_{i \in \s}\Si_{i,(\bullet,l)} \neq \emptyset.\]
Thus $\mc{N}(\mc{A}_\Si) = \mc{N}(\mc{A}_{\Si,(\bullet)}) = \mc{N}(\mc{A}_{\Si,(\bullet,l)})$, and similarly $\mc{N}(\mc{A}_{\Si'}) = \mc{N}(\mc{A}_{\Si',(\bullet)}) = \mc{N}(\mc{A}_{\Si',(\bullet,l)})$.

For the third claim, again let $i \in I$. If $l \not\in I$, then $\Si_{i,(\bullet,l)} = \Si_{i,(\bullet)}$, so we are done by the inductive assumption. Suppose $l \in I$. Since $\Si_{i,(\bullet)}$ is induced by the inductive assumption, we have:
\begin{align*}
\Si_{i,(\bullet,l)} &= \Si_{i,(\bullet)} \cup (\pow(V(\Si_{i,(\bullet)})\cap V(\Si_{l,(\bullet)})))\\ 
&= (\Si^{(\bullet)} \cap \pow(V(\Si_{i,(\bullet)}))) \cup (\pow(V(\Si_{i,(\bullet)})) \cap \pow(V(\Si_{l,(\bullet)})))\\
& = (\Si^{(\bullet)} \cup \pow(V(\Si_{l,(\bullet)}))) \cap \pow(V(\Si_{i,(\bullet)}))\\
&= \Si^{(\bullet, l)} \cap \pow(V(\Si_{i,(\bullet)})) = \Si^{(l)} \cap \pow(V(\Si_{i,(\bullet,l)})).
\end{align*}
Thus $\Si_{i,(\bullet,l)}$ is induced. The same argument holds for the $I'$ case. 

For the fourth claim, let $\s \subseteq I$, and suppose $\cap_{i \in \s} \Si_{i,(\bullet,l)}$ is nonempty. By the previous claim, each $\Si_{i,(\bullet,l)}$ is induced. Thus we write:
\begin{align*}
\cap_{i\in \s}\Si_{i,(\bullet,l)} &= \Si^{(\bullet,l)} \cap \pow(\cap_{i \in \s}V(\Si_{i,(\bullet,l)})) \\
&= \Si^{(\bullet,l)} \cap \pow(\cap_{i \in \s}V(\Si_{i,(\bullet)}))\\
&= (\Si^{(\bullet)} \cup \pow(V(\Si_{l,(\bullet)})))  \cap \pow(\cap_{i \in \s}V(\Si_{i,(\bullet)}))\\
&= (\cap_{i \in \s} (\Si^{(\bullet)} \cap \pow(V(\Si_{i,(\bullet)}))) )	 \cup \pow(\cap_{i \in \s}V(\Si_{i,(\bullet)}) \cap V(\Si_{l,(\bullet)}))\\
&= (\cap_{i \in \s} \Si_{i,(\bullet)} )	 \cup \pow(\cap_{i \in \s}V(\Si_{i,(\bullet)}) \cap V(\Si_{l,(\bullet)})).
\end{align*}
For convenience, define $A:=(\cap_{i \in \s} \Si_{i,(\bullet)} )$ and $B:= \pow(\cap_{i \in \s}V(\Si_{i,(\bullet)}) \cap V(\Si_{l,(\bullet)}))$. Then $|A|$ is contractible by  inductive assumption, and $|B|$ is a full simplex, hence contractible. Also, $A\cap B$ has the form
\begin{align*}
&(\cap_{i \in \s} (\Si^{(\bullet)} \cap \pow(V(\Si_{i,(\bullet)}))) )	 \cap \pow(\cap_{i \in \s}V(\Si_{i,(\bullet)}) \cap V(\Si_{l,(\bullet)})) \\ 
=& \Si^{(\bullet)} \cap \pow(\cap_{i \in \s}V(\Si_{i,(\bullet)}) \cap V(\Si_{l,(\bullet)})) \\
=& \cap_{i \in \s}\Si_{i,(\bullet)} \cap \Si_{l,(\bullet)},
\end{align*}
and the latter is contractible by inductive assumption. Thus by Lemma \ref{lem:glue}, we have $|A\cup B|$ contractible. This proves the claim for the case $\s \subseteq I$. The case $\t \subseteq I'$ is similar.

Now we proceed to the final claim.
Since $\Si_{l,(\bullet)}$ is induced, we have $\Si_{l,(\bullet)} = \Si^{(\bullet)} \cap \pow(V(\Si_{l,(\bullet)}))$. By the contractibility assumption, we know that $|\Si_{l,(\bullet)}|$ is contractible. 
Also we know that $|\Si'_{l,(\bullet)}| = |\Si'^{(\bullet)} \cap \pow(V(\Si'_{l,(\bullet)}))|$ is contractible. By assumption we also have $V(\Si_{l,(\bullet)}) \subseteq V(\Si'_{l,(\bullet)})$. Thus by Lemma \ref{lem:glue-func}, we obtain homotopy equivalences $\Phi_l : |\Si^{(\bullet,l)}| \r |\Si^{(\bullet)}|$ and $\Phi'_l : |\Si'^{(\bullet,l)}| \r |\Si'^{(\bullet)}|$ such that $\Phi'_l$ extends $\Phi_l$.
Furthermore, the homotopy inverses of $\Phi_l$ and $\Phi'_l$ are just the inclusions $|\Si^{(\bullet)}| \hr |\Si^{(\bullet,l)}|$ and $|\Si'^{(\bullet)}| \hr |\Si'^{(\bullet,l)}|$.

Now let $\iota: |\Si^{(\bullet)}| \r |\Si'^{(\bullet)}|$ and $\iota_l: |\Si^{(\bullet,l)}| \r |\Si'^{(\bullet,l)}|$ denote the canonical inclusions. We wish to show the equality $\Phi'_l\circ \iota_l = \iota\circ \Phi_l$. 
Let $x \in|\Si^{(\bullet,l)}|$. Because $\Phi'_l$ extends $\Phi_l$ (this is why we needed the \emph{functorial} gluing lemma), we have
\[\Phi'_l(\iota_l(x))) = \Phi'_l(x) = \Phi_l(x) = \iota(\Phi_l(x)).\]
Since $x \in|\Si^{(\bullet,l)}|$ was arbitrary, the equality follows immediately. By the inductive assumption, we already have homotopy equivalences $|\Si^{(\bullet)}| \r |\Si|$ and $|\Si'^{(\bullet)}| \r |\Si'|$ that commute with the canonical inclusions. Composing these maps with $\Phi_l$ and $\Phi'_l$ completes the proof of the claim.\end{subproof}

By the preceding work, we replace the subcomplexes $\Si_l, \Si'_l$ by full simplices of the form $\Si_{l,(\bullet,l)},\Si_{l,(\bullet,l)}'$. In this process, the nerves remain unchanged and the complexes $\Si, \Si'$ are replaced by homotopy equivalent complexes $\Si^{(\bullet,l)}, \Si'^{(\bullet,l)}$. Furthermore, this process is functorial---the homotopy equivalences commute with the canonical inclusions $\Si \hr \Si^{(\bullet,l)}$ and $\Si' \hr \Si'^{(\bullet,l)}$.

Repeating the inductive process in Claim \ref{cl:nerve-cover-of-simplices} for all the finitely many $l \in I$ yields a simplicial complex $\Si^{(I)}$ along with a cover of simplices $\mc{A}_{\Si,(I)}$. 
We also perform the same procedure for all $l \in I' \setminus I$ (this does not affect $\Si^{(I)}$) to obtain a simplicial complex $\Si'^{(I')}$ along with a cover of simplices $\mc{A}_{\Si',(I')}$. 
Furthermore, $\Si^{(I)}$ and $\Si'^{(I)}$ are related to $\Si$ and $\Si'$ by a finite sequence of homotopy equivalences that commute with the canonical inclusions. Also, we have $\mc{N}(\mc{A}_{\Si}) = \mc{N}(\mc{A}_{\Si,(I)})$ and $\mc{N}(\mc{A}_{\Si'}) = \mc{N}(\mc{A}_{\Si',(I')})$. Thus we obtain the following picture: 

\begin{center}
\begin{tikzpicture}[]
\node (SI') at (3,0){$|\Si'^{(I')}|$};
\node (00) at (0,0){$\cdots$};
\node (1) at (-3,0){$|\Si'^{(l_1)}|$};
\node (2) at (-6,0){$|\Si'|$};
\node (3) at (9,0){$|\mc{N}(\mc{A}_{\Si'})|$};
\node (01) at (6,0){$|\mc{N}(\mc{A}_{\Si',(I')})|$};

\node (SI) at (3,2){$|\Si^{(I)}|$};
\node (02) at (0,2){$\cdots$};
\node (4) at (-3,2){$|\Si^{(l_1)}|$};
\node (5) at (-6,2){$|\Si|$};
\node (6) at (9,2){$|\mc{N}(\mc{A}_{\Si})|$};
\node (03) at (6,2){$|\mc{N}(\mc{A}_{\Si,(I)})|$};

\draw (SI) edge[->]node[above]{$\simeq$}  (02);
\draw (SI') edge[->]node[above]{$\simeq$}  (00);
\draw (00) edge[->] (1);
\draw (3) edge[double distance=2pt] (01);
\draw (02) edge[->] (4);
\draw (6) edge[double distance=2pt] (03);

\draw (1) edge[->]node[above]{$\simeq$}  (2);
\draw (00) edge[->]node[above]{$\simeq$}  (1);
\draw (4) edge[->]node[above]{$\simeq$}  (5);
\draw (02) edge[->]node[above]{$\simeq$}  (4);
\draw (1) edge[<-] node[right]{$\iota_{(l_1)}$} (4);
\draw (2) edge[<-] node[right]{$\iota$} (5);
\draw (3) edge[<-] node[right]{$\iota_{\mc{N}}$} (6);
\draw (01) edge[<-] node[right]{$\iota_{\mc{N},(I)}$} (03);
\draw (SI) edge[->]node[right]{$\iota_{(I)}$} (SI');
\draw (SI) edge[->]node[above]{$\simeq$} (03);
\draw (SI') edge[->]node[above]{$\simeq$} (01);
\end{tikzpicture}
\end{center}

By applying Theorem \ref{thm:nerve-functorial-I} to the block consisting of $|\Si^{(I)}|$, $|\Si'^{(I')}|$, $|\mc{N}(\mc{A}_{\Si,I})|$ and $|\mc{N}(\mc{A}_{\Si',I'})|$, we obtain a square that commutes up to homotopy. Then by composing the homotopy equivalences constructed above, we obtain a square consisting of $|\Si|$, $|\Si'|$, $|\mc{N}(\mc{A}_{\Si})|$, and $|\mc{N}(\mc{A}_{\Si'})|$ that commutes up to homotopy. Thus we obtain homotopy equivalences $|\Si| \simeq |\mc{N}(\mc{A}_\Si)|$ and $|\Si'| \simeq |\mc{N}(\mc{A}_{\Si'})|$ via maps that commute up to homotopy with the canonical inclusions. \end{proof}

\section{Dowker persistence diagrams and asymmetry}\label{sec:symmetry}

From the very definition of the Rips complex at any given resolution, one can see that the Rips complex is blind to asymmetry in the input data (Remark \ref{rem:rips-symm}). In this section, we argue that either of the Dowker source and sink complexes is sensitive to asymmetry. Thus when analyzing datasets containing asymmetric information, one may wish to use the Dowker filtration instead of the Rips filtration. In particular, this property suggests that the Dowker persistence diagram is a stronger invariant for directed networks than the Rips persistence diagram. 

In this section, we consider a family of examples, called \emph{cycle networks}, for which the Dowker persistence diagrams capture
meaningful structure, whereas the Rips persistence diagrams do not.

We then probe the question ``What happens to the Dowker or Rips persistence diagram of a network upon reversal of one (or more) edges?" Intuitively, if either of these persistence diagrams captures asymmetry, we would see a change in the diagram after applying this reversal operation to an edge.

\subsection{Cycle networks}\label{sec:cycle}

For each $n\in \N$, let $(X_n,E_n,W_{E_n})$ denote the weighted graph with vertex set $X_n:=\set{x_1,x_2,\ldots,x_n}$, edge set $E_n:=\set{(x_1,x_2),(x_2,x_3),\ldots,(x_{n-1},x_n),(x_n,x_1)}$, and edge weights $W_{E_n}:E_n \r \R$ given by writing $W_{E_n}(e)=1$ for each $e\in E_n$. Next let  $\w_{G_n}:X_n\times X_n \r \R$ denote the shortest path distance induced on $X_n\times X_n$ by $W_{E_n}$. Then we write $G_n:=(X_n,\w_{G_n})$ to denote the network with node set $X_n$ and weights given by $\w_{G_n}$. Note that $\w_{G_n}(x,x)=0$ for each $x\in X_n$. See Figure \ref{fig:cycle} for some examples.

We say that $G_n$ is the \emph{cycle network of length n}. One can interpret cycle networks as being highly asymmetric, because for every consecutive pair of nodes $(x_i,x_{i+1})$ in a graph $G_n$, where $1\leq i\mod(n)\leq n$, we have $\w_{G_n}(x_i,x_{i+1})=1$, whereas $\w_{G_n}(x_{i+1},x_i)=\diam(G_n) = n-1$, which is much larger than 1 when $n$ is large. 

To provide further evidence that Dowker persistence is sensitive to asymmetry, we computed both the Rips and Dowker persistence diagrams, in dimensions 0 and 1, of cycle networks $G_n$, for values of $n$ between 3 and 6. Computations were carried out using \texttt{Javaplex} in Matlab with $\Z_2$ coefficients. The results are presented in Figure \ref{fig:cycle}. Based on our computations, we were able to conjecture and prove the result in Theorem \ref{thm:cycleH1}, which gives a precise characterization of the 1-dimensional Dowker persistence diagram of a cycle network $G_n$, for any $n$. Furthermore, the 1-dimensional Dowker persistence barcode for any $G_n$ contains only one persistent interval, which agrees with our intuition that there is only one nontrivial loop in $G_n$. On the other hand, for large $n$, the 1-dimensional Rips persistence barcodes contain more than one persistent interval. This can be seen in the Rips persistence barcode of $G_6$, presented in Figure \ref{fig:cycle}. Moreover, for $n=3,4$, the 1-dimensional Rips persistence barcode does not contain any persistent interval at all. This suggests that Dowker persistence diagrams/barcodes are an appropriate method for analyzing cycle networks, and perhaps asymmetric networks in general.

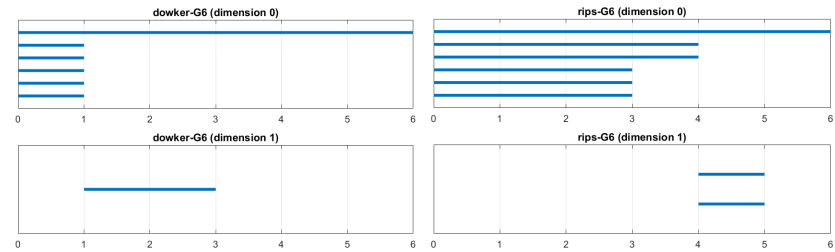
\begin{figure}
\centering
\begin{subfigure}{0.3\linewidth}
\centering
\begin{tikzpicture}
\tikzset{>=latex}
\node[circle,draw](1) at (0,1.5){$x_1$};
\node[circle,draw](2) at (-1,0){$x_2$};
\node[circle,draw](3) at (1,0){$x_3$};
\node (4) at (0,-1){};

\path[->] (1) edge [bend right] node[above,pos=0.5]{$1$} (2);
\path[->] (2) edge [bend right] node[above,pos=0.5]{$1$} (3);
\path[->] (3) edge [bend right] node[above,pos=0.5]{$1$} (1);
\end{tikzpicture}
\caption*{The cycle network $G_3$}
\end{subfigure}
\begin{subfigure}{0.33\linewidth}
\centering
\includegraphics[width=\linewidth]{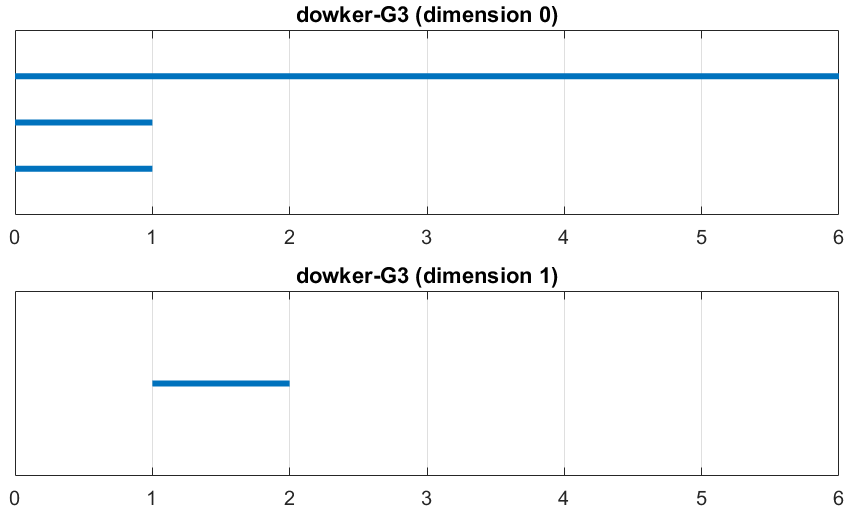}
\end{subfigure}
\begin{subfigure}{0.33\linewidth}
\centering
\includegraphics[width=\linewidth]{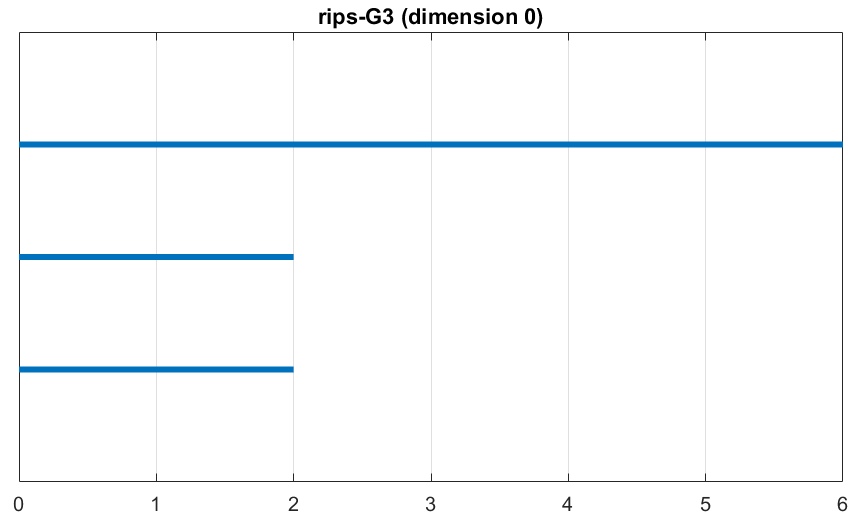}
\end{subfigure}
\hfill

\bigskip

\begin{subfigure}{0.3\linewidth}
\centering
\begin{tikzpicture}
\tikzset{>=latex}
\node[circle,draw](1) at (0,1.5){$x_1$};
\node[circle,draw](2) at (-1.5,0){$x_2$};
\node[circle,draw](3) at (0,-1.5,0){$x_3$};
\node[circle,draw](4) at (1.5,0){$x_4$};
\node (5) at (0,-2){};

\path[->] (1) edge [bend right] node[above,pos=0.5]{$1$} (2);
\path[->] (2) edge [bend right] node[above,pos=0.5]{$1$} (3);
\path[->] (3) edge [bend right] node[above,pos=0.5]{$1$} (4);
\path[->] (4) edge [bend right] node[above,pos=0.5]{$1$} (1);
\end{tikzpicture}
\caption*{The cycle network $G_4$}
\end{subfigure}
\begin{subfigure}{0.33\linewidth}
\centering
\includegraphics[width=\linewidth]{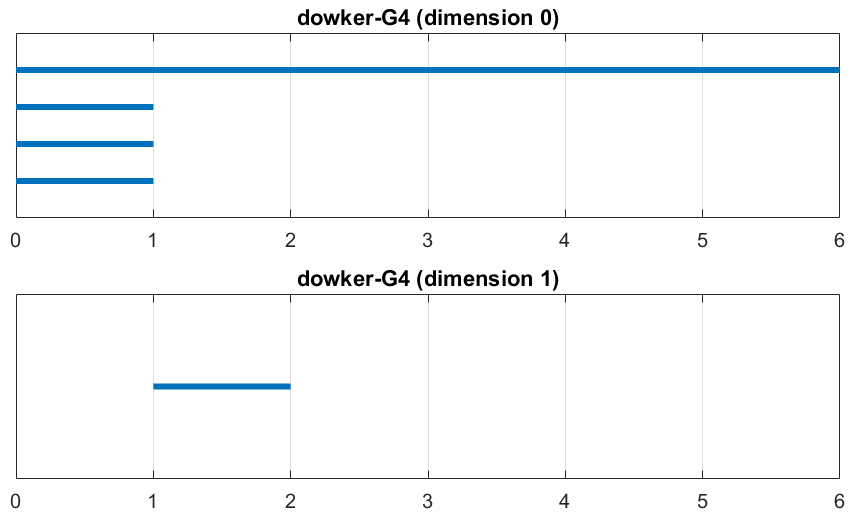}
\end{subfigure}
\begin{subfigure}{0.33\linewidth}
\centering
\includegraphics[width=\linewidth]{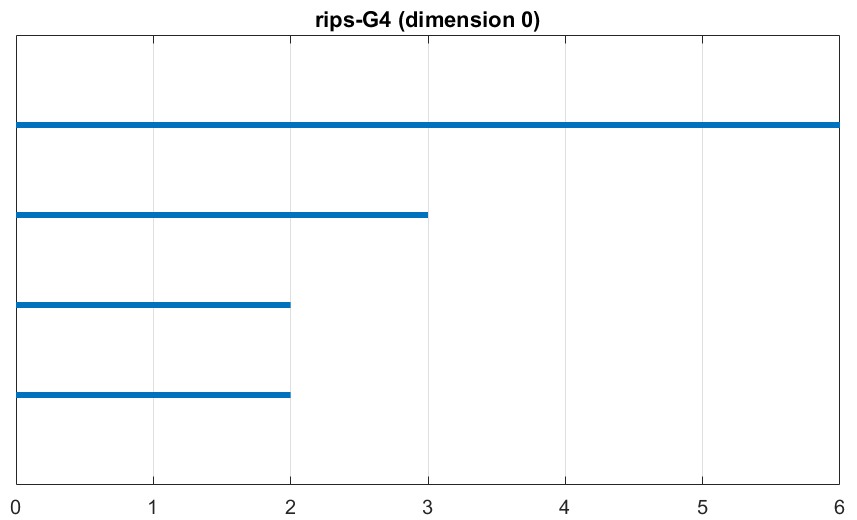}
\end{subfigure}\hfill

\bigskip

\begin{subfigure}{0.3\linewidth}
\centering
\begin{tikzpicture}
\tikzset{>=latex}
\node[circle,draw](1) at (0,1.5){$x_1$};
\node[circle,draw](2) at (1.5,0){$x_2$};
\node[circle,draw](3) at (1,-1.75){$x_3$};
\node[circle,draw](4) at (-1,-1.75){$x_4$};
\node[circle,draw](5) at (-1.5,0){$x_5$};
\node (6) at (0,-2){};

\path[->] (1) edge node[above,pos=0.5]{$1$} (2);
\path[->] (2) edge node[left,pos=0.5]{$1$} (3);
\path[->] (3) edge node[above,pos=0.5]{$1$} (4);
\path[->] (4) edge node[right,pos=0.5]{$1$} (5);
\path[->] (5) edge node[above,pos=0.5]{$1$} (1);
\end{tikzpicture}
\caption*{The cycle network $G_5$}
\end{subfigure}
\begin{subfigure}{0.33\linewidth}
\centering
\includegraphics[width=\linewidth]{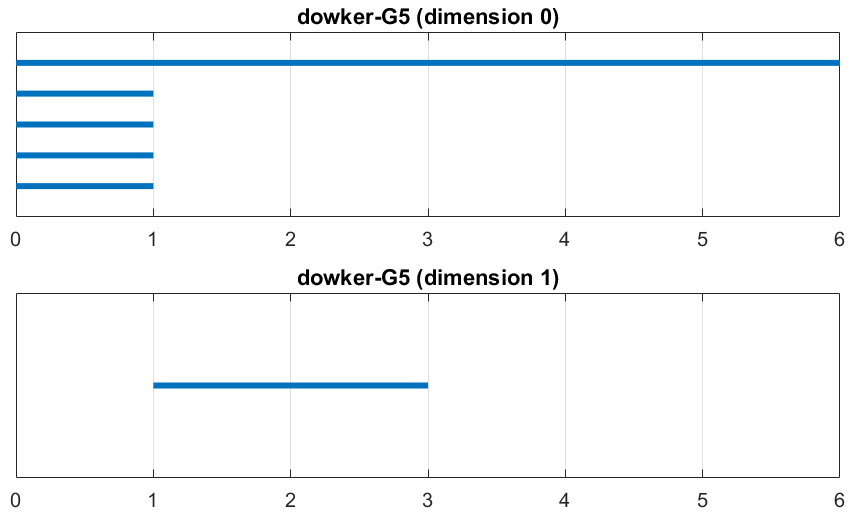}
\end{subfigure}
\begin{subfigure}{0.33\linewidth}
\centering
\includegraphics[width=\linewidth]{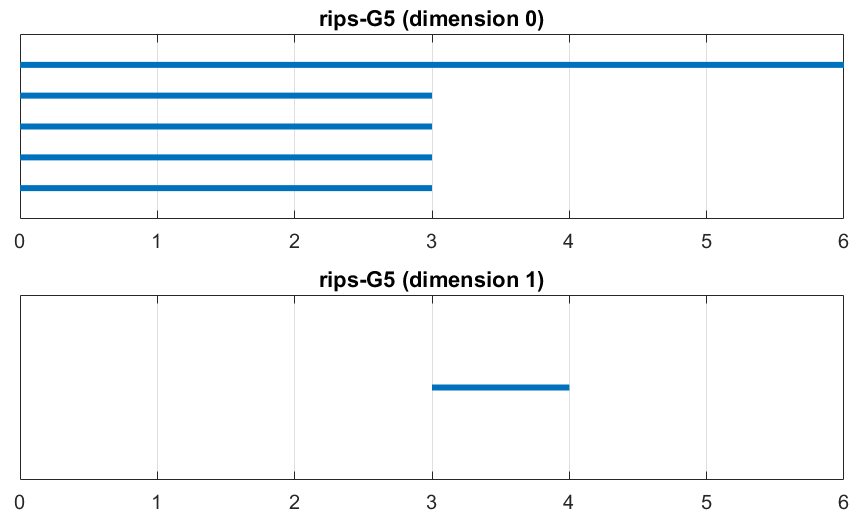}
\end{subfigure}\hfill

\bigskip

\begin{subfigure}{0.3\linewidth}
\centering
\begin{tikzpicture}
\tikzset{>=latex}
\node[circle,draw](1) at (-1,1.5){$x_1$};
\node[circle,draw](2) at (1,1.5){$x_2$};
\node[circle,draw](3) at (1.5,0){$x_3$};
\node[circle,draw](4) at (1,-1.5){$x_4$};
\node[circle,draw](5) at (-1,-1.5){$x_5$};
\node[circle,draw](6) at (-1.5,0){$x_6$};
\node (7) at (0,-2){};

\path[->] (1) edge node[above,pos=0.5]{$1$} (2);
\path[->] (2) edge node[left,pos=0.5]{$1$} (3);
\path[->] (3) edge node[left,pos=0.5]{$1$} (4);
\path[->] (4) edge node[above,pos=0.5]{$1$} (5);
\path[->] (5) edge node[right,pos=0.5]{$1$} (6);
\path[->] (6) edge node[right,pos=0.5]{$1$} (1);

\end{tikzpicture}
\caption*{The cycle network $G_6$}
\end{subfigure}
\begin{subfigure}{0.33\linewidth}
\centering
\includegraphics[width=\linewidth]{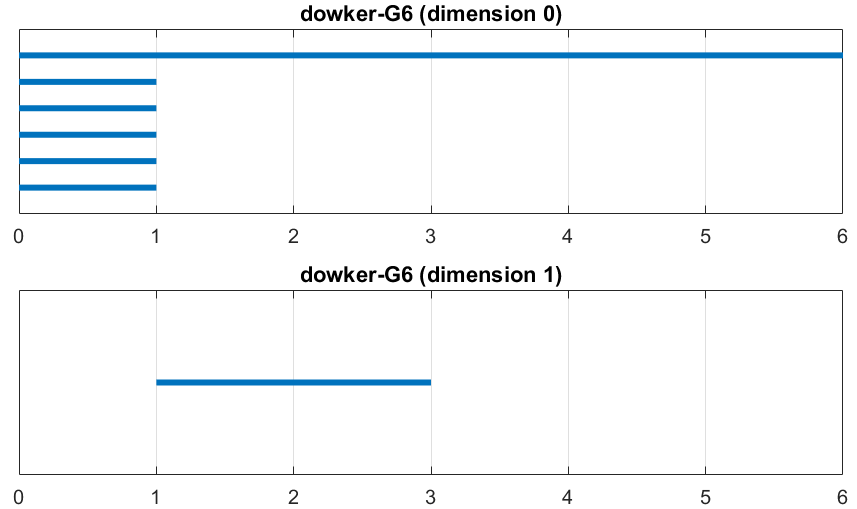}
\end{subfigure}
\begin{subfigure}{0.33\linewidth}
\centering
\includegraphics[width=\linewidth]{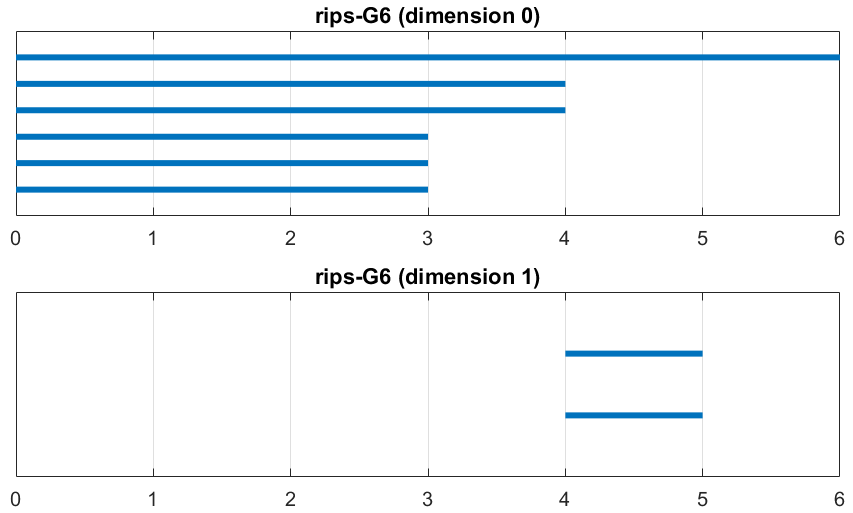}
\end{subfigure}\hfill

\caption{The first column contains illustrations of cycle networks $G_3,G_4,G_5$ and $G_6$. The second column contains the corresponding Dowker persistence barcodes, in dimensions 0 and 1. Note that the persistent intervals in the 1-dimensional barcodes agree with the result in Theorem \ref{thm:cycleH1}. The third column contains the Rips persistence barcodes of each of the cycle networks. Note that for $n=3,4$, there are no persistent intervals in dimension 1. On the other hand, for $n=6$, there are two persistent intervals in dimension 1. }
\label{fig:cycle}
\end{figure}

\medskip
\noindent
\textbf{Notation.} In the remainder of this section, we will prove results involving Dowker sink complexes of the cycle networks $G_n$ and associated vector spaces at a range of resolutions $\d$. For convenience, we will write $\sink_\d := \sink_{\d,G_n}$ (where $n$ will be fixed) and $C_k^\d := C_k(\sink_\d)$, the $k$-chain vector space associated to $\sink_\d$ for each $k\in \Z_+$. For each $k\in \Z_+$, the boundary map from $C_k^\d$ to $C_{k-1}^\d$ will be denoted $\p_k^\d$. Whenever we write $x_i$ to denote a vertex of $G_n$, the subscript $i$ should be understood as $i\Mod{n}$. We write $e_i$ to denote the 1-simplex $[x_i,x_{i+1}]$ for each $1\leq i \leq n$, where $x_{n+1}$ is understood to be $x_1$. Given an element $\g \in \ker(\p_k^\d)\subseteq C_1^\d$, we will write $\la \g \ra_\d$ to denote its equivalence class in the quotient vector space $\ker(\p_k^\d)/\im(\p^\d_k)$. We will refer to the operation of taking this quotient as \emph{passing to homology}.

The following theorem contains the characterization result for 1-dimensional Dowker persistence diagrams of cycle networks.

\begin{theorem}\label{thm:cycleH1} Let $G_n=(X_n,\w_{G_n})$ be a cycle network for some $n\in \N$, $n\geq 3$. 
Then we obtain:
\[\dgm^{\mf{D}}_1(G_n)=\set{(1,\ceil{n/2})\in \R^2}.\]
Thus $\dgm^{\mf{D}}_1(G_n)$ consists of precisely the point $(1,\ceil{n/2})\in \R^2$ with multiplicity 1.

\end{theorem}

\begin{proof}[Proof of Theorem \ref{thm:cycleH1}]

The proof occurs in three stages: first we show that a 1-cycle appears at $\d=1$, next we show that this 1-cycle does not become a boundary until $\d=\ceil{n/2}$, and finally that any other 1-cycle belongs to the same equivalence class upon passing to homology (this shows that the single point in the persistence diagram has multiplicity 1).

Note that for $\d<1$, there are no 1-simplices in $\sink_\d$, and so $H_1(\sink_\d)$ is trivial. Suppose $1\leq \d < 2$. 

\begin{claim}\label{cl:cycle-2simp} There are no 2-simplices in $\sink_\d$ for $1\leq \d < 2$.
\end{claim}
\begin{subproof} To see this, let $x_i,x_j,x_k$ be any three distinct vertices in $X_n$. Assume towards a contradiction that there exists $x\in X_n$ such that $(x_i,x),(x_j,x),(x_k,x)\in R_{\d,X_n},$ where $R_{\d,X_n}$ is as given by Equation \ref{eq:relation}. Thus $\w_{G_n}(x_i,x)\in \{0,1\}$, so either $x=x_i$ or $x=x_{i+1}$. Similarly we get that $x=x_j$ or $x=x_{j+1}$, and that $x=x_k$ or $x=x_{k+1}$. But this is a contradiction, since $x_i,x_j,x_k$ are all distinct. \end{subproof}

By the claim, there are no 2-simplices in $\sink_\d$, so $\im(\p^\d_2)$ is trivial and the only 1-chains are linear combinations of $e_i$ terms. Next, we define: 
\[v_n:=e_1 + e_2 + \ldots + e_n=[x_1,x_2]+[x_2,x_3]+ \ldots + [x_n,x_1].\]
Note that $v_n\in C_1^\d$ for all $\d \geq 1$. One can further verify that $\p^\d_1(v_n) = 0$, for any $\d \geq 1$. In other words, $v_n$ is a 1-cycle for any $\d \geq 1$. 

\begin{claim}\label{cl:cycle-v} Let $1\leq \d < 2$. Then $v_n$ generates $\ker(\p^\d_1)\subseteq C_1^\d$.
\end{claim}

\begin{subproof} The only 1-simplices in $\sink_\d$ are of the form $e_i$, for $1\leq i \leq n$. So it suffices to show that any linear combination of the $e_i$ terms is a multiple of $v_n$. Let $u = \sum_{i=1}^na_ie_i \in \ker(\p^\d_1)$, for some $a_1,\ldots, a_n \in \mathbb{K}$. Then,
\begin{align*}
0=\p^\d_1(u)=\sum_{i=1}^na_i\p^\d_1(e_i) &= \sum_{i=1}^na_i([x_{i+1}]-[x_i])\\
&=\sum_{i=1}^n(a_{i-1}-a_i)[x_i], &&\text{where $x_0$ is understood to be $x_n$.}
\end{align*}
Since all the $[x_i]$ are linearly independent, it follows that $a_1=a_2=\ldots=a_n$. Thus it follows that $u$ is a constant multiple of $v_n$. This proves the claim. \end{subproof}

By the two preceding claims, it follows that $\set{\la v_n \ra_\d}$ is a basis for $H_1(\sink_\d)$, for $\d \in [1,2)$. More specifically, $\la v_n \ra_\d$ is a cycle that appears at $\d=1$ and does not become a boundary until at least $\d=2$, and any other cycle in $C_1^\d$, for $\d \in [1,2)$, is in the linear span of $v_n$. Next, suppose $\d \geq 2$. Note that this allows the appearance of cycles that are not in the span of $v_n$. In the next claim, we show that upon passing to homology, the equivalence class of any such cycle coincides with that of $v_n$. This will show that there can be at most one nontrivial element in $\dgm_1^{\si}(G_n)$. 

\begin{claim}\label{cl:cycle-decomp} Let $\d\geq 2$, and let $y = \sum_{i=1}^pa_i\s_i \in \ker(\p^\d_1)$ for some $p\in \N$, some $a_1,\ldots,a_p \in \mathbb{K}$, and some $\s_1,\ldots, \s_p\in \sink_\d$. Then there exists a choice of coefficients $(b_i)_{i=1}^n \in \mathbb{K}^n$ such that $z=\sum_{i=1}^nb_ie_i \in \ker(\p^\d_1)$ and $y-z \in \im (\p^\d_2)$. Moreover, we obtain $\la y \ra_\d = \la z \ra_\d = \la v_n \ra_\d$ upon passing to homology.
\end{claim}

\begin{figure}
\includegraphics[scale=0.5]{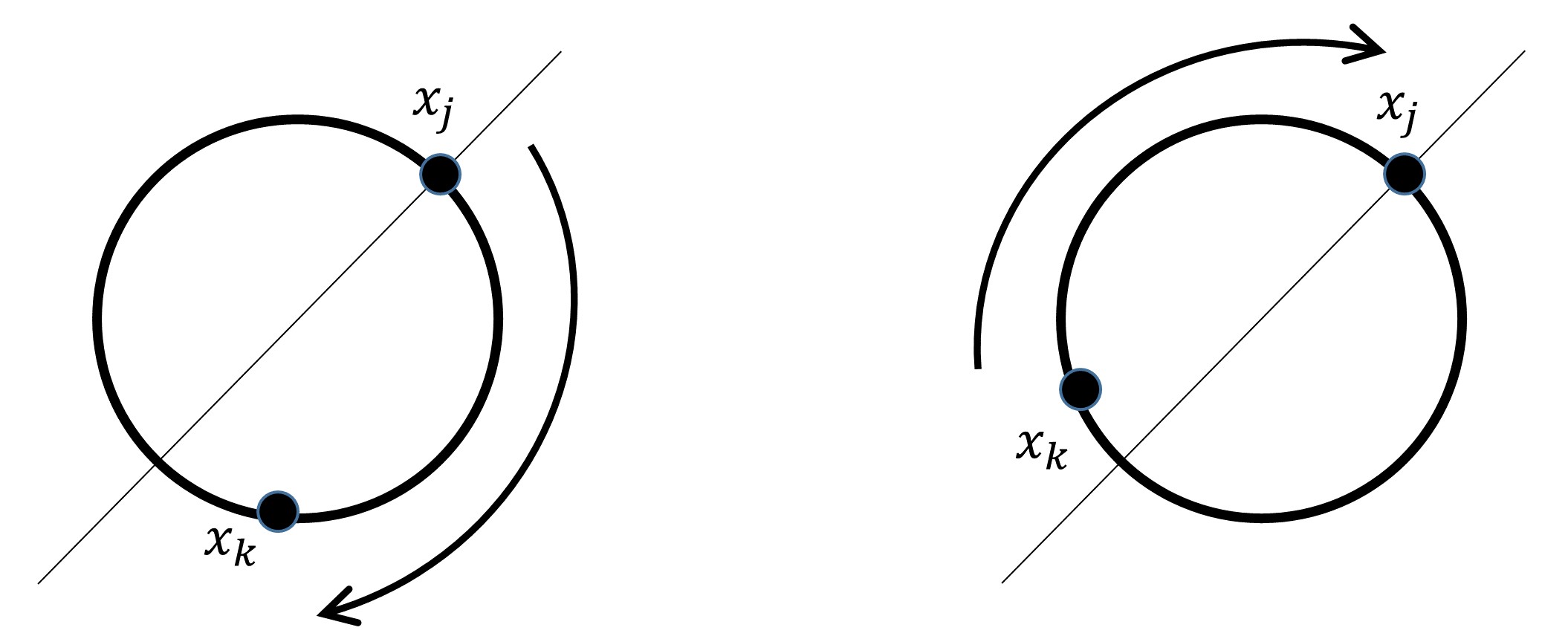}
\caption{Given two points $x_j,x_k \in X_n$, we have either $\w_{G_n}(x_j,x_k)\leq n/2$, or $\w_{G_n}(x_k,x_j)< n/2$. To see this, note that $\w_{G_n}(x,x')+\w_{G_n}(x',x)=n$ for any $x\neq x'\in X_n$.}
\label{fig:dowker-cycle-prop}
\end{figure}

\begin{figure}
\includegraphics[scale=0.5]{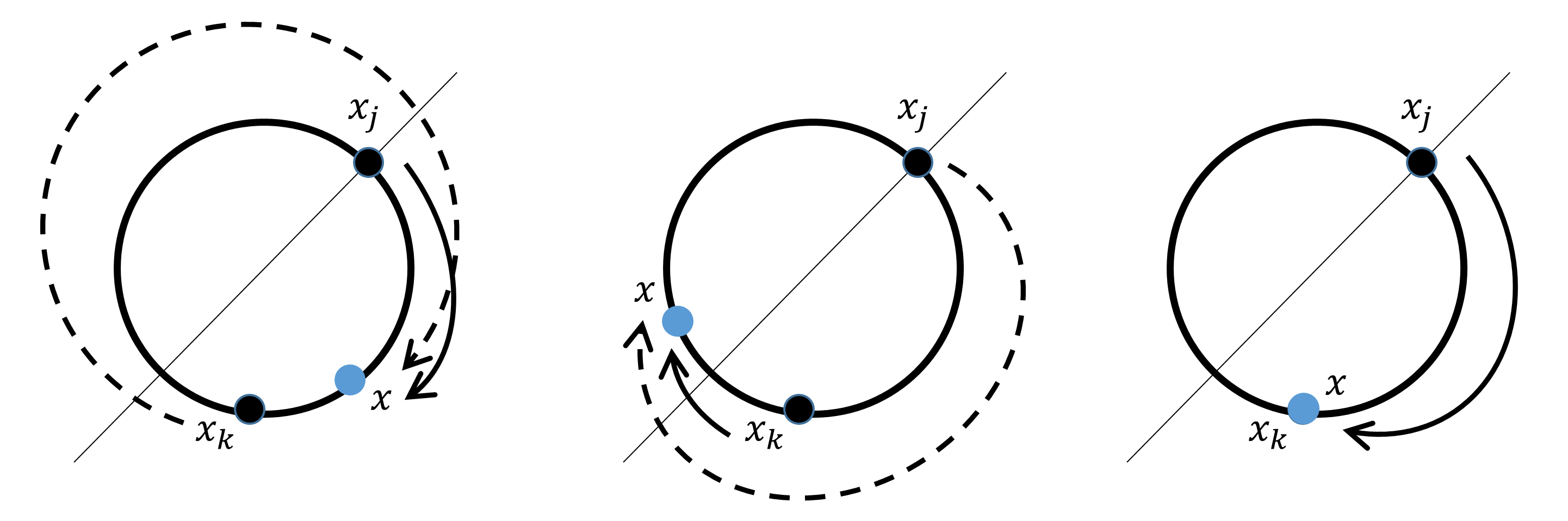}
\caption{Three possible locations for a $\d$-sink $x$ of a simplex $[x_j,x_k]$, assuming that $\w_{G_n}(x_j,x_k)\leq n/2$. For the figure on the left, note that $\w_{G_n}(x_k,x)\geq n/2 \geq \w_{G_n}(x_j,x_k)$. For the figure in the middle, note that $\w_{G_n}(x_j,x)\geq \w_{G_n}(x_j,x_k)$. Finally, for the figure on the right, where $x=x_k$, note that $\w_{G_n}(x_j,x)=\w_{G_n}(x_j,x_k)$ and $\w_{G_n}(x_k,x)=0$.}
\label{fig:dowker-cycle-prop-sink}
\end{figure}

\begin{subproof}
To see this, fix $\s_i \in \sink_\d$, and write $\s_i= [x_j,x_k]$ for some $1\leq j,k \leq n$. If $k=j+1$ (resp. $k=j-1$), then we already have $\s_i=e_j$ (resp. $\s_i=e_k$), so there is nothing more to show. Assume $k\not\in \{j+1,j-1\}$. Since $\w_{G_n}(x_j,x_k) + \w_{G_n}(x_k,x_j) = n$, we have two cases: (1) $\w_{G_n}(x_j,x_k)\leq n/2$, or (2) $\w_{G_n}(x_k,x_j)< n/2$. In the first case, we have $k=j+l $ for some integer $l\in [2,n/2]$ (all numbers are taken modulo $n$). In the second case, $j = k+l$ for some integer $l\in [2,n/2)$ (also modulo $n$). The situation is illustrated in Figure \ref{fig:dowker-cycle-prop}. Both cases are similar, so we only prove the case $\w_{G_n}(x_j,x_k) \leq n/2$. 

Recall that any $\d$-sink $x\in X_n$ for $[x_j, x_k]$ satisfies $\max(\w_{G_n}(x_j,x), \w_{G_n}(x_k,x)) \leq \d$, by the $\d$-sink condition (Equation \ref{eq:d-sink}). Also note that such a $\d$-sink $x$ satisfies 
\[\max(\w_{G_n}(x_j,x),\w_{G_n}(x_k,x)) \geq \w_{G_n}(x_j,x_k),\]
as can be seen from Figure \ref{fig:dowker-cycle-prop-sink}. So whenever some $x\in X_n$ is a $\d$-sink for $[x_j,x_k]$, we have $x_k$ as a valid $\d$-sink for $[x_j,x_k]$. Since $[x_j,x_k] \in \sink_\d$, it must have a $\d$-sink $x\in X_n$. Thus $x_k$ is a valid $\d$-sink for $[x_j,x_k]$. Next let $l\in [2,n/2]$ be an integer such that $k=j+l$ (modulo $n$). Notice that:
\[0=\w_{G_n}(x_k,x_k)=\w_{G_n}(x_{j+l},x_k)< \w_{G_n}(x_{j+l-1},x_k)< \ldots < \w_{G_n}(x_{j+1},x_k)<\w_{G_n}(x_j,x_k) \leq \d.\]
Then observe that: 
\[[x_j,x_{j+1},x_k],[x_{j+1},x_{j+2},x_k],\ldots, [x_{k-2},x_{k-1},x_k] \in \sink_\d,\] 
since $x_k$ is a $\d$-sink for all these 2-simplices. One can then verify the following:
\begin{align*}
&\p^\d_2\left([x_j,x_{j+1},x_k]+[x_{j+1},x_{j+2},x_k]+\ldots+[x_{k-2},x_{k-1},x_k]\right)\\
&=\p_2^\d\left(\sum_{q=0}^{k-j-2}[x_{j+q},x_{j+q+1},x_k]\right)\\
&=\sum_{q=0}^{k-j-2}[x_{j+q+1},x_k]-\sum_{q=0}^{k-j-2}[x_{j+q},x_k]+\sum_{q=0}^{k-j-2}[x_{j+q},x_{j+q+1}]\\
&=\sum_{q=0}^{k-j-2}[x_{j+q+1},x_k] - [x_j,x_k] - \sum_{q=0}^{k-j-3}[x_{j+q+1},x_k] +\sum_{q=0}^{k-j-2}[x_{j+q},x_{j+q+1}]\\
&=[x_j,x_{j+1}]+[x_{j+1},x_{j+2}]+\ldots+[x_{k-1},x_k]-[x_j,x_k]\\
&=e_j+e_{j+1} + \ldots + e_{k-1} -\s_i.
\end{align*}
Thus $a_i(e_j+e_{j+1}+\ldots+e_{k-1})-a_i\s_i \in\im (\p^\d_2)$. Repeating this process for all $\s_i$, $i\in \{1,\ldots,p\}$, we may obtain the coefficients $(b_i)_{i=1}^n$ such that $\sum_{i=1}^pa_i\s_i - \sum_{i=1}^nb_ie_i \in \im(\p^\d_2).$ Let $z=\sum_{i=1}^nb_ie_i$. Then $y-z \in \im(\p^\d_2)$. Moreover, since $\p^\d_1\circ\p^\d_2 = 0$, it follows that $\p^\d_1(y)-\p^\d_1(z)=0$, so $z\in \ker(\p^\d_1)$. 

Finally, note that an argument analogous to that of Claim \ref{cl:cycle-v} shows that $b_1=b_2=\ldots=b_n$. Hence it follows that $z$ is a multiple of $v_n$. Thus $\la z \ra_\d = \la v_n \ra_\d$. This proves the claim.
\end{subproof}

By Claims \ref{cl:cycle-v} and \ref{cl:cycle-decomp}, it follows that $H_1(\sink_\d)$ is generated by $\la v_n \ra_\d$ for all $\d \geq 1$, so $\dim(H_1(\sink_\d))\leq 1$ for all $\d \geq 1$. It remains to show that $\la v_n \ra_\d$ does not become trivial until $\d=\ceil{n/2}$. 

The cases $n=3,4$ can now be completed quickly, so we focus on these simpler situations first. For either of $n=3,4$, we have $\ceil{n/2}=2.$ Suppose $\d=2$ and $n=3$. Then we have $[x_1,x_2,x_3]\in \sink_\d$ because $\diam(G_n)=2$ and any of $x_1,x_2,x_3$ can be a $2$-sink for $[x_1,x_2,x_3]$. Then,
\[\p_2^\d([x_1,x_2,x_3])=[x_2,x_3]-[x_1,x_3]+[x_1,x_2]=e_1+e_2+e_3 = v_3.\]
Recall that by Claim \ref{cl:cycle-2simp}, $v_3\not\in \im(\p_2^\d)$ for any $\d<2$. Thus by Claim \ref{cl:cycle-v} and the preceding equation, $v_3$ generates $\ker(\p_1^\d)$ for $1\leq \d <2$, and becomes a boundary for precisely $\d\geq 2$. Thus $\dgm_1^{\si}(G_3)=\set{(1,2)}.$ Next, suppose $\d=2$ and $n=4$. Then we have $[x_1,x_2,x_3],[x_1,x_3,x_4]\in \sink_\d$ with $x_3, x_1$ as $2$-sinks, respectively. By a direct computation, we then have:
\[\p_2^\d([x_1,x_2,x_3]+[x_1,x_3,x_4])=e_1+e_2+e_3+e_4 = v_4.\]
By following the same argument as for the case $n=3$, we see that $\dgm_1^{\si}(G_4)=\set{(1,2)}.$

In the sequel, we assume that $n > 4$. Recall that it remains to show that $\la v_n \ra_\d$ does not become trivial until $\d=\ceil{n/2}$, and that $\la v_n \ra_\d = 0$ for all $\d\geq \ceil{n/2}$. We have already shown that $\la v_n \ra_\d$ is not trivial for $\d \in [1,2)$. We proceed by defining the following: 
\[\g_n:=[x_1,x_2,x_3] + [x_1,x_3,x_4] + \ldots + [x_1,x_{n-1},x_n] = \sum_{i=1}^{n-2}[x_1,x_{i+1},x_{i+2}].\]

\begin{claim} For each $\d \geq \ceil{n/2}$, we have $\g_n \in C_2^\d$ and $\p_2^\d(\g_n)=v_n$. In particular, $\la v_n \ra_\d = 0$ for all such $\d$. 
\end{claim}

\begin{subproof} Let $\d \geq \ceil{n/2}$. Notice that 
\[\w_{G_n}(x_{\ceil{n/2}+1},x_1) = n - \w_{G_n}(x_1,x_{\ceil{n/2}+1}) = n- \ceil{n/2} \leq n/2 \leq \ceil{n/2} \leq \d,\]
so $\w_{G_n}(x_i,x_1) \leq \d$ for each $i \in \{\ceil{n/2}+1,\ceil{n/2}+2,\ldots, n\}$. Then for each $i\in \{\ceil{n/2}+1,\ceil{n/2}+2,\ldots, n-1\}$, we have $[x_i,x_{i+1},x_1] \in \sink_\d$, with $x_1$ as a $\d$-sink.

Also notice that for each $i\in \{1,\ldots,\ceil{n/2}\}$,
\[\w_{G_n}(x_i,x_{\ceil{n/2}+1})\leq\w_{G_n}(x_1,x_{\ceil{n/2}+1})=\ceil{n/2} \leq \d,\]
so $\w_{G_n}(x_i,x_{\ceil{n/2}+1}) \leq \d$. Thus
for any $i\in \{2,\ldots,\ceil{n/2}\},$ we have $[x_1,x_i,x_{i+1}] \in \sink_\d$, with $x_{\ceil{n/2}+1}$ as a $\d$-sink. 

Combining the two preceding observations, we see that for any $i\in \{2,\ldots, n-2\}$, we have $[x_1,x_{i+1},x_{i+2}]\in \sink_\d$. It follows that $\g_n \in C_2^\d$.

Next we observe the following:
\begin{align*}
\p_2^\d(\g_n)&=\p_2^\d\left(\sum_{i=1}^{n-2}[x_1,x_{i+1},x_{i+2}]\right)\\
&=\sum_{i=1}^{n-2}[x_{i+1},x_{i+2}] - \sum_{i=1}^{n-2}[x_1,x_{i+2}]+
\sum_{i=1}^{n-2}[x_1,x_{i+1}]\\
&=\sum_{i=1}^{n-2}[x_{i+1},x_{i+2}] - 
\sum_{i=1}^{n-2}[x_1,x_{i+2}]+[x_1,x_2]+
\sum_{i=2}^{n-2}[x_1,x_{i+1}]\\
&=\sum_{i=1}^{n-2}[x_{i+1},x_{i+2}] + [x_1,x_2] - [x_1,x_n] = v_n.
\end{align*}
It follows that for any $\d \geq \ceil{n/2}$, we have $v_n \in \im(\p_2^\d)$, and so $\la v_n \ra_\d =0$ for each such $\d$. \end{subproof}

\begin{figure}
\includegraphics[scale=0.5]{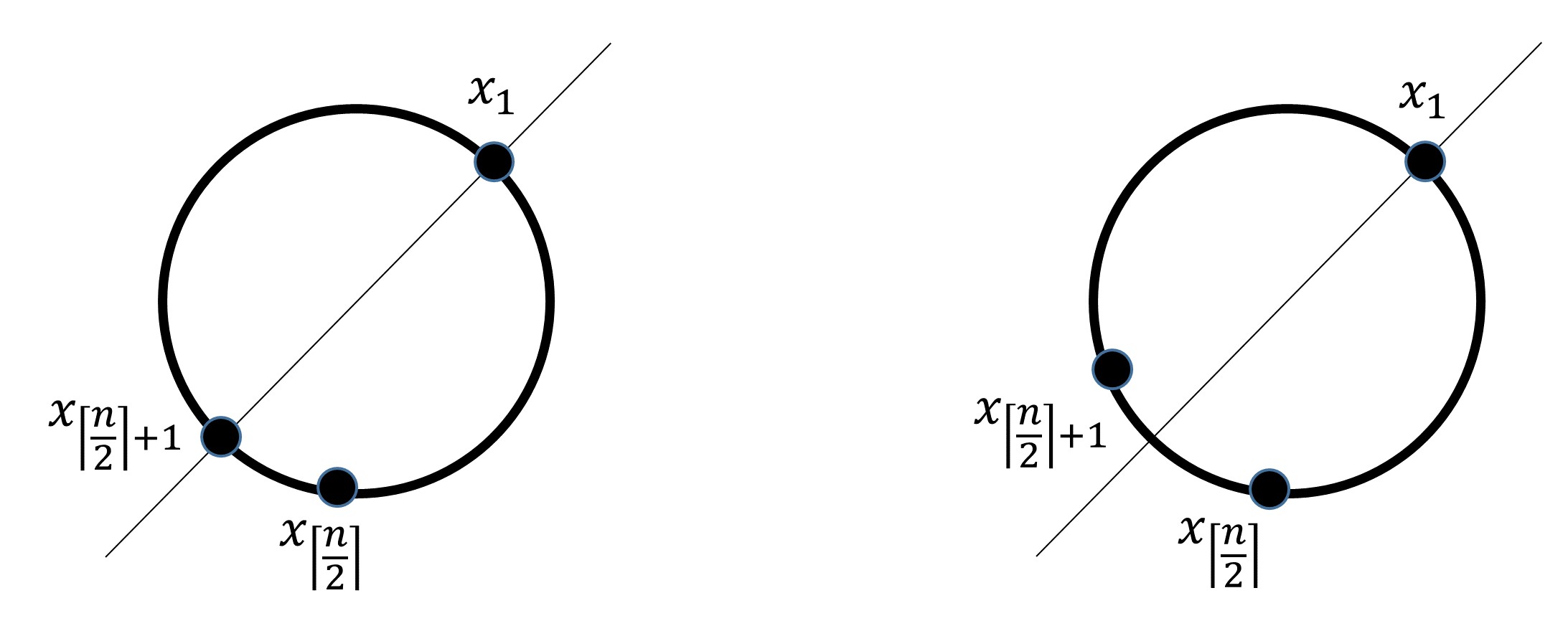}
\caption{Placement of $x_{\ceil{n/2}}$ and $x_{\ceil{n/2}-1}$, depending on whether $n$ is even or not.}
\label{fig:dowker-cycle-prop2}
\end{figure}

\begin{claim} There does not exist $\d \in [2,\ceil{n/2})$ such that $\la v_n \ra_\d$ is trivial. 
\end{claim} 
\begin{subproof} Let $2\leq \d< \ceil{n/2}$. As a first step, we wish to show that $\g_n \not\in C_2^\d$. For this step, it suffices to show that the 2-simplex $\s:= [x_1,x_{\ceil{n/2}},x_{\ceil{n/2}+1}]$ does not belong to $\sink_\d$. The placement of $x_{\ceil{n/2}}$ and $x_{\ceil{n/2}+1}$ is illustrated in Figure \ref{fig:dowker-cycle-prop2}. 

By an argument similar to that used in Figure \ref{fig:dowker-cycle-prop-sink}, one can verify that there exists a $\d$-sink for $\s$ if and only if at least one of $x_1,x_{\ceil{n/2}}, x_{\ceil{n/2}+1}$ is a $\d$-sink for $\s$. But note the following:
\begin{align*}
\w_{G_n}(x_{\ceil{n/2}},x_1)= n-(\ceil{n/2}-1) &= \begin{cases}
n/2 + 1 &: n \text{ even}\\
\ceil{n/2} &: n\text{ odd}
\end{cases}\\
&\geq \ceil{n/2} > \d,
\end{align*}
so $x_1$ cannot be a $\d$-sink for $\s$. Similarly we note that $\w_{G_n}(x_{\ceil{n/2}+1},x_{\ceil{n/2}}) = n > \d$ and $\w_{G_n}(x_1,x_{\ceil{n/2}+1})=\ceil{n/2} > \d,$ so neither $x_{\ceil{n/2}}$ nor $x_{\ceil{n/2}+1}$ can be $\d$-sinks for $\s$. Thus $\s \not \in \sink_\d$, and so $\g_n \not\in C_2^\d$.

Suppose there exists $\g' \in C_2^\d$ such that $\p_2^\d(\g') = v_n$. Since $[x_1,x_2]$ is a summand of $v_n$, we must have $a_j[x_1,x_2,x_j]$ as a summand of $\g'$, for some coefficient $a_j$ and some $3\leq j \leq n$. First suppose that $x_j$ is a sink for $[x_1,x_2,x_j]$. We claim that $\g'$ is homologous to a chain containing $[x_1,x_2,x_3]$ as a summand. If $j=3$, then we are done, so suppose $j > 3$. Then we also know that $[x_1,x_2,x_3,x_j]$ is a 3-simplex in $\sink_{\d}$. Let $\g''$ be the chain obtained from $\g'$ by replacing $[x_1,x_2,x_j]$ with $[x_1,x_2,x_3] - [x_2,x_3,x_j] + [x_1,x_3,x_j]$. Since $\p_3^\d([x_1,x_2,x_3,x_j]) = [x_2,x_3,x_j] - [x_1,x_3,x_j] + [x_1,x_2,x_j] -[x_1,x_2,x_3]$ and $\p_2^\d\circ \p_3^\d = 0$, we know that $\p_2^\d(\g'') = \p_2^\d(\g') = v_n$. 

Now $\p_2^\d([x_1,x_2,x_3])$ contributes an $[x_1,x_3]$ summand which does not appear in $v_n$, so it must be cancelled by some other terms in $\g'$ (resp. $\g''$). Thus there must exist another 2-simplex $[x_1,x_3,x_k]$ in $\sink_\d$, where $k\neq 2$. But we can repeat the preceding argument to obtain a chain homologous to $\g'$ containing both $[x_1,x_2,x_3]$ and $[x_1,x_3,x_4]$ as summands. Proceeding in this way, we obtain a chain homologous to $\g'$ that contains $[x_1,x_{\ceil{n/2}},x_{\ceil{n/2}+1}]$ as a summand. But this is a contradiction to what we have shown previously, i.e. that $[x_1,x_{\ceil{n/2}},x_{\ceil{n/2}+1}]$ is not a simplex in $\sink_\d$. 

In the case where $x_j$ is not a sink for $[x_1,x_2,x_j]$, we must have $x_2$ as a sink instead. Using similar reasoning as above, we can replace $\g'$ in this instance by a homologous chain containing $[x_n,x_1,x_2]$ as a summand. Since $[x_n,x_2]$ is not a summand of $v_n$, we can obtain another homologous chain containing $[x_{n-1},x_n,x_1]$ as a summand, then a homologous chain containing $[x_{n-1},x_n,x_1],[x_{n-2},x_{n-1},x_1]$ as summands, and so on until we again obtain a homologous chain containing $[x_1,x_{\ceil{n/2}},x_{\ceil{n/2}+1}]$ as a summand. Once again, this is a contradiction. This proves the claim. \end{subproof}

Thus we have shown that $v_n$ is a nontrivial cycle that appears at $\d=1$, and becomes a boundary at exactly $\d=\ceil{n/2}$. Furthermore, we have shown that upon passing to homology, the equivalence classes of all cycles coincide with that of $v_n$. Thus there is only one off-diagonal point $(1,\ceil{n/2})$ on the 1-dimensional persistence diagram, which appears with multiplicity one. This concludes the proof.\end{proof}

\begin{remark} From our experimental results (see Figure \ref{fig:cycle}), it appears that the 1-dimensional Rips persistence diagram of a cycle network does not admit a characterization as simple as that given by Theorem \ref{thm:cycleH1} for the 1-dimensional Dowker persistence diagram. Moreover, the Rips complexes $\mf{R}^\d_{G_n}, \d \in \R, n\in \N$ correspond to certain types of \emph{independence complexes} that appear independently in the literature, and whose homotopy types remain open \cite[Question 5.3]{engstrom2009complexes}. On a related note, we point the reader to \cite{adamaszek2015vietoris} for a complete characterization of the homotopy types of Rips complexes of points on the circle (equipped with the restriction of the arc length metric). 

To elaborate on the connection to \cite{engstrom2009complexes}, we write $H^k_n$ to denote the undirected graph with vertex set $\set{1,\ldots, n}$, and edges given by pairs $(i,j)$ where $1\leq i < j \leq n$ and either $j-i < k$ or $(n+i) - j < k$. Next we write $\ind(H^k_n)$ to denote the \emph{independence complex} of $H^k_n$, which is the simplicial complex consisting of subsets $\s \subseteq \set{1,2,\ldots, n}$ such that no two elements of $\s$ are connected by an edge in $H^k_n$. Then we have $\ind(H^k_n) = \mf{R}^{n-k}_{G_n}$ for each $k, n\in \N$ such that $k < n$. To gain intuition for this equality, fix a basepoint $1$, and consider the values of $j\in \N$ for which the simplex $[1,j]$ belongs to $\ind(H^k_n)$ and to $\mf{R}^{n-k}_{G_n}$, respectively. In either case, we have $k+1 \leq j \leq n-k+1$. Using the rotational symmetry of the points, one can then obtain the remaining 1-simplices. Rips complexes are determined by their 1-skeleton, so this suffices to construct $\mf{R}^{n-k}_{G_n}$. Analogously, $\ind(H^k_n)$ is determined by the edges in $H^k_n$, and hence also by its 1-skeleton. In \cite[Question 5.3]{engstrom2009complexes}, the author writes that the homotopy type of $\ind(H^k_n)$ is still unsolved. Characterizing the persistence diagrams $\dgm_k^{\mf{R}}(G_n)$ thus seems to be a useful future step, both in providing a computational suggestion for the homotopy type of $\ind(H^k_n)$, and also in providing a valuable example in the study of persistence of directed networks.

  \end{remark}

\begin{remark}
Theorem \ref{thm:cycleH1} has the following implication for data analysis: nontrivial 1-dimensional homology in the Dowker persistence diagram of an asymmetric network suggests the presence of directed cycles in the underlying data. Of course, it is not necessarily true that nontrivial 1-dimensional persistence can occur \emph{only} in the presence of a directed circle. 

\end{remark}

\begin{remark} Our motivation for studying cycle networks is that they constitute directed analogues of circles, and we were interested in seeing if the 1-dimensional Dowker persistence diagram would be able to capture this analogy. Theorem \ref{thm:cycleH1} shows that this is indeed the case: we get a single nontrivial 1-dimensional persistence interval, which is what we would expect when computing the persistent homology of a circle in the metric space setting. We further studied the 2-dimensional Dowker persistence diagrams of cycle networks. Our computational examples, some of which are illustrated in Figure \ref{fig:cycle-2dim}, enabled us to conjecture:

\begin{conjecture}\label{conj:dowker-dgm-2} Let $n\in \N$, $n \geq 3$, and let $G_n$ be a cycle network. If $n$ is odd, then $\dgm_2^{\mf{D}}(G_n)$ is trivial. If $n$ is even, then $\dgm_2^{\mf{D}}(G_n) = [(\frac{n}{2},\frac{n}{2}+1)\in \R^2],$ and the multiplicity of this point is $\frac{n}{2}-1$. 
\end{conjecture}

This computationally motivated conjecture is in fact true; moreover, we have a full characterization of the persistence diagram of a cycle network across all dimensions $k \in \Z_+$. This characterization relies on results in \cite{adamaszek2015vietoris} and \cite{adamaszek2016nerve}, and is stated for even and odd dimensions below:

\begin{restatable}[Even dimension]{theorem}{thmdowkercyceven}
\label{thm:dowker-cyc-even} Fix $n\in \N$, $n \geq 3$. If $l\in \N$ is such that $n$ is divisible by $(l+1)$, and $k:=\tfrac{nl}{l+1}$ is such that $0\leq k \leq n-2$, then $\dgm^{\mf{D}}_{2l}(G_n)$ consists of precisely the point $(\tfrac{nl}{l+1},\tfrac{nl}{l+1} + 1)$ with multiplicity $\tfrac{n}{l+1} -1$. If $l$ or $k$ do not satisfy the conditions above, then $\dgm^{\mf{D}}_{2l}(G_n)$ is trivial. 
\end{restatable}

As a special case, Theorem \ref{thm:dowker-cyc-even} proves Conjecture \ref{conj:dowker-dgm-2} by setting $l=1$. If $n$ is odd, then it is not divisible by $(l+1) = 2$, and so $\dgm^{\mf{D}}_2(G_n)$ is trivial. If $n$ is even, then it is divisible by $(l+1)=2$, and $\tfrac{nl}{l+1} = \tfrac{n}{2} \leq n-2$ because $n$ is at least 4. Thus $\dgm^{\mf{D}}_2(G_n)$ consists of the point $(\tfrac{n}{2},\tfrac{n}{2} + 1)$ with multiplicity $\tfrac{n}{2} - 1$.

\begin{restatable}[Odd dimension]{theorem}{thmdowkercycodd}\label{thm:dowker-cyc-odd} Fix $n\in \N$, $n\geq 3$. Then for $l\in \N$, define $M_l:=\set{m \in \N : \tfrac{nl}{l+1} < m < \tfrac{n(l+1)}{l+2}}$. If $M_l$ is empty, then $\dgm^{\mf{D}}_{2l+1}(G_n)$ is trivial. Otherwise, we have: 
\[\dgm^{\mf{D}}_{2l+1}(G_n) = \set{\left(a_l,\ceil{\tfrac{n(l+1)}{l+2}}\right)},\]
where $a_l:=\min\set{m \in M_l}.$
We use set notation (instead of multisets) to mean that the multiplicity is 1.
\end{restatable}

In particular, for $l=0$, we have $\tfrac{nl}{l+1} = 0$ and $\tfrac{n(l+1)}{l+2} = \tfrac{n}{2} \geq 3/2$, so $1 \in M_l$. Thus we have $\dgm^{\mf{D}}_1(G_n) = \set{\left(1,\ceil{\tfrac{n}{2}}\right)},$ and so Theorem \ref{thm:dowker-cyc-odd} recovers Theorem \ref{thm:cycleH1} as a special case. However, whereas the proof of Theorem \ref{thm:cycleH1} is elementary and pedagogical (it relies on intuitive observations about the structure of a cycle network), the proofs of Theorems \ref{thm:dowker-cyc-even} and \ref{thm:dowker-cyc-odd} use sophisticated machinery developed across \cite{adamaszek2015vietoris} and \cite{adamaszek2016nerve}. We provide details for Theorem \ref{thm:cycleH1} in the body of the paper, and relegate full details of Theorems \ref{thm:dowker-cyc-even} and \ref{thm:dowker-cyc-odd} to Appendix \ref{sec:cycle-addendum}.

\begin{figure}
\begin{subfigure}{0.48\linewidth}
\includegraphics[width = \textwidth]{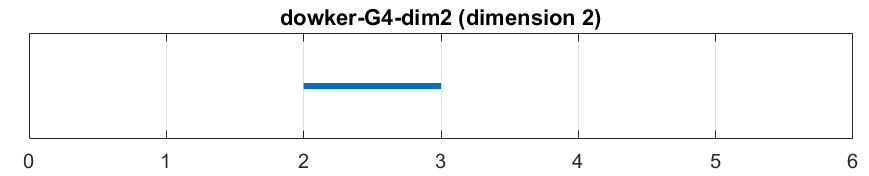}
\end{subfigure}
\begin{subfigure}{0.48\linewidth}
\includegraphics[width = \textwidth]{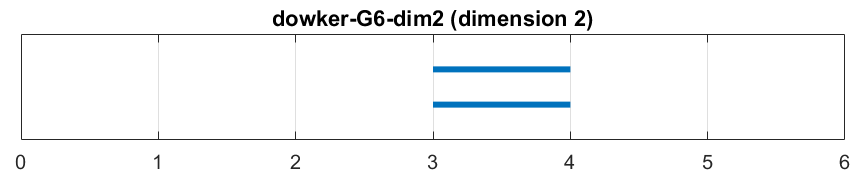}
\end{subfigure}
\begin{subfigure}{0.48\linewidth}
\includegraphics[width = \textwidth]{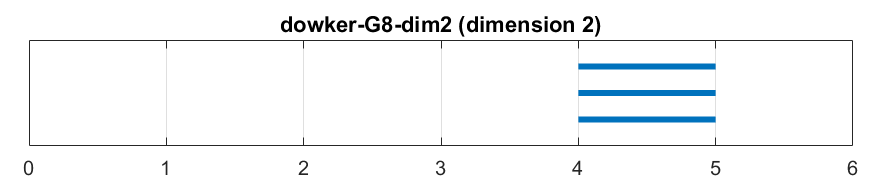}
\end{subfigure}
\caption{Sample 2-dimensional Dowker persistence barcodes for cycle
  networks $G_4,G_6,G_8$. In our experiments, 2-dimensional Dowker persistence barcodes for $G_n$ were always empty for $n$ odd.}

\label{fig:cycle-2dim}
\end{figure}

\end{remark}

\subsection{Sensitivity to network transformations}

We first make the following:
\begin{definition}[Pair swaps] Let $(X,\w_X) \in \Ncal$ be a network. For any $z,z'\in X$, define the \emph{$(z,z')$-swap} of $(X,\w_X)$ to be the network $S_X(z,z'):=(X^{z,z'},\w_X^{z,z'})$ defined as follows:
\begin{align*}
X^{z,z'}&:= X,\\
\text{For any $x,x'\in X^{z,z'}$,}\qquad \w_X^{z,z'}(x,x')&:=\begin{cases}
\w_X(x',x) &: x=z,x'=z'\\
\w_X(x',x) &: x'=z,x=z'\\
\w_X(x,x') &: \text{otherwise.}
\end{cases}
\end{align*}
\end{definition}

We then pose the following question:
\begin{quote}
\emph{Given a network $(X,\w_X)$ and an $(x,x')$-swap $S_X(x,x')$ for some
$x,x'\in X$, how do the Rips or Dowker persistence diagrams of
$S_X(x,x')$ differ from those of $(X,\w_X)$?}
\end{quote}
This situation is
illustrated in Figure \ref{fig:3-node-networks}. Example
\ref{ex:3-node} shows an example where the Dowker persistence diagram captures the
variation in a network that occurs after a pair swap, whereas the Rips
persistence diagram fails to capture this difference. Furthermore,
Remark \ref{rem:swap-rips} shows that Rips persistence diagrams always
fail to do so.

 We also consider the extreme situation where all the directions of the edges of a network are reversed, i.e. the network obtained by applying the pair swap operation to each pair of nodes. We would intuitively expect that the persistence diagrams would not change. The following discussion shows that the Rips and Dowker persistence diagrams are invariant under taking the transpose of a network.

 \begin{proposition}\label{prop:si-so}  
Recall the transposition map $\mf{t}$ and the shorthand notation $X^{\top}=\mf{t}(X)$ from Definition \ref{defn:sym-trans}.  Let $k\in \Z_+$. Then $\dgm_k^{\si}(X)=\dgm_k^{\so}(X^\top),$ and therefore $\dgm_k^{\mf{D}}(X)=\dgm_k^{\mf{D}}(X^\top)$ by Theorem \ref{thm:dowker-functorial}. 
%
 \end{proposition}

\begin{remark}[Pair swaps and their effect]\label{rem:swap-rips} Let $(X,\w_X)\in \Ncal$, let $z,z'\in X$, and let $\s \in \pow(X)$. Then we have:
\[\max_{x,x'\in \s}\w_X(x,x') = \max_{x,x'\in \s}\w_X^{z,z'}(x,x').\]
Using this observation, one can then repeat the arguments used in the proof of Proposition \ref{prop:si-so} to show that: 
\[\dgm_k^{\mf{R}}(X)=\dgm_k^{\mf{R}}(S_X(z,z')),\text{ for each } k \in \Z_+.\]
This encodes the intuitive fact that Rips persistence diagrams are blind to
pair swaps. Moreover, succesively applying the pair swap operation over all pairs produces the transpose of the original network, and so it follows that $\dgm_k^{\mf{R}}(X)=\dgm_k^{\mf{R}}(X^\top)$.

On the other hand, $k$-dimensional Dowker persistence diagrams are not necessarily invariant to pair swaps when $k\geq
1$. Indeed, Example \ref{ex:3-node} below constructs a space $X$ for which
there exist points $z,z'\in X$ such that 
\[\dgm_1^{\mf{D}}(X)\neq\dgm_1^{\mf{D}}(S_X(z,z')).\]

However, 0-dimensional Dowker persistence diagrams are still invariant to pair swaps:

\begin{proposition}
\label{prop:dowker0pair}
Let $(X,\w_X)\in \Ncal$, let $z,z'$ be any two points in $Z$, and let $\s \in \pow(X)$. Then we have:
\[\dgm_0^{\mf{D}}(X)=\dgm_0^{\mf{D}}(S_X(z,z')).\]
\end{proposition}

\end{remark}

\begin{example} 
\label{ex:3-node}
Consider the three node dissimilarity networks $(X,\w_X)$ and
$(Y,\w_Y)$ in Figure \ref{fig:3-node-networks}. Note that $(Y,\w_Y)$
coincides with $S_X(a,c)$. We present both the Dowker and Rips persistence barcodes obtained from these networks. Note that the Dowker persistence barcode is sensitive to the difference between $(X,\w_X)$ and $(Y,\w_Y)$, whereas the Rips barcode is blind to this difference. We refer the reader to \S\ref{sec:exp} for details on how we compute these barcodes.

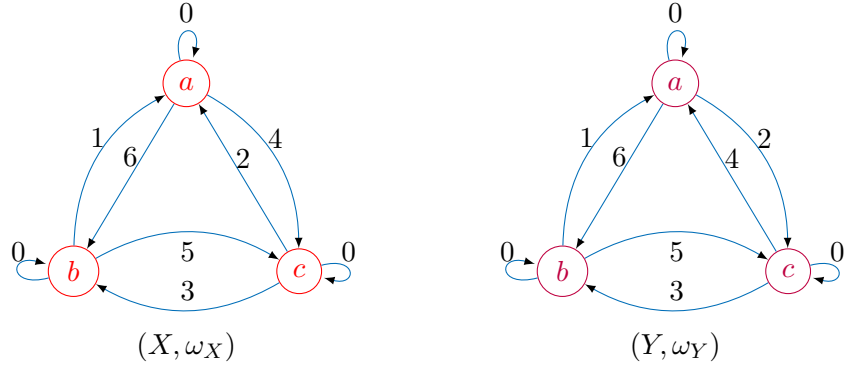
\begin{figure}
\begin{center}

\begin{tikzpicture}
\tikzset{>=latex}
\begin{scope}[draw,red]
\node[circle,draw](1) at (0,2.5){$a$};
\node[circle,draw](2) at (-1.5,0){$b$};
\node[circle,draw](3) at (1.5,0){$c$};
\node[black] at (0,-1){$(X,\w_X)$};
\end{scope}

\begin{scope}[xshift=0cm,draw,purple]
\node[circle,draw](4) at (6.5,2.5){$a$};
\node[circle,draw](5) at (5,0){$b$};
\node[circle,draw](6) at (8,0){$c$};
\node[black] at (6.5,-1){$(Y,\w_Y)$};
\end{scope}

\begin{scope}[draw=NavyBlue]
\path[->] (1) edge [loop above] node[above,pos=0.5]{$0$} (1);
\path[->] (2) edge [loop left] node[above,pos=0.5]{$0$} (2);
\path[->] (3) edge [loop right] node[above,pos=0.5]{$0$} (3);

\path[->] (4) edge [loop above] node[above,pos=0.5]{$0$} (4);
\path[->] (5) edge [loop left] node[above,pos=0.5]{$0$} (5);
\path[->] (6) edge [loop right] node[above,pos=0.5]{$0$} (6);

\path[->] (1) edge [bend left,in=180,out=0] node[above,pos=0.5]{$6$} (2);
\path[->] (2) edge [bend left] node[above,pos=0.5]{$1$} (1);
\path[->] (1) edge [bend left] node[above,pos=0.5]{$4$} (3);
\path[->] (3) edge [bend left,out=0,in=180] node[above,pos=0.5]{$2$} (1);
\path[->] (2) edge [ bend left] node[below,pos=0.5]{$5$} (3);
\path[->] (3) edge [bend left] node[above,pos=0.5]{$3$} (2);

\path[->] (4) edge [bend left,in=180,out=0] node[above,pos=0.5]{$6$} (5);
\path[->] (5) edge [bend left] node[above,pos=0.5]{$1$} (4);
\path[->] (4) edge [bend left] node[above,pos=0.5]{$2$} (6);
\path[->] (6) edge [bend left,in=180,out=0] node[above,pos=0.5]{$4$} (4);
\path[->] (5) edge [ bend left] node[below,pos=0.5]{$5$} (6);
\path[->] (6) edge [bend left] node[above,pos=0.5]{$3$} (5);
\end{scope}

\end{tikzpicture}
\end{center}
\caption{$(Y,\w_Y)$ is the $(a,c)$-swap of $(X,\w_X)$.}
\label{fig:3-node-networks}
\end{figure}

\begin{figure}
\begin{subfigure}{0.49\linewidth}
\includegraphics[width = \textwidth,keepaspectratio]
{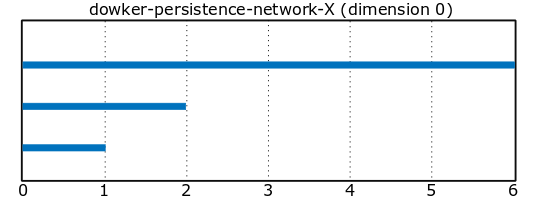}
\end{subfigure}
\begin{subfigure}{0.49\linewidth}
\includegraphics[width = \textwidth]
{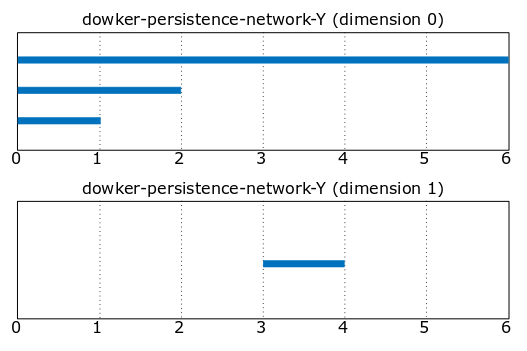}
\end{subfigure}
\caption{Dowker persistence barcodes of networks $(X,\w_X)$ and $(Y,\w_Y)$ from Figure \ref{fig:3-node-networks}.}
\label{fig:pers-nets}
\end{figure}

\begin{figure}
\begin{subfigure}{0.49\linewidth}
\includegraphics[width = \textwidth,keepaspectratio]
{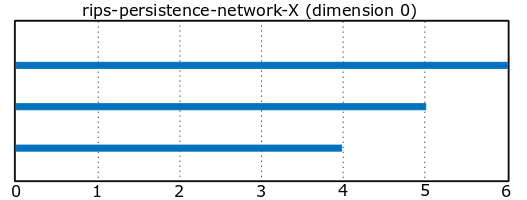}
\end{subfigure}
\begin{subfigure}{0.49\linewidth}
\includegraphics[width = \textwidth]
{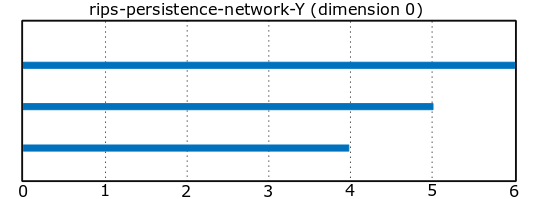}
\end{subfigure}
\caption{Rips persistence barcodes of networks $(X,\w_X)$ and $(Y,\w_Y)$ from Figure \ref{fig:3-node-networks}. Note that the Rips diagrams indicate no persistent homology in dimensions higher than 0, in contrast with the Dowker diagrams in Figure \ref{fig:pers-nets}.}
\label{fig:rips-pers-nets}
\end{figure}

To show how the Dowker complex is constructed, we also list the Dowker sink complexes of the networks in Figure \ref{fig:3-node-networks}, and also the corresponding homology dimensions across a range of resolutions. Note that when we write $[a,b](a)$, we mean that $a$ is a sink corresponding to the simplex $[a,b]$.
\begin{align*}
\sink_{0,X} = \set{[a],[b],[c]} && \dim(H_1(\sink_{0,X})) = 0\\
\sink_{1,X} = \set{[a],[b],[c],[a,b](a)} && \dim(H_1(\sink_{1,X})) = 0\\ 
\sink_{2,X} = \set{[a],[b],[c],[a,b](a),[a,c](a),[b,c](a),[a,b,c](a)} && \dim(H_1(\sink_{2,X})) = 0\\ 
\sink_{3,X} = \set{[a],[b],[c],[a,b](a),[a,c](a),[b,c](a),[a,b,c](a)} && \dim(H_1(\sink_{3,X})) = 0
\end{align*}

\begin{align*}
\sink_{0,Y} = \set{[a],[b],[c]} && \dim(H_1(\sink_{0,Y})) = 0\\
\sink_{1,Y} = \set{[a],[b],[c],[a,b](a)} && \dim(H_1(\sink_{1,Y})) = 0\\ 
\sink_{2,Y} = \set{[a],[b],[c],[a,b](a),[a,c](c)} && \dim(H_1(\sink_{2,Y})) = 0\\ 
\sink_{3,Y} = \set{[a],[b],[c],[a,b](a),[a,c](c),[b,c](b)} && \dim(H_1(\sink_{3,Y})) = 1\\ 
\sink_{4,Y} = \set{[a],[b],[c],[a,b](a),[a,c](a),[b,c](a),[a,b,c](a)} && \dim(H_1(\sink_{4,Y})) = 0\\ 
\end{align*}

Note that for $\d \in [3,4)$, $\dim(H_1(\sink_{\d,Y})) = 1$, whereas $\dim(H_1(\sink_{\d,X}))=0$ for each $\d\in \R$. 

\end{example}

Based on the discussion in Remark \ref{rem:swap-rips}, Proposition \ref{prop:dowker0pair}, and Example \ref{ex:3-node}, we conclude the following:
\begin{center}
\textbf{Moral: }\textit{Unlike Rips persistence diagrams, Dowker persistence diagrams are truly sensitive to asymmetry.}
\end{center}

We summarize some of these results:

\begin{theorem} 
\label{thm:sym-trans-summary}
Recall the symmetrization and transposition maps $\mf{s}$ and $\mf{t}$ from Definition \ref{defn:sym-trans}. Then:
\begin{enumerate}
\item $\mf{R}\circ \mf{s} = \mf{R}$,
\item $\mf{D}^{\so}\circ \mf{t} = \mf{D}^{\si}$, and 
\item $\mf{D}^{\si}\circ \mf{t} = \mf{D}^{\so}$.
\end{enumerate}
Also, there exist $(X,\w_X), (Y,\w_Y) \in \Ncal$ such that $(\mf{D}^{\si}\circ \mf{s})(X) \neq \mf{D}^{\si}(X)$, and $(\mf{D}^{\so}\circ \mf{s})(Y) \neq \mf{D}^{\so}(Y)$.

\end{theorem}
\begin{proof} These follow from Example \ref{ex:3-node}, Remark \ref{rem:rips-symm}, and Proposition \ref{prop:si-so}.\end{proof}


\section{Implementation and an experiment on network classification}\label{sec:exp}

In this section, we present the results of an experiment where we applied our methods to perform a classification task on a database of networks.
All persistent homology computations were carried out using the \texttt{Javaplex} package for Matlab. A full description of Javaplex can be found in \cite{tausz2011javaplex}. We used $\mathbb{K}=\Z_2$ as the field of coefficients for all our persistence computations. The dataset and software used for our computations are available as part of the \texttt{PersNet} software package on \url{https://research.math.osu.edu/networks/Datasets.html}. A version of our simulated hippocampal networks experiment has appeared in \cite{dowker-asilo}.


All networks in the following experiment were normalized to have weights in the range $[0,1]$. For each network, we computed Dowker sink complexes at resolutions $\d = 0.01,0.02,0.03,\ldots,1.00$. This filtration was then passed into Javaplex, which produced 0 and 1-dimensional Dowker persistence barcodes.

\subsection{Simulated hippocampal networks}
\label{sec:exp-arenas}
In the neuroscience literature, it has been shown that as an animal explores a given \emph{environment} or \emph{arena}, specific ``place cells" in the hippocampus show increased activity at specific spatial regions, called ``place fields" \cite{o1971hippocampus}. Each place cell shows a \emph{spike} in activity when the animal enters the place field linked to this place cell, accompanied by a drop in activity as the animal moves far away from this place field. To understand how the brain processes this data, a natural question to ask is the following: Is the time series data of the place cell activity, referred to as ``spike trains", enough to detect the structure of the arena?

Approaches based on homology \cite{curto2008cell} and persistent homology \cite{dabaghian2012topological} have shown positive results in this direction. In \cite{dabaghian2012topological}, the authors simulated the trajectory of a rat in an arena containing ``holes." A simplicial complex was then built as follows: whenever $n+1$ place cells with overlapping place fields fired together, an $n$-simplex was added. This yield a filtered simplicial complexed indexed by a time parameter. By computing persistence, it was then shown that the number of persistent bars in the 1-dimensional barcode of this filtered simplicial complex would accurately represent the number of holes in the arena. 

We repeated this experiment with the following change in methodology: we simulated the movement of an animal, and corresponding hippocampal activity, in arenas with a variety of obstacles. We then induced a directed network from each set of hippocampal activity data, and computed the associated 1-dimensional Dowker persistence diagrams. We were interested in seeing if the bottleneck distances between diagrams arising from similar arenas would differ significantly from the bottleneck distance between diagrams arising from different arenas. 
To further exemplify our methods, we repeated our analysis after computing the 1-dimensional Rips persistence diagrams from the hippocampal activity networks.

In our experiment, there were five arenas. The first was a square of side length $L=10$, with four circular ``holes" or ``forbidden zones" of radius $0.2L$ that the trajectory could not intersect. The other four arenas were those obtained by removing the forbidden zones one at a time. In what follows, we refer to the arenas of each type as \emph{4-hole, 3-hole, 2-hole, 1-hole,} and \emph{0-hole arenas}. For each arena, a random-walk trajectory of 5000 steps was generated, where the animal could move along a square grid with 20 points in each direction. The grid was obtained as a discretization of the box $[0,L]\times [0,L]$, and each step had length $0.05L$. The animal could move in each direction with equal probability. If one or more of these moves took the animal outside the arena (a disallowed move), then the probabilities were redistributed uniformly among the allowed moves. Each trajectory was tested to ensure that it covered the entire arena, excluding the forbidden zones. Formally, we write the time steps as a set $T:=\set{1,2,\ldots, 5000}$, and denote the trajectory as a map $\operatorname{traj}:T \r [0,L]^2$.

For each of the five arenas, 20 trials were conducted, producing a total of 100 trials. For each trial $l_k$, an integer $n_k$ was chosen uniformly at random from the interval $[150,200]$. Then $n_k$ place fields of radius $0.05L$ were scattered uniformly at random inside the corresponding arena for each $l_k$. An illustration of the place field distribution is provided in Figure \ref{fig:arenas-rasters}. A spike on a place field was recorded whenever the trajectory would intersect it. So for each $1\leq i\leq n_k$, the spiking pattern of cell $x_i$, corresponding to place field PF$_i$, was recorded via a function $r_i:T\r \set{0,1}$ given by:
\[r_i(t)=\begin{cases}
1 &:\text{if } \operatorname{traj}(t)\text{ intersects  } \text{PF}_i,\\
0 &: \text{otherwise}\end{cases} \qquad\qquad t\in T.\]

The matrix corresponding to $r_i$ is called the \emph{raster} of cell $x_i$. A sample raster is illustrated in Figure \ref{fig:arenas-rasters}. For each trial $l_k$, the corresponding network $(X_k,\w_{X_k})$ was constructed as follows: $X_k$ consisted of $n_k$ nodes representing place fields, and for each $1\leq i,j\leq n_k$, the weight $\w_{X_k}(x_i,x_j)$ was given by:
\begin{align*}
\w_{X_k}(x_i,x_j) &:=1-\frac{N_{i,j}(5)}{\sum_{i=1}^{n_k}N_{i,j}(5)},\\
\text{ where }
N_{i,j}(5)&=\card\left(\set{(s,t)\in T^2:t\in  [2,5000], t-s\in [1,5], r_j(t)=1,r_i(s)=1}\right).
\end{align*}

In words, $N_{i,j}(5)$ counts the pairs of times $(s,t), s < t,$ such that cell $x_j$ spikes (at a time $t$) after cell $x_i$ spikes (at a time $s$), and the delay between the two spikes is fewer than 5 time steps. The idea is that if cell $x_j$ frequently fires within a short span of time after cell $x_i$ fires, then place fields PF$_i$ and PF$_j$ are likely to be in close proximity to each other. The column sum of the matrix corresponding to $\w_{X_k}$ is normalized to 1, and so $\w_{X_k}^\top$ can be interpreted as the transition matrix of a Markov process.

\begin{figure}
\begin{subfigure}{0.3\linewidth}
\includegraphics[width = \textwidth,keepaspectratio]{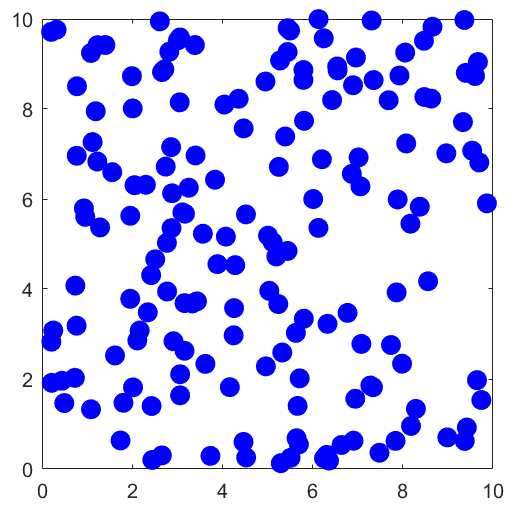}
\end{subfigure}
\begin{subfigure}{0.3\linewidth}
\includegraphics[width = \textwidth,keepaspectratio]{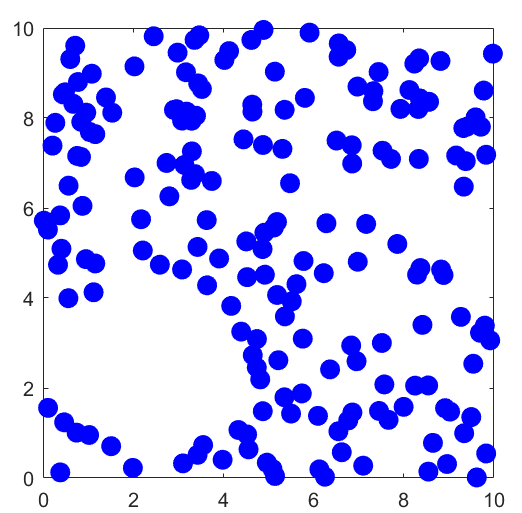}
\end{subfigure}
\begin{subfigure}{0.3\linewidth}
\includegraphics[width = \textwidth,keepaspectratio]{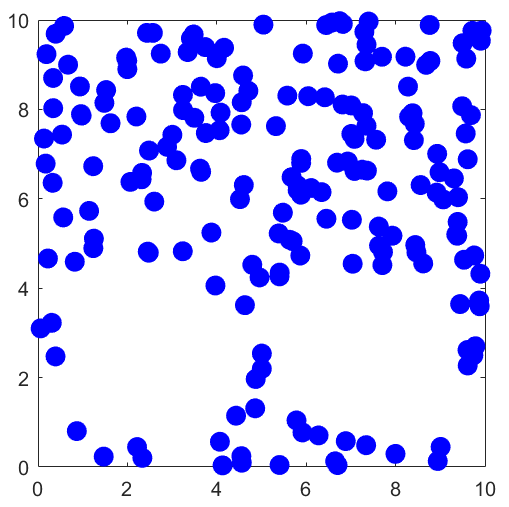}
\end{subfigure}
\begin{subfigure}{0.3\linewidth}
\includegraphics[width = \textwidth,keepaspectratio]{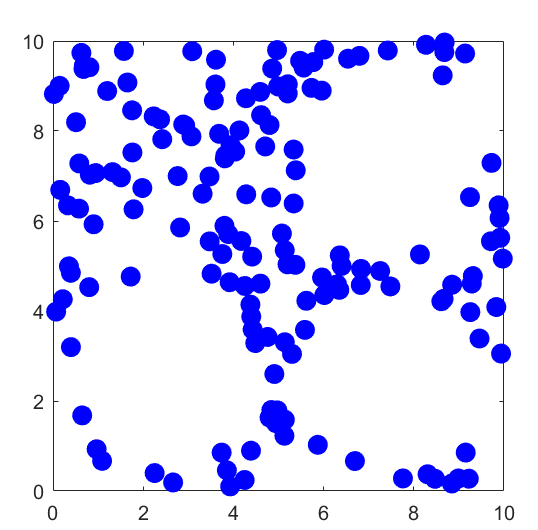}
\end{subfigure}
\begin{subfigure}{0.3\linewidth}
\includegraphics[width = \textwidth,keepaspectratio]{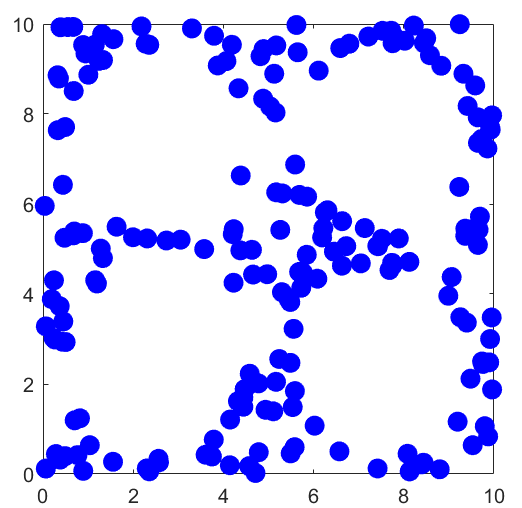}
\end{subfigure}
\begin{subfigure}{0.3\linewidth}
\includegraphics[width = \textwidth,keepaspectratio]{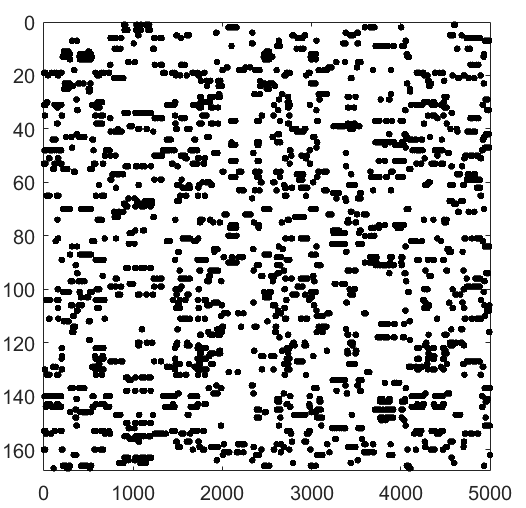}
\end{subfigure}
\caption{\textbf{Bottom right:} Sample place cell spiking pattern
  matrix. The $x$-axis corresponds to the number of time steps, and the $y$-axis corresponds to the number of place cells. Black dots represent spikes. \textbf{Clockwise from bottom
    middle:} Sample distribution of place field centers in 4, 3, 0, 1,
  and 2-hole arenas. 
}
\label{fig:arenas-rasters}
\end{figure}

Next, we computed the 1-dimensional Dowker persistence diagrams of each of the 100 networks. Note that $\dgm_1^{\mf{D}}(\w_X)=\dgm_1^{\mf{D}}(\w_X^\top)$ by Proposition \ref{prop:si-so}, so we are actually obtaining the 1-dimensional Dowker persistence diagrams of transition matrices of Markov processes. We then computed a $100\times 100$ matrix consisting of the bottleneck distances between all the 1-dimensional persistence diagrams. The single linkage dendrogram generated from this bottleneck distance matrix is shown in Figure \ref{fig:dendro-dowker-arenas}. The labels are in the format \texttt{env-<nh>-<nn>}, where \texttt{nh} is the number of holes in the arena/environment, and \texttt{nn} is the number of place fields. Note that with some exceptions, networks corresponding to the same arena are clustered together. We conclude that the Dowker persistence diagram succeeded in capturing the intrinsic differences between the five classes of networks arising from the five different arenas, even when the networks had different sizes.

We then computed the Rips persistence diagrams of each network, and computed the $100\times 100$ bottleneck distance matrix associated to the collection of 1-dimensional diagrams. The single linkage dendrogram generated from this matrix is given in Figure \ref{fig:dendro-rips-arenas}. Notice that the Rips dendrogram does not do a satisfactory job of classifying arenas correctly.

\begin{remark} We note that an alternative method of comparing the networks obtained from our simulations would have been to compute the pairwise network distances, and plot the results in a dendrogram. But $\dn$ is NP-hard to compute---this follows from the fact that computing $\dn$ includes the problem of computing Gromov-Hausdorff distance between finite metric spaces, which is NP-hard \cite{schmiedl}. So instead, we are computing the bottleneck distances between 1-dimensional Dowker persistence diagrams, as suggested by Remark \ref{rem:dowker-benefits}.

\end{remark}

\begin{figure}
\begin{center}
\includegraphics[width = 0.95\textwidth]{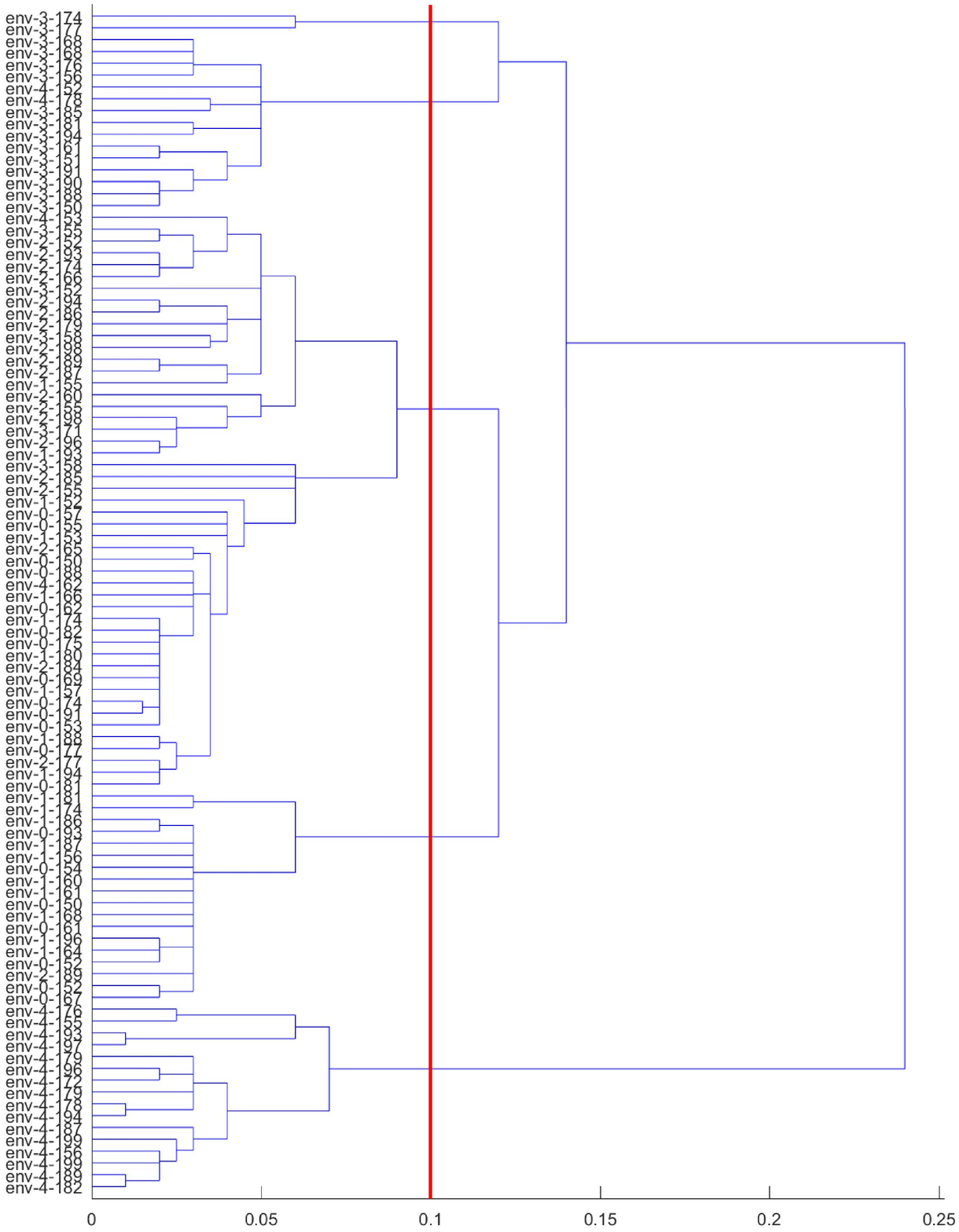}
\end{center}
\caption{Single linkage dendrogram corresponding to the distance matrix obtained by computing bottleneck distances between 1-dimensional Dowker persistence diagrams of our database of hippocampal networks (\S\ref{sec:exp-arenas}). Note that the 4, 3, and 2-hole arenas are well separated into clusters at threshold 0.1.} \label{fig:dendro-dowker-arenas}
\end{figure}

\begin{figure}
\begin{center}
\includegraphics[width = 0.95\textwidth]{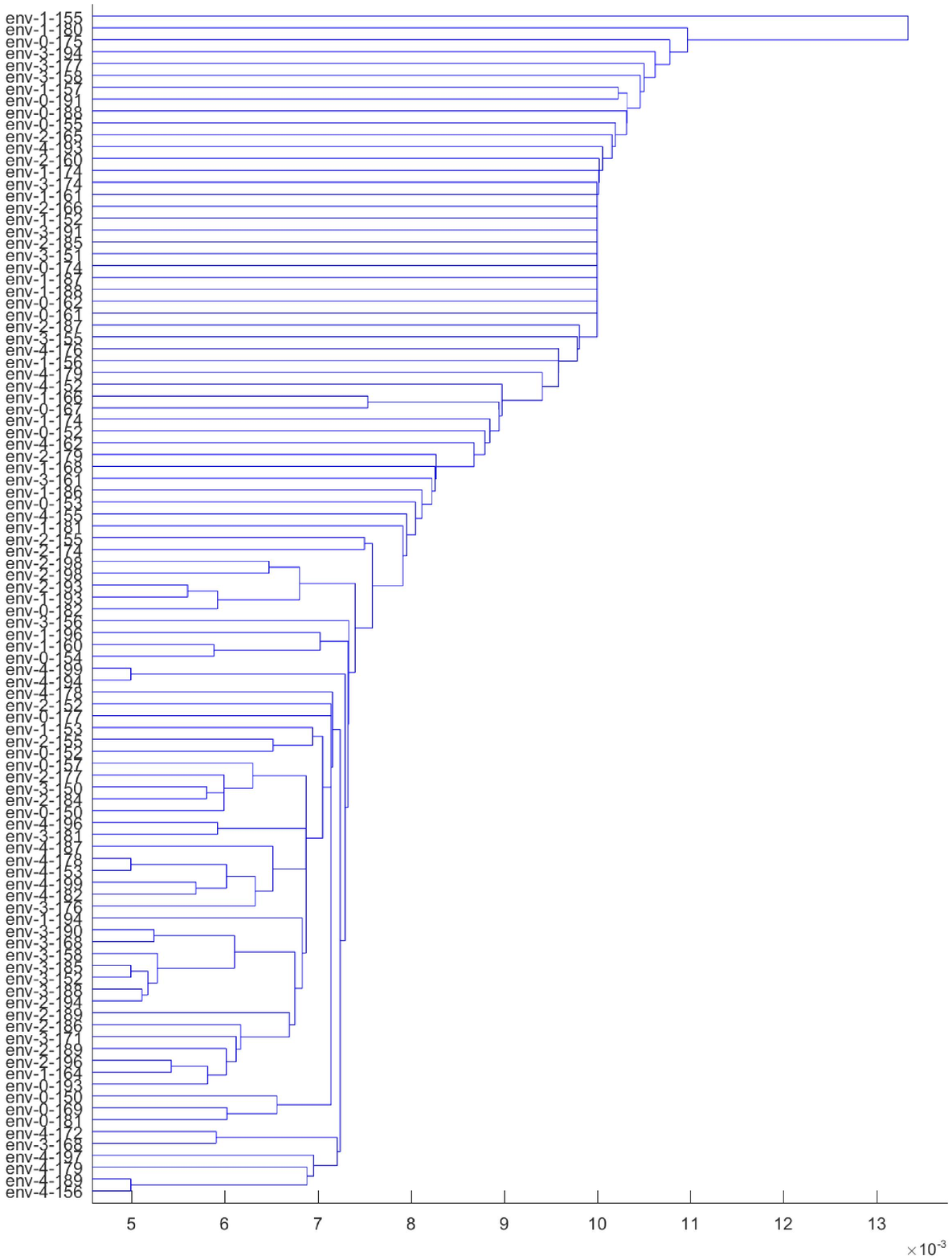}
\end{center}
\caption{Single linkage dendrogram corresponding to the distance matrix obtained by computing bottleneck distances between 1-dimensional Rips persistence diagrams of our database of hippocampal networks (\S\ref{sec:exp-arenas}). Notice that the hierarchical clustering fails to capture the correct arena types.} \label{fig:dendro-rips-arenas}
\end{figure}

\section{Discussion}
We provided a complete description of the Rips and
Dowker persistence diagrams of general networks. The stability results
we have provided give quantitative guarantees on the robustness of
these persistence diagrams. As a building block, we proved a functorial generalization of Dowker's theorem, which also yields an independent proof of a folklore strengthening of Dowker's theorem. We have provided numerous examples suggesting that
Dowker persistence diagrams are an appropriate method for analyzing
general asymmetric networks. For a particular class of such examples,
the family of cycle networks, we have fully characterized their Dowker persistence diagrams in all dimensions. Finally, we have implemented our methods for a classification task on a database of networks, and provided interpretations for our results.

We believe that the story of ``persistent homology of asymmetric
networks" has more aspects to be uncovered. Of particular interest to
us is the analysis of alternative methods of producing simplicial
complexes from asymmetric networks, for example, the \emph{directed
  flag complex} construction of \cite{dlotko2016topological}. Yet another interesting extension to the non-metric framework has appeared in \cite{edelsbrunner2016topological}, in the context of computing generalized \v{C}ech and Rips complexes for Bregman divergences. We
remark that a persistent homology framework for the directed flag
complex has been proposed by \cite{turner}, but the computational
aspects of this construction have not been addressed in the current
literature.  Another approach for computing persistence diagrams from
asymmetric networks, which bypasses the construction of any simplicial
complex and operates directly at the chain level is given in \cite{pph}. Some other interesting questions relate to cycle networks: for example, we would like to obtain a characterization of the Rips persistence diagrams of cycle networks for any dimension $k\geq 1$. Finally, it is important to devise more efficient implementations for the Dowker complexes we present here. It is likely that ideas from the literature on efficient construction of \v{C}ech complexes \cite{dantchev2012efficient, edelsbrunner2016topological} will be helpful in this regard.

\medskip
\paragraph{\textbf{Acknowledgments.}} This work was supported by NSF grants IIS-1422400 and CCF-1526513. We thank Pascal Wild and Zhengchao Wan for pointing out errors on an early preprint, and also Osman Okutan and Tim Porter for useful discussions. We are especially thankful to Henry Adams for numerous helpful observations and suggestions, especially regarding the material in Appendix \ref{sec:cycle-addendum}, and for suggesting the proof strategy for Theorem \ref{thm:cyc-cech-main}.

\bibliographystyle{alpha}
\bibliography{biblio}

\appendix

\section{Proofs} \label{app:proofs}

\begin{proof}[Proof of Lemma \ref{lem:stab}] The first inequality holds by the Algebraic Stability Theorem. For the second inequality, note that the contiguous simplicial maps in the diagrams above induce chain maps between the corresponding chain complexes, and these in turn induce equal linear maps at the level of homology vector spaces. To be more precise, first consider the maps $t_{\d+\eta,\d'+\eta}\circ \ph_\d$ and $\ph_{\d'}\circ s_{\d,\d'}$. These simplicial maps induce linear maps  $(t_{\d+\eta,\d'+\eta}\circ \ph_\d)_\#, (\ph_{\d'}\circ s_{\d,\d'})_\#: H_k(\mf{F}^\d) \r H_k(\mf{G}^{\d'+\eta})$. Because the simplicial maps are assumed to be contiguous, we have:
\[(t_{\d+\eta,\d'+\eta}\circ \ph_\d)_\# = 
(\ph_{\d'}\circ s_{\d,\d'})_\#.\]
By invoking functoriality of homology, we then have:
\[(t_{\d+\eta,\d'+\eta})_\# \circ (\ph_\d)_\# = 
(\ph_{\d'})_\#\circ (s_{\d,\d'})_\#.\]
Analogous results hold for the other pairs of contiguous maps. Thus we obtain commutative diagrams upon passing to homology, and so $\H_k(\mf{F}), \H_k(\mf{G})$ are $\eta$-interleaved for each $k\in \Z_+$. Thus we get:
\[\di(\H_k(\mf{F}), \H_k(\mf{G}))\leq \eta. \qedhere\]
\end{proof}

\begin{proof}[Proof of Proposition \ref{prop:dn-ko}] First we show that: 
\[ \dn(X,Y) \geq \tfrac{1}{2}\inf\{\max(\dis(\ph),\dis(\psi),C_{X,Y}(\ph,\psi), C_{Y,X}(\psi,\ph)) : \ph:X \r Y, \psi:Y \r X \text{ any maps}\}.\] 
Let $\eta = \dn(X,Y)$, and let $R$ be a correspondence such that $\dis(R) = 2\eta$. We can define maps $\ph:X\r Y$ and $\psi:Y\r X$ as follows: for each $x\in X$, set $\ph(x)=y$ for some $y$ such that $(x,y)\in R$. Similarly, for each $y\in Y$, set $\psi(y)=x$ for some $x$ such that $(x,y)\in R$. 
Thus for any $x \in X, y\in Y$, we obtain $|\w_X(x,\psi(y)) - \w_Y(\ph(x),y)| \leq 2\eta$ and $|\w_X(\psi(y),x) - \w_Y(y,\ph(x))| \leq 2\eta$. So we have both $C_{X,Y}(\ph,\psi) \leq 2\eta$ and $C_{Y,X}(\psi,\ph) \leq 2\eta$. Also for any $x,x' \in X$, we have $(x,\ph(x)),(x',\ph(x')) \in R$. Thus we also have 
$$|\w_X(x,x') - \w_Y(\ph(x),\ph(x'))| \leq 2\eta.$$
So $\dis(\ph) \leq 2\eta$ and similarly $\dis(\psi) \leq 2\eta$. This proves the ``$\geq$" case.

Next we wish to show:
\[ \dn(X,Y) \leq \tfrac{1}{2}\inf\{\max(\dis(\ph),\dis(\psi),C_{X,Y}(\ph,\psi), C_{Y,X}(\psi,\ph)) : \ph:X \r Y, \psi:Y \r X \text{ any maps}\}.\] 
Suppose $\ph, \psi$ are given, and $\frac{1}{2}\max(\dis(\ph),\dis(\psi),C_{X,Y}(\ph,\psi),C_{Y,X}(\psi,\ph)) = \eta$.

Let $R_X = \set{(x,\ph(x) : x\in X}$ and let $R_Y = \set{(\psi(y),y) : y\in Y}$. Then $R = R_X \cup R_Y$ is a correspondence. We wish to show that for any $z = (a,b), z' = (a',b') \in R$, \[|\w_X(a,a') - \w_Y(b,b')| \leq 2\eta.\] This will show that $\dis(R) \leq 2\eta$, and so $\dn(X,Y) \leq \eta$.

To see this, let $z,z' \in R$. Note that there are four cases: (1) $z,z' \in R_X$, (2) $z,z' \in R_Y$, (3) $z \in R_X, z' \in R_Y$, and (4) $z\in R_Y, z'\in R_X$. In the first two cases, the desired inequality follows because $\dis(\ph), \dis(\psi) \leq 2\eta$. The inequality follows in cases (3) and (4) because $C_{X,Y}(\ph,\psi) \leq 2\eta$ and $C_{Y,X}(\psi,\ph) \leq 2\eta$, respectively. Thus $\dn(X,Y) \leq \eta$. \end{proof}

\begin{proof}[Proof of Proposition \ref{prop:subdiv-identity}] It suffices to show that $\Phi$ is a simplicial approximation to $\mc{E}_{|\Sigma|}$, i.e. whenever $\mc{E}_{|\Sigma|}(x) \in |{\s}|$ for some vertex $x \in |\Si\1|$ and some simplex $\s \in |\Si|$, we also have $|\Phi|(x) \in |{\s}|$ \cite[\S 3.4]{spanier-book}. Here $|\s|$ denotes the \emph{closed simplex} of $\s$; for any simplex $\s=[v_0,\ldots, v_k]$, this is the collection of formal convex combinations $\sum_{i=0}^ka_iv_i$ with $a_i \geq 0$ for each $0\leq i \leq k$ and $\sum_{i=0}^ka_i =1$. 

Let $x = \sum_{i=0}^ka_i\s_i$ be a vertex in $|\Sigma\1|$, with each $a_i > 0$. Then we have $\mc{E}_{|\Sigma|}(x) = \sum_{i=0}^ka_i\mc{B}(\s_i) = \sum_{i=0}^ka_i\sum_{v\in \s_i}v/{\card(\s_i)},$ a vertex in $|\s_k|$. 

Also we have $|\Phi|(x) = \sum_{i=0}^ka_i\Phi(\s_i)$, a vertex in $|{\s_k}|$. Thus $\Phi$ is a simplicial approximation to $\mc{E}_{|\Sigma|}$, and so we have $|\Phi|\simeq \mc{E}_{|\Sigma|}$. \end{proof}

\begin{proof}[Proof of Proposition \ref{prop:si-so}]
Let $\d \in \R$. We first claim that $\sink_\d(X) = \src_\d(X^\top)$. Let $\s \in \sink_\d(X)$. Then there exists $x'$ such that $\w_X(x,x')\leq \d$ for any $x\in \s$. Thus $\w_{X^\top}(x',x)\leq \d$. So $\s \in \src_\d(X^{\top})$. A similar argument shows the reverse containment. This proves our claim. Thus for $\d \leq \d' \leq \d''$, we obtain the following diagram:
 \[\begin{tikzcd}
 \sink_\d(X) \arrow{r}\ar[-, double equal sign distance=3pt]{d} & \sink_{\d'}(X) \arrow{r}\ar[-, double equal sign distance=3pt]{d} & \sink_{\d''}(X) \arrow{r}\ar[-, double equal sign distance=3pt]{d} & \ldots \\
 \src_\d(X^\top) \arrow{r} & \src_{\d'}(X^\top)\arrow{r} & \src_{\d''}(X^\top)\arrow{r} & \ldots
 \end{tikzcd}
 \]
 Since the maps $\sink_\d \r \sink_{\d'}$, $\src_\d \r \src_{\d'}$ for $\d'\geq \d$ are all inclusion maps, it follows that the diagrams commute. Thus at the homology level, we obtain, via functoriality of homology, a commutative diagram of vector spaces where the intervening vertical maps are isomorphisms. By the Persistence Equivalence Theorem (\ref{thm:pet}), the diagrams $\dgm_k^{\si}(X)$ and $\dgm_k^{\so}{(X^\top)}$ are equal. By invoking Corollary \ref{cor:dowker-dual}, we obtain $\dgm_k^{\mf{D}}(X)=\dgm_k^{\mf{D}}(X^\top)$.\end{proof}

\begin{proof}[Proof of Proposition \ref{prop:dowker0pair}] Let $\d \in \R$. For notational convenience, we write, for each $k\in \Z_+$, 
\begin{center}
\begin{tabular}{lll}
$\sink_\d:=\sink_{\d,X} $&  $C_k^\d:=C_k(\sink_{\d,X})$ & $\p_k^\d:=\p_k^\d : C_k^\d \r C_{k-1}^\d$ \\
$\sink_{\d,S}:=\sink_{\d,S_X(z,z')}$&  $C_k^{\d,S}:=C_k(\sink_{\d,S_X(z,z')})$ & $\p_k^{\d,S}:=\p_k^{\d,S}:C_k^{\d,S} \r C_{k-1}^{\d,S}.$ 
\end{tabular}
\end{center} 
First note that pair swaps do not affect the entry of 0-simplices into the Dowker filtration. More precisely, for any $x\in X$, we can unpack the definition of $R_{\d,X}$ (Equation \ref{eq:relation}) to obtain:
\[[x]\in \sink_\d \iff \w_X(x,x)\leq \d \iff \w_X^{z,z'}(x,x)\leq \d \iff [x]\in \sink_{\d,S}.\]
Thus for any $\d \in \R$, we have $C_0^\d = C_0^{\d,S}$. Since all 0-chains are automatically 0-cycles, we have $\ker(\p_0^\d)=\ker(\p_0^{\d,S})$.

Next we wish to show that $\im(\p_1^{\d})=\im(\p_1^{\d,S})$ for each $\d \in \R$. Let $\g \in C_1^\d$. We first need to show the forward inclusion, i.e. that $\p_1^\d(\g) \in \im(\p_1^{\d,S})$. It suffices to show this for the case that $\g$ is a single 1-simplex $[x,x']\in \sink_\d$; the case where $\g$ is a linear combination of 1-simplices will then follow by linearity. Let $\g=[x,x']\in \sink_\d$ for $x,x'\in X$. Then we have the following possibilities:
\begin{enumerate}
\item $x'' \in X\setminus\{z,z'\}$ is a $\d$-sink for $[x,x']$.
\item $z$ (or $z'$) is the only $\d$-sink for $[x,x']$, and $x,x'\not\in \set{z,z'}$.
\item $z$ (or $z'$) is the only $\d$-sink for $[x,x']$, and either $x$ or $x'$ belongs to $\set{z,z'}$.
\item $z$ (or $z'$) is the only $\d$-sink for $[x,x']$, and both $x,x'$ belong to $\set{z,z'}$.
\end{enumerate}

In cases (1), (2), and (4), the $(z,z')$-pair swap has no effect on $[x,x']$, in the sense that we still have $[x,x']\in \sink_{\d,S}$. So $[x']-[x]=\p_1^\d(\g)=\p_1^{\d,S}(\g)\in \im(\p_1^{\d,S})$. Next consider case (3), and assume for notational convenience that $[x,x']=[z,x']$ and $z'$ is the only $\d$-sink for $[z,x']$. By the definition of a $\d$-sink, we have $\overline{\w}_X(z,z')\leq \d$ and $\overline{\w}_X(x',z')\leq \d$. Notice that we also have:
\[[z,z'],[z',x']\in \sink_\d, \text{ with $z'$ as a $\d$-sink}.\]

After the $(z,z')$-pair swap, we still have $\overline{\w}_X^{z,z'}(x',z')\leq \d$, but possibly $\overline{\w}_X^{z,z'}(z,z')> \d$. So it might be the case that $[z,x']\not\in \sink_{\d,S}$. However, we now have:
\begin{align*}
&[z',x']\in \sink_{\d,S}, \text{ with $z'$ as a $\d$-sink, and}\\
&[z,z'] \in \sink_{\d,S}, \text{ with $z$ as a $\d$-sink}.
\end{align*}
Then we have:
\begin{align*}
\p_1^\d(\g)=\p_1^\d([z,x'])=x'-z &=z'-z + x'-z'\\
&=\p_1^{\d}([z,z'])+\p_1^{\d}([z',x'])\\
&=\p_1^{\d,S}([z,z'])+\p_1^{\d,S}([z',x'])\in \im(\p_1^{\d,S}),
\end{align*}
where the last equality is defined because we have checked that $[z,z'],[z',x']\in \sink_{\d,S}$. Thus $\im(\p_1^{\d})\subseteq \im(\p_1^{\d,S})$, and the reverse inclusion follows by a similar argument. 

Since $\d\in \R$ was arbitrary, this shows that $\im(\p_1^{\d})= \im(\p_1^{\d,S})$ for each $\d \in \R$. Previously we had $\ker(\p_0^\d)=\ker(\p_0^{\d,S})$ for each $\d \in \R$. It then follows that $H_0(\sink_\d)=H_0(\sink_{\d,S})$ for each $\d \in \R$. 

Next let $\d'\geq \d \in \R$, and for any $k\in \Z_+$, let $f_k^{\d,\d'}:C_k^\d \r C_k^{\d'}, g_k^{\d,\d'}:C_k^{\d,S} \r C_k^{\d',S}$ denote the chain maps induced by the inclusions $\sink_\d \hr \sink_{\d'}, \sink_{\d,S} \hr \sink_{\d',S}$. Since $\sink_\d$ and $\sink_{\d,S}$ have the same 0-simplices at each $\d \in \R$, we know that $f_0^{\d,\d'}\equiv g_0^{\d,\d'}$. 

Let $\g\in \ker(\p_0^\d)=\ker(\p_0^{\d,S})$, and let $\g+\im(\p_1^\d) \in H_0(\sink_\d)$. Then observe that
\begin{align*}
(f_0^{\d,\d'})_\#(\g + \im(\p_1^\d)) &=f_0^{\d,\d'}(\g) + \im(\p_1^{\d'}) &&\text{($f_0^{\d,\d'}$ is a chain map)} \\
&=g_0^{\d,\d'}(\g) + \im(\p_1^{\d'}) &&\text{($f_0^{\d,\d'}\equiv g_0^{\d,\d'}$)}\\
&=g_0^{\d,\d'}(\g) + \im(\p_1^{\d',S}) &&\text{($\im(\p_1^{\d'})=\im(\p_1^{\d',S})$)}\\
&=(g_0^{\d,\d'})_\#(\g + \im(\p_1^{\d,S})).&&\text{($g_0^{\d,\d'}$ is a chain map)}
\end{align*}

Thus $(f_0^{\d,\d'})_\#=(g_0^{\d,\d'})_\#$ for each $\d'\geq \d \in \R$. Since we also have $H_0(\sink_\d)=H_0(\sink_{\d,S})$ for each $\d \in \R$, we can then apply the Persistence Equivalence Theorem (Theorem \ref{thm:pet}) to conclude the proof. \end{proof}

\section{Higher dimensional Dowker persistence diagrams of cycle networks}
\label{sec:cycle-addendum}
The contents of this section rely on results in \cite{adamaszek2015vietoris} and \cite{adamaszek2016nerve}. We introduce some minimalistic versions of definitions from the referenced papers to use in this section. The reader should refer to these papers for the original definitions. 

Given a metric space $(M,d_M)$ and $m\in M$, we will write $\overline{B(m,\e)}$ to denote a closed $\e$-ball centered at $m$, for any $\e > 0$. For a subset $X\subseteq M$ and some $\e>0$, the \emph{\v{C}ech complex} of $X$ at resolution $\e$ is defined to be the following simplicial complex:
\[\cech(X,\e):=\set{\s \subseteq X : \cap_{x\in \s}\overline{B(x,\e)} \neq \emptyset}.\]

In the setting of metric spaces, the \v{C}ech complex coincides with the Dowker source and sink complexes. We will be interested in the special case where the underlying metric space is the circle. We write $S^1$ to denote the circle with unit circumference. Next, for any $n\in \N$, we write $\mb{X}_n:=\set{0,\tfrac{1}{n},\tfrac{2}{n},\ldots, \tfrac{n-1}{n}}$ to denote the collection of $n$ equally spaced points on $S^1$ with the restriction of the arc length metric on $S^1$. Also let $G_n$ denote the $n$-node cycle network with vertex set $\mb{X}_n$ (in contrast with $\mb{X}_n$, here $G_n$ is equipped with the asymmetric weights defined in \S\ref{sec:cycle}). The connection between $\mb{X}_n$ and Dowker complexes of the cycle networks $G_n$ is highlighted by the following observation:

\begin{proposition}\label{prop:dowker-cech-cplx} Let $n\in \N$. Then for any $\d \in [0,1]$, we have $\cech(\mb{X}_n,\frac{\d}{2}) = \mf{D}^{\si}_{n\d,G_n}.$
\end{proposition}

The scaling factor arises because $G_n$ has diameter $\sim n$, whereas $\mb{X}_n\subseteq S^1$ has diameter $\sim 1/2$. This proposition provides a pedagogical step which helps us transport results from the setting of \cite{adamaszek2015vietoris} and \cite{adamaszek2016nerve} to that of the current paper.

\begin{proof}
For $\d =0$, both the \v{C}ech and Dowker complexes consist of the $n$ vertices, and are equal. Similarly for $\d=1$, both $\cech(\mb{X}_n,1)$ and $\mf{D}^{\si}_{n,G_n}$ are equal to the $(n-1)$-simplex. 

Now suppose $\d \in (0,1)$. Let $\s \in \mf{D}^{\si}_{n\d,G_n}$. Then $\s$ is of the form $[\tfrac{k}{n},\tfrac{k+1}{n},\ldots, \tfrac{\floor{k+n\d}}{n}]$ for some integer $0\leq k \leq n-1$, where the $n\d$-sink is $\tfrac{\floor{k+n\d}}{n}$ and all the numerators are taken modulo $n$. We claim that $\s \in \cech(\mb{X}_n,\tfrac{\d}{2})$. 
To see this, observe that $d_{S^1}(\tfrac{k}{n},\tfrac{\floor{k+n\d}}{n}) \leq \d$, and so $\overline{B(\tfrac{k}{n},\tfrac{\d}{2})} \cap \overline{B(\tfrac{\floor{k+n\d}}{n},\tfrac{\d}{2})} \neq \emptyset$. 
Then we have $\s \in \bigcap_{i=0}^{n\d}\overline{B\left(\tfrac{\floor{k+i}}{n},\tfrac{\d}{2}\right)}$, and so $\s \in \cech(\mb{X}_n,\tfrac{\d}{2})$.

Now let $\s \in \cech(\mb{X}_n,\tfrac{\d}{2})$. Then $\s$ is of the form $[\tfrac{k}{n},\tfrac{k+1}{n},\ldots, \tfrac{k+j}{n}]$ for some integer $0\leq k\leq n-1$, where $j$ is an integer such that $\tfrac{j}{n} \leq \d$. In this case, we have $\s = \mb{X}_n \cap_{i=0}^j\overline{B\left(\tfrac{k+i}{n},\d\right)}$. 
Then in $G_n$, after applying the scaling factor $n$, we have $\s \in \mf{D}^{\si}_{n\d,G_n}$, with $\tfrac{k+j}{n}$ as an $n\d$-sink in $G_n$. This shows equality of the two simplicial complexes.\end{proof}

\begin{theorem}[Theorem 3.5, \cite{adamaszek2016nerve}] 
\label{thm:cech-S1}
Fix $n\in \N$, and let $0\leq k \leq n-2$ be an integer. Then,
\[\cech(\mb{X}_n,\tfrac{k}{2n})\simeq \begin{cases}
\bigvee^{n-k-1}S^{2l} & \text{if } \tfrac{k}{n} = \tfrac{l}{l+1},\\
S^{2l+1} &\text{or if } \tfrac{l}{l+1} < \tfrac{k}{n} < \tfrac{l+1}{l+2},
\end{cases}\]
for some $l \in \Z_+$. Here $\bigvee$ denotes the wedge sum, and $\simeq$ denotes homotopy equivalence.

\end{theorem}

\thmdowkercyceven*

\begin{proof}[Proof of Theorem \ref{thm:dowker-cyc-even}] Let $l \in \N$ be such that $(l+1)$ divides $n$ and $0\leq k\leq n-2$. 
Then $\mf{D}^{\si}_{k,G_n} = \cech(\mb{X}_n,\tfrac{k}{2n})$ has the homotopy type of a wedge sum of $(n-k-1)$ copies of $S^{2l}$, by Theorem \ref{thm:cech-S1}. Here the equality follows from Proposition \ref{prop:dowker-cech-cplx}. Notice that $n-k-1 = \tfrac{n}{l+1}-1$. 
Furthermore, by another application of Theorem \ref{thm:cech-S1}, it is always possible to choose $\e >0$ small enough so that $\mf{D}^{\si}_{k-\e,G_n} = \cech(\mb{X}_n,\tfrac{k-\e}{2n})$ and $\mf{D}^{\si}_{k+\e,G_n} = \cech(\mb{X}_n,\tfrac{k+\e}{2n})$ have the homotopy types of odd-dimensional spheres. Thus the inclusions $\mf{D}^{\si}_{k-\e,G_n} \subseteq \mf{D}^{\si}_{k,G_n} \subseteq \mf{D}^{\si}_{k+\e,G_n}$ induce zero maps upon passing to homology. It follows that $\dgm^{\mf{D}}_{2l}(G_n)$ consists of the point $(\tfrac{nl}{l+1},\tfrac{nl}{l+1} + 1)$ with multiplicity $\tfrac{n}{l+1} -1$.

If $l \in \N$ does not satisfy the condition described above, then there does not exist an integer $1\leq j \leq n-2$ such that $j/n = l/(l+1)$. So for each $1\leq j \leq n-2$, $\mf{D}^{\si}_{j,G_n} = \cech(\mb{X}_n,\tfrac{j}{2n})$ has the homotopy type of an odd-dimensional sphere by Theorem \ref{thm:cech-S1}, and thus does not contribute to $\dgm^{\mf{D}}_{2l}(G_n)$. If $l$ satisfies the condition but $k \geq n-1$, then $\cech(\mb{X}_n,\tfrac{k}{2n})$ is just the $(n-1)$-simplex, hence contractible. \end{proof}


Theorem \ref{thm:dowker-cyc-even} gives a characterization of the even dimensional Dowker persistence diagrams of cycle networks. The most interesting case occurs when considering the 2-dimensional diagrams: we see that cycle networks of an even number of nodes have an interesting barcode, even if the bars are all short-lived. For dimensions 4, 6, 8, and beyond, there are fewer and fewer cycle networks with nontrivial barcodes (in the sense that only cycle networks with number of nodes equal to a multiple of 4, 6, 8, and so on have nontrivial barcodes). For a complete picture, it is necessary to look at odd-dimensional persistence diagrams. This is made possible by the next set of constructions. 

We have already recalled the definition of a Rips complex of a metric space. To facilitate the assessment of the connection to \cite{adamaszek2015vietoris}, we temporarily adopt the notation $\rips(X,\e)$ to denote the Vietoris-Rips complex of a metric space $(X,d_X)$ at resolution $\e >0$, i.e. the simplicial complex $\set{\s \subseteq X : \diam(\s) \leq \e}$. 

\begin{theorem}[Theorem 9.3, Proposition 9.5, \cite{adamaszek2015vietoris}] 
\label{thm:vr-cech-cd}
Let $0< r < \tfrac{1}{2}$. Then there exists a map $T_r: \pow(S^1) \r \pow(S^1)$ and a map $\pi_r: S^1 \r S^1$ such that there is an induced homotopy equivalence
\[\rips(T_r(X), \tfrac{2r}{1+2r}) \xr{\simeq} \cech(X,r).\]
Next suppose $X\subseteq S^1$ and let $0< r \leq r' < \tfrac{1}{2}$. Then there exists a map $\eta: S^1 \r S^1$ such that the following diagram commutes:
\[ \begin{tikzcd}[column sep=large]
\rips(T_r(X), \tfrac{2r}{1+2r}) \arrow{r}{\eta} \arrow{d}{\simeq}[swap]{\pi_r} &
\rips(T_{r'}(X), \tfrac{2r'}{1+2r'}) \arrow{d}{\simeq}[swap]{\pi_{r'}}\\
\cech(X,r) \arrow[hookrightarrow]{r}{\subseteq} &
\cech(X,r')
\end{tikzcd} \]

\end{theorem}

\begin{theorem}\label{thm:cyc-cech-main} Consider the setup of Theorem \ref{thm:vr-cech-cd}. If $\cech(X,r)$ and $\cech(X,r')$ are homotopy equivalent, then the inclusion map between them is a homotopy equivalence.
\end{theorem}

Before providing the proof, we show how it implies Theorem \ref{thm:dowker-cyc-odd}.

\thmdowkercycodd*

\begin{proof}[Proof of Theorem \ref{thm:dowker-cyc-odd}] 
By Proposition \ref{prop:dowker-cech-cplx} and Theorem \ref{thm:cech-S1}, we know that $\sink_{k,G_n} = \cech(\mb{X}_n,\tfrac{k}{2n}) \simeq S^1 $ for integers $0 < k < \tfrac{n}{2}$. Let $b \in \N$ be the greatest integer less than $n/2$. Then by Theorem \ref{thm:cyc-cech-main}, we know that each inclusion map in the following chain is a homotopy equivalence:
\[\sink_{1,G_n} \subseteq \ldots \subseteq \sink_{b,G_n} = \sink_{\ceil{n/2}^-,G_n}.\]
It follows that $\dgm^{\mf{D}}_1(G_n) = \set{\left(1,\ceil{\tfrac{n}{2}}\right)}$. The notation in the last equality means that $\sink_{b,G_n} = \sink_{\d,G_n}$ for all $\d \in [b,b+1)$, where $b+1 = \ceil{n/2}$.

In the more general case, let $l \in \N$ and let $M_l$ be as in the statement of the result. Suppose first that $M_l$ is empty. Then by Proposition \ref{prop:dowker-cech-cplx} and Theorem \ref{thm:cech-S1}, we know that $\sink_{k,G_n}$ has the homotopy type of a wedge of even-dimensional spheres or an odd-dimensional sphere of dimension strictly different from $(2l+1)$, for any choice of integer $k$. Thus $\dgm^{\mf{D}}_{2l+1}(G_n)$ is trivial.

Next suppose $M_l$ is nonempty. By another application of Proposition \ref{prop:dowker-cech-cplx} and Theorem \ref{thm:cech-S1}, we know that $\sink_{k,G_n} = \cech(\mb{X}_n,\tfrac{k}{2n}) \simeq S^{2l+1} $ for integers $\tfrac{nl}{l+1} < k < \tfrac{n(l+1)}{l+2}$. 
Write $a_l:=\min\set{m\in M_l}$ and $b_l:=\max\set{m\in M_l}$. Then by Theorem \ref{thm:cyc-cech-main}, we know that each inclusion map in the following chain is a homotopy equivalence:
\[\sink_{a_l,G_n} \subseteq \ldots \subseteq \sink_{b_l,G_n} = \sink_{\ceil{n(l+1)/(l+2)}^-,G_n}.\]
It follows that $\dgm^{\mf{D}}_{2l+1}(G_n) = \set{\left(a_l,\ceil{\tfrac{n(l+1)}{l+2}}\right)}$.
\end{proof}

It remains to provide a proof of Theorem \ref{thm:cyc-cech-main}. For this, we need some additional machinery.

\paragraph{Cyclic maps and winding fractions}
We introduce some more terms from \cite{adamaszek2015vietoris}, but for efficiency, we try to minimize the scope of the definitions to only what is needed for our purpose. Recall that we write $S^1$ to denote the circle with unit circumference. So any $x\in S^1$ can be naturally identified with a point in $[0,1)$. We fix a choice of $0\in S^1$, and for any $x,x' \in S^1$, the length of a clockwise arc from $x$ to $x'$ is denoted by $\overrightarrow{d_{S^1}}(x,x')$. Then, for any finite subset $X\subseteq S^1$ and any $r \in (0,1/2)$, the \emph{directed Vietoris-Rips graph} $\vrg(X,r)$ is defined to be the graph with vertex set $X$ and edge set $\{(x,x') : 0 < \overrightarrow{d_{S^1}}(x,x') < r\}$. Next, let $\overrightarrow{G}$ be a Vietoris-Rips graph such that the vertices are enumerated as $x_0,x_1,\ldots, x_{n-1}$, according to the \emph{clockwise} order in which they appear. A \emph{cyclic map} between $\overrightarrow{G}$ and a Vietoris-Rips graph $\overrightarrow{H}$ is a map of vertices $f$ such that for each edge $(x,x') \in \overrightarrow{G}$, we have either $f(x)=f(x')$, or $(f(x),f(x')) \in \overrightarrow{H}$, and $\sum_{i=0}^{n-1}\overrightarrow{d_{S^1}}(f(x_i),f(x_{i+1})) = 1$. Here $x_n:=x_0$.

Next, the \emph{winding fraction} of a Vietoris-Rips graph $\overrightarrow{G}$ with vertex set $V(\overrightarrow{G})$ is defined to be the infimum of numbers $\tfrac{k}{n}$ such that there is an order-preserving map $V(\overrightarrow{G}) \r \Z/n\Z$ such that each edge is mapped to a pair of numbers at most $k$ apart. A key property of the winding fraction, denoted $\wf$, is that if there is a cyclic map between Vietoris-Rips graphs $\overrightarrow{G} \r \overrightarrow{H}$, then $\wf(\overrightarrow{G}) \leq \wf(\overrightarrow{H})$.  

\begin{theorem}[Corollary 4.5, Proposition 4.9, \cite{adamaszek2015vietoris}]
\label{thm:vrcirc-main}
Let $X\subseteq S^1$ be a finite set and let $0 < r < \tfrac{1}{2}$. Then,
\[\rips(X,r) \simeq \begin{cases}
S^{2l+1} &: \tfrac{l}{2l+1} < \wf(\vrg(X,r)) < \tfrac{l+1}{2l+3} \text{ for some } l\in \Z_+,\\
\bigvee^j S^{2l} &: \wf(\vrg(X,r)) = \tfrac{l}{2l+1}, \text{ for some } j\in \N.
\end{cases}\]
Next let $X' \subseteq S^1$ be another finite set, and let $r \leq r' < \tfrac{1}{2}$. Suppose $f:\vrg(X,r) \r \vrg(X',r')$ is a cyclic map between Vietoris-Rips graphs and $\tfrac{l}{2l+1} < \wf(\vrg(X,r)) \leq \wf(\vrg(X',r')) < \tfrac{l+1}{2l+3}$. Then $f$ induces a homotopy equivalence between $\rips(X,r)$ and $\rips(X',r')$.
\end{theorem}

We now have the ingredients for a proof of Theorem \ref{thm:cyc-cech-main}. 

\begin{proof}[Proof of Theorem \ref{thm:cyc-cech-main}] Since the maps $\pi_r$ and $\pi_{r'}$ induce homotopy equivalences, it follows that 
\[\rips(T_r(X),\tfrac{2r}{1+2r}) \simeq \rips(T_{r'}(X),\tfrac{2r'}{1+2r'}).\] 
By the characterization result in Theorem \ref{thm:vrcirc-main}, we know that there exists $l \in \Z_+$ such that 
\[\tfrac{l}{2l+1}  < \wf(\vrg(T_r(X),\tfrac{2r}{1+2r})) \leq \wf(\vrg(T_{r'}(X),\tfrac{2r'}{1+2r'})) < \tfrac{l+1}{2l+3}.\]
The map $\eta$ in Theorem \ref{thm:vr-cech-cd} appears in \cite[Proposition 9.5]{adamaszek2015vietoris} through an explicit construction. Moreover, it is shown that $\eta$ induces a cyclic map $\wf(\vrg(T_r(X),\tfrac{2r}{1+2r})) \r \wf(\vrg(T_{r'}(X),\tfrac{2r'}{1+2r'}))$. Thus by Theorem \ref{thm:vrcirc-main}, $\eta$ induces a homotopy equivalence between $\rips(T_r(X),\tfrac{2r}{1+2r})$ and $\rips(T_{r'}(X),\tfrac{2r'}{1+2r'})$. Finally, the commutativity of the diagram in Theorem \ref{thm:vr-cech-cd} shows that the inclusion $\cech(X,r) \subseteq \cech(X,r')$ induces a homotopy equivalence. \end{proof}

\begin{remark} The analogue of Theorem \ref{thm:cyc-cech-main} for \v{C}ech complexes appears as Proposition 4.9 of \cite{adamaszek2015vietoris} for Vietoris--Rips complexes. We prove Theorem \ref{thm:cyc-cech-main} by connecting \v{C}ech and Vietoris-Rips complexes using Proposition 9.5 of \cite{adamaszek2015vietoris}. However, as remarked in \S9 of \cite{adamaszek2015vietoris}, one could prove Theorem \ref{thm:cyc-cech-main} directly using a parallel theory of winding fractions for \v{C}ech complexes.

\end{remark}

\end{document}